\documentclass[12pt]{amsart}
\usepackage{amsmath,amsopn,amsfonts,amsthm,amssymb,xcolor}
\usepackage[colorlinks=true, pdfstartview=FitV, linkcolor=blue, citecolor=blue, urlcolor=blue]{hyperref}
\usepackage{mathtools}
\usepackage{setspace}
\usepackage{fullpage}

\usepackage{tikz}
\usepackage{tikz-cd}
\usetikzlibrary{arrows,calc} 

\tikzstyle{dot} = [inner sep=0.5mm,circle,draw,minimum size=1mm]
\tikzstyle{continued}=[dashed]

\newtheorem{thm}{Theorem}[section]
\newtheorem*{thm*}{Theorem}
\newtheorem{lemma}[thm]{Lemma}
\newtheorem{conj}[thm]{Conjecture}
\newtheorem{coro}[thm]{Corollary}
\newtheorem{prop}[thm]{Proposition}

\newtheorem{thmIntro}{Theorem}  

\newcommand{\thistheoremname}{}
\newtheorem*{genericthm}{\thistheoremname}
\newenvironment{namedthm}[1]
  {\renewcommand{\thistheoremname}{#1}%
   \begin{genericthm}}
  {\end{genericthm}}

\theoremstyle{definition}
\newtheorem{defn}[thm]{Definition}
\newtheorem{idefn}[thm]{Informal Definition}

\theoremstyle{remark}

\usepackage{manfnt}
\newenvironment{ex}{\refstepcounter{thm}\begin{proof}[Example \emph{\thethm}]}{\end{proof}}
\newenvironment{rem}{\refstepcounter{thm}\begin{proof}[Remark \emph{\thethm}]}{\end{proof}}
\newenvironment{conv}{\refstepcounter{thm}\begin{proof}[Convention \emph{\thethm}]}{\end{proof}}
\newenvironment{warn}{\refstepcounter{thm}\begin{proof}[Warning \emph{\thethm}]}{\end{proof}}

\numberwithin{equation}{section}

\DeclareMathOperator{\rank}{rank}
\DeclareMathOperator{\ord}{ord}
\DeclareMathOperator {\sA} {\mathcal{A}} 
\newcommand{\RR}{\mathbb{R}}
\newcommand{\RRpos}{\mathbb{R}_{>0}}

\newcommand{\SSS}{\sA(\RR_{\geq 1})} 
\newcommand{\yy}{\mathbf{y}}
\newcommand{\xx}{\mathbf{x}} 
\newcommand{\cluster}{\xx} 
\newcommand{\subcluster}{\xx} 
\newcommand{\ClusterVariable}[1]{x_{#1}} 
\newcommand{\Subface}[1]{\SSS_{#1}} 
\newcommand{\arxiv}[1]{\href{http://arxiv.org/abs/#1}{\texttt{arXiv:#1}}}

\DeclareRobustCommand{\SkipTocEntry}[5]{}

\def\opa{.5}

\newcommand{\MarkedPoint}[2]{
    ({(#1^2-#2^2)/(#1^2+#2^2)},{(2*#1*#2)/(#1^2+#2^2)})
}
\newcommand{\Horocycle}[2]{
    ({(#1^2-#2^2)/(#1^2+#2^2+1)},{(2*#1*#2)/(#1^2+#2^2+1)}) circle ({1/(#1^2+#2^2+1)})
}
\newcommand{\HorocycleCenter}[2]{
    ({(#1^2-#2^2)/(#1^2+#2^2+1)},{(2*#1*#2)/(#1^2+#2^2+1)})
}
\newcommand{\Geodesic}[4]{
    ({(#1*#3 - #2*#4)/(#1*#3 + #2*#4)},{(#1*#4 + #2*#3)/(#1*#3 + #2*#4)}) circle ({(#1*#4-#2*#3)/(#1*#3+#2*#4)})
}

\title[Superunitary regions of cluster algebras]{Superunitary regions of cluster algebras}
\author{Emily Gunawan}
\author{Greg Muller}
 
\begin{document}

\begin{abstract}
This note introduces the \emph{superunitary region} of a cluster algebra, the subspace of the totally positive region on which each cluster variable is at least 1. Our main result is that the superunitary region of a finite type cluster algebra is a regular CW complex which is homeomorphic to the generalized associahedron of the cluster algebra. As an application, the compactness of the superunitary region implies that each Dynkin diagram admits finitely many positive integral friezes.
\end{abstract}

\makeatletter
\@namedef{subjclassname@2020}{%
  \textup{2020} Mathematics Subject Classification}
\makeatother


\maketitle

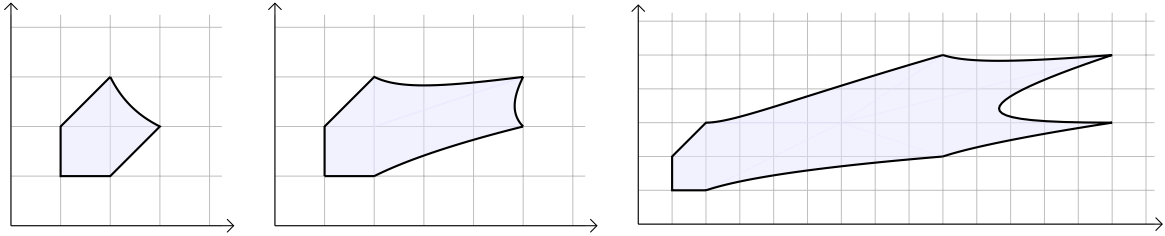
\begin{figure}[h!t]
    \begin{tikzpicture}[scale=.66,baseline=(current bounding box.center)]
        \draw[step=1,black!25,very thin] (0,0) grid (4.25,4.25);
        \draw[-angle 90] (0,0) to (4.5,0);
        \draw[-angle 90] (0,0) to (0,4.5);
        
        \path[fill=blue!10,opacity=\opa, domain=1:2, variable=\t] plot ({\t+1}, {2*\t^(-1)+1}) to (2,1) to (1,1) to (1,2) to (2,3);
        \draw[thick,domain=1:2, variable=\t] plot ({\t+1}, {2*\t^(-1)+1}) to (2,1) to (1,1) to (1,2) to (2,3);
    \end{tikzpicture}
\hspace{.25cm}
    \begin{tikzpicture}[x={(1,1)},y={(1,-1)},x={(1,-1)},scale=.66,baseline=(current bounding box.center)]
     \draw[step=1,black!25,very thin] (0,0) grid (4.25,6.25);
		\draw[-angle 90] (0,0) to (4.5,0);
  		\draw[-angle 90] (0,0) to (0,6.5);
		
		\draw[blue!10,fill=blue!10,opacity=\opa,variable=\t,domain=1:2] plot (\t,{\t^2+1}) plot ({\t+1},{\t+2+2/\t}) to (2,2) to (1,2);
 		\draw[blue!10,fill=blue!10,opacity=\opa,variable=\t,domain=1:2] plot ({\t+2*\t^(-1)},{1+4*\t^(-2)}) to (2,2) to (3,5);
 		\draw[blue!10,fill=blue!10,opacity=\opa,variable=\t,domain=1:2] (1,1) to (1,2) to (3,2) to (2,1) to (1,1);
     \draw[thick,variable=\t,domain=1:2] (1,1) to (1,2) plot (\t,\t^2+1) plot (\t+1,{\t+2+2*\t^(-1)}) plot ({\t+2*\t^(-1)},{1+4*\t^(-2)}) to (2,1) to (1,1);
\end{tikzpicture}
\hspace{.25cm}
\begin{tikzpicture}[x={(1,1)},y={(1,-1)},x={(1,-1)},scale=.45,baseline=(current bounding box.center)]
    \draw[step=1,black!25,very thin] (0,0) grid (6.25,15.25);
 		\draw[-angle 90] (0,0) to (6.5,0);
		\draw[-angle 90] (0,0) to (0,15.5);
		
		\draw[blue!10,fill=blue!10,opacity=\opa,variable=\t,domain=1:2] (3,6) plot (\t,\t^3+1) to (3,6);
		\draw[blue!10,fill=blue!10,opacity=\opa,variable=\t,domain=1:2] (3,6) plot ({\t+1},{\t^2+3*\t+3+2*\t^(-1)}) to (3,6);
		\draw[blue!10,fill=blue!10,opacity=\opa,variable=\t,domain=1:2] (3,6) plot ({\t^2+2*\t^(-1)},{\t^3+5+8*\t^(-3)}) to (3,6);
		\draw[blue!10,fill=blue!10,opacity=\opa,variable=\t,domain=1:2] (3,6) plot ({\t+2+2*\t^(-1)},{\t+3+6*\t^(-1)+4*\t^(-2)}) to (3,6);
		\draw[blue!10,fill=blue!10,opacity=\opa,variable=\t,domain=1:2] (3,6) plot ({\t+4*\t^(-2)},{1+8*\t^(-3)}) to (3,6);
		\draw[blue!10,fill=blue!10,opacity=\opa,variable=\t,domain=1:2] (3,6) to (3,2) to (2,1) to (1,1) to (1,2) to (3,6);
		\draw[thick,variable=\t,domain=1:2] plot (\t,\t^3+1) plot ({\t+1},{\t^2+3*\t+3+2*\t^(-1)}) plot ({\t^2+2*\t^(-1)},{\t^3+5+8*\t^(-3)}) plot ({\t+2+2*\t^(-1)},{\t+3+6*\t^(-1)+4*\t^(-2)}) plot ({\t+4*\t^(-2)},{1+8*\t^(-3)}) to (2,1) to (1,1) to (1,2);
  \end{tikzpicture} 
\caption{The superunitary regions of types $A_2$, $B_2/C_2$, and $G_2$ (embedded in $\mathbb{R}_{>0}^2$)}
\end{figure}

\begin{center}
\emph{A table of contents is provided at the end of the document.}
\end{center}


\section{Summary of results}

\subsection{Cluster algebras}

Cluster algebras are commutative algebras with distinguished elements, called \emph{clusters variables}, and distinguished sets of cluster variables, called \emph{clusters}. While the set of cluster variables may be infinite, cluster algebras with finitely many cluster variables (\emph{finite type} cluster algebras) have been completely classified, and are in bijection with Dynkin diagrams.\footnote{Here and throughout, `Dynkin diagrams' are finite type (e.g.~not affine) but not necessarily connected.}


Clusters can be \emph{mutated}, a process which constructs another cluster in the same cluster algebra by replacing a single element in the set with a new cluster variable. Any two clusters may be connected by repeating this process, and so the entire set of cluster variables can be recovered from a single \emph{initial cluster}.
These relations can be visualized with the \emph{exchange graph}, which has a vertex for each cluster and an edge for each mutation between clusters.

Since two clusters are related by a mutation iff they have all but one element in common, the exchange graph can be equivalently defined as the graph whose vertices are clusters and whose edges are subsets of clusters containing all but one element. In \cite{CFZ02}, exchange graphs of finite type cluster algebras were shown to be the 1-skeleta of polytopes (called \emph{generalized associahedra}) whose faces are indexed by \emph{subclusters} of $\mathcal{A}$; that is, subsets of clusters. These polytopes may be related to other incarnations of Dynkin diagrams; for example, roots of the associated Lie algebra and representations of an associated quiver.

\begin{ex}\label{ex:A2 cluster algebra}
One of the simplest cluster algebras is the \emph{$A_2$ cluster algebra}, given below.
\[\mathbb{Z}[x_1,x_2,x_3,x_4,x_5]/\langle x_{i-1}x_{i+1}-(x_i+1), i\in \{1,2,...,5\}\rangle\]
Here, the cluster variables are the five generators $\{x_1,x_2,x_3,x_4,x_5\}$ (with indices considered mod $5$), and the clusters are the five pairs of cluster variables with adjacent indices:
\[
\{x_1,x_2\},\{x_2,x_3\},\{x_3,x_4\},\{x_4,x_5\},\{x_5,x_1\}\]
Each cluster $\{x_i,x_{i+1}\}$ has two mutations, $\{x_{i-1},x_i\}$ and $\{x_{i+1},x_{i+2}\}$, and so the exchange graph is a 5-cycle. The generalized associahedron contains the exchange graph and a 2-dimensional face corresponding to $\varnothing$ which fills in the interior of a pentagon (Figure \ref{fig: a2ass}).
\end{ex}

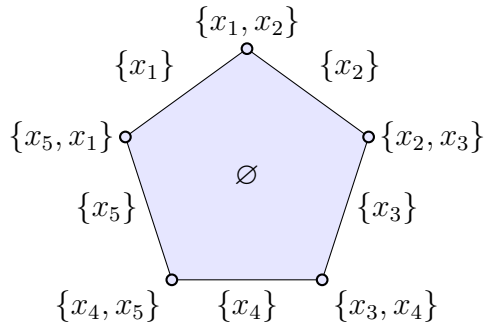
\begin{figure}[h!t]
\begin{tikzpicture}[scale=.85]
    \coordinate (1) at (90:2);
    \coordinate (2) at (90-1*72:2);
    \coordinate (3) at (90-2*72:2);
    \coordinate (4) at (90-3*72:2);
    \coordinate (5) at (90-4*72:2);

    \path[fill=blue!10] (1) to (2) to (3) to (4) to (5) to (1);
    \draw[] (1) to (2) to (3) to (4) to (5) to (1);
    \node[dot,thick,fill=blue!10] (1d) at (1) {};
    \node[dot,thick,fill=blue!10] (2d) at (2) {};
    \node[dot,thick,fill=blue!10] (3d) at (3) {};
    \node[dot,thick,fill=blue!10] (4d) at (4) {};
    \node[dot,thick,fill=blue!10] (5d) at (5) {};

    \node[above] at (90:2) {$\{x_1,x_2\}$};
    \node[right] at (90-1*72:2) {$\{x_2,x_3\}$};
    \node[below right] at (90-2*72:2) {$\{x_3,x_4\}$};
    \node[below left] at (90-3*72:2) {$\{x_4,x_5\}$};
    \node[left] at (90-4*72:2) {$\{x_5,x_1\}$};
    \node[above left] at (90-9*36:1.618) {$\{x_1\}$};
    \node[above right] at (90-1*36:1.618) {$\{x_2\}$};
    \node[right] at (90-3*36:1.618) {$\{x_3\}$};
    \node[below] at (90-5*36:1.618) {$\{x_4\}$};
    \node[left] at (90-7*36:1.618) {$\{x_5\}$};
    \node[] at (0,0) {$\varnothing$};
\end{tikzpicture}
\caption{The generalized associahedron of the $A_2$ cluster algebra}
\label{fig: a2ass}
\end{figure}

\subsection{Totally positive regions}

Cluster algebras often arise as rings of functions on notable spaces (such as Lie groups, Grassmannians, and Teichm\"uller spaces), with the cluster variables corresponding to some distinguished class of functions (like generalized minors, Pl\"ucker coordinates, or Penner coordinates). 
These distinguished functions can be used to characterize particularly well-behaved regions in these spaces. For example, a point in a such a space is \emph{totally positive} if every cluster variable has a positive real value at that point.
In special cases, this can be used to define \emph{totally positive matrices}, \emph{totally positive Grassmannians}, and \emph{decorated Teichm\"uller spaces}.

To define the totally positive points of a general cluster algebra $\sA$, we first identify the elements of $\sA$ with continuous, real-valued functions on the \emph{space of real points} of $\sA$: 
\[ \mathcal{A}(\mathbb{R}) := \{ \text{ring homomorphisms $p:\mathcal{A}\rightarrow \mathbb{R}$}\}\]
Each element $a\in \mathcal{A}$ gives a function $f_a:\mathcal{A}(\mathbb{R})\rightarrow \mathbb{R}$ by the rule $f_a(p):=p(a)$, and we endow $\mathcal{A}(\mathbb{R})$ with the coarsest topology for which each 
$f_a:\mathcal{A}(\mathbb{R})\rightarrow \mathbb{R}$ is continuous.

A \textbf{totally positive} point of $\mathcal{A}$ is a point $p\in \mathcal{A}(\mathbb{R})$ on which the function $f_x$ is positive for each cluster variable $x$; or equivalently, $p$ is a ring homomorphism $\mathcal{A}\rightarrow\mathbb{R}$ which sends each cluster variable to a positive real number. The set of totally positive points forms the \textbf{totally positive region} 
\[\mathcal{A}(\mathbb{R}_{>0})\subset \mathcal{A}(\mathbb{R})\]
Totally positive regions of cluster algebras have been studied extensively, in general and in special cases. A fundamental result is that, for any cluster $\mathbf{x}=(x_1,x_2,...,x_r)$, the $r$-tuple of functions $f_\mathbf{x}:=(f_{x_1},f_{x_2},...,f_{x_r})$ on $\mathcal{A}(\mathbb{R})$ restricts to a homeomorphism
\[ f_\mathbf{x}:\mathcal{A}(\mathbb{R}_{>0})\xrightarrow{\sim} \mathbb{R}_{>0}^r\]
As a consequence, each choice of cluster gives a parametrization of the totally positive region by an orthant in $r$-dimensional space. In special cases, these give well-known parametrizations of totally positive matrices \cite{Lus94,BFZ96} and decorated Teichm\"uller spaces \cite{Pen87}.

\begin{conv}
Here, $\mathbb{R}_{>0}$ denotes the set $\{c\in \mathbb{R} \mid c>0\}$. We extend this convention to other inequalities and other totally ordered sets; so for example, $\mathbb{Z}_{\geq1}:=\{c\in \mathbb{Z} \mid c\geq1\}$.
\end{conv}

\subsection{Superunitary regions}

In this paper, we consider a variation of the previous construction. 
A point $p\in \mathcal{A}(\mathbb{R})$ is \textbf{superunitary}\footnote{Literally, `greater than $1$'.} if $f_x(p)\geq1$ for every cluster variable $x$; or equivalently, $p$ is a ring homomorphism $\mathcal{A}\rightarrow \mathbb{R}$ which sends each cluster variable into $\mathbb{R}_{\geq1}$.

The set of superunitary points forms the \textbf{superunitary region} $\mathcal{A}(\mathbb{R}_{\geq1})\subset \mathcal{A}(\mathbb{R})$. Since every superunitary point is also totally positive, the superunitary region is a subspace of the totally positive region $\mathcal{A}(\mathbb{R}_{>0})$.
For any cluster $\mathbf{x}$, the homeomorphism $f_\mathbf{x}:\mathcal{A}(\mathbb{R}_{>0})\xrightarrow{\sim} \mathbb{R}_{>0}^r$ embeds the superunitary region into the positive orthant in $r$-dimensional space:
\[ f_\mathbf{x}: \mathcal{A}(\mathbb{R}_{\geq1}) \hookrightarrow \mathbb{R}_{>0}^r\]
The image has a  nice characterization. Every element $a \in \mathcal{A}$ can be expressed as a Laurent polynomial in the cluster $\mathbf{x}$, which we denote $\ell_{a,\mathbf{x}}$ (this fact is called the \emph{Laurent phenomenon}). 
The embedding $f_\mathbf{x}$ then identifies the superunitary region $\mathcal{A}(\mathbb{R}_{\geq1})$ with the subset of $\mathbb{R}_{>0}^r$ on which each Laurent polynomial $\ell_{a,\mathbf{x}}$ has value $\geq1$, for all cluster variables $a\in \mathcal{A}$.

\begin{figure}[h!t]
    \begin{tikzpicture}[xshift=.75in,scale=1,baseline=(current bounding box.center)]
        \draw[step=1,black!25,very thin] (-.25,-.25) grid (4.25,4.25);
        \draw[-angle 90] (-.25,0) to (4.5,0) node[above] {$x_1$};
        \draw[-angle 90] (0,-.25) to (0,4.5) node[right] {$x_2$};
        
        \path[fill=blue!10,opacity=\opa, domain=1:2, variable=\t] plot ({\t+1}, {2*\t^(-1)+1}) to (2,1) to (1,1) to (1,2) to (2,3);
        \draw[domain=1:2, variable=\t] plot ({\t+1}, {2*\t^(-1)+1}) to (2,1) to (1,1) to (1,2) to (2,3);
        		
        \draw[continued] (-.25,.75) to (1,2);
        \draw[continued] (2,3) to (3.25,4.25);
        \draw[continued] (1,-.25) to (1,1);
        \draw[continued] (1,2) to (1,4.25);
        \draw[continued] (-.25,1) to (1,1);
        \draw[continued] (2,1) to (4.25,1);
        \draw[continued] (.75,-.25) to (2,1);
        \draw[continued] (3,2) to (4.25,3.25);
        \draw[continued,variable=\t,domain=.62:1] plot (\t+1,{2*\t^(-1)+1});
        \draw[continued,variable=\t,domain=2:3.25] plot (\t+1,{2*\t^(-1)+1});
    \end{tikzpicture}
    \caption{The superunitary region of the $A_2$ cluster algebra (embedded in $\mathbb{R}^2_{>0}$)}
    \label{fig:A2_superunitary_region}
\end{figure}
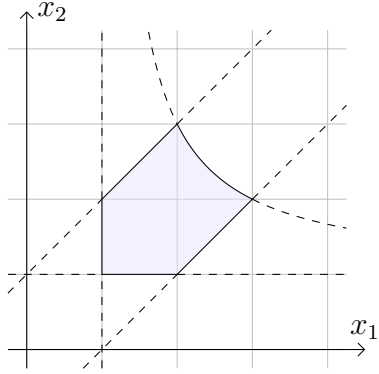

\begin{ex}
The five cluster variables $x_1,x_2,x_3,x_4,x_5$ in the $A_2$ cluster algebra 
may be written as Laurent polynomials in the cluster $\mathbf{x}:=\{x_1,x_2\}$  as follows:
\[ \ell_{x_1,\mathbf{x}} = x_1,
\;\;\;
\ell_{x_2,\mathbf{x}} = x_2 ,
\;\;\;
\ell_{x_3,\mathbf{x}} = \frac{x_2+1}{x_1} ,
\;\;\;
\ell_{x_4,\mathbf{x}} = \frac{x_1+x_2+1}{x_1x_2},
\;\;\;
\ell_{x_5,\mathbf{x}} = \frac{x_1+1}{x_2}
\]
The map $f_\mathbf{x}$ identifies the superunitary region with the subset of $\mathbb{R}_{>0}^2$ defined by
\[
x_1\geq1 ,
\;\;\;
x_2 \geq 1,
\;\;\;
\frac{x_2+1}{x_1}\geq 1 ,
\;\;\; 
\frac{x_1+x_2+1}{x_1x_2} \geq 1 ,
\;\;\;
\frac{x_1+1}{x_2}\geq 1
\]
These inequalities carve out a topological pentagon bounded by four lines and a hyperbola (Figure~\ref{fig:A2_superunitary_region}). Each of the faces may be indexed by the set of cluster variables with value $1$ on that face (Figure~\ref{fig:A2_superunitary_region_subclusters}); note that this indexes each vertex by a cluster, each edge by a single cluster variable, and the interior by the empty set.
\end{ex}

\begin{figure}[h!t]
    \begin{tikzpicture}[scale=1.2,baseline=(current bounding box.center)]
        \draw[step=1,black!25,very thin] (-.25,-.25) grid (4.25,4.25);
        \draw[-angle 90] (-.25,0) to (4.5,0) node[above] {$x_1$};
        \draw[-angle 90] (0,-.25) to (0,4.5) node[right] {$x_2$};
        
        \path[fill=blue!10,opacity=\opa, domain=1:2, variable=\t] plot ({\t+1}, {2*\t^(-1)+1}) to (2,1) to (1,1) to (1,2) to (2,3);
        \draw[domain=1:2,thick, blue, variable=\t] plot ({\t+1}, {2*\t^(-1)+1}) to (2,1) to (1,1) to (1,2) to (2,3);
        
        \node[dot,thick,blue,fill=blue!10] (a) at (1,1) {};
        \node[dot,thick,blue,fill=blue!10] (b) at (2,1) {};
        \node[dot,thick,blue,fill=blue!10] (c) at (3,2) {};
        \node[dot,thick,blue,fill=blue!10] (d) at (2,3) {};
        \node[dot,thick,blue,fill=blue!10] (e) at (1,2) {};
        
        \node[below left] (al) at (a) {$\scriptstyle \{x_1,x_2\}$};
        \node[below right] (bl) at (b) {$\scriptstyle \{x_2,x_3\}$};
        \node[right] (cl) at (c) {$\scriptstyle \{x_3,x_4\}$};
        \node[above] (dl) at (d) {$\scriptstyle \{x_4,x_5\}$};
        \node[above left] (el) at (e) {$\scriptstyle \{x_1,x_5\}$};

        \node[below] at (1.5,1) {$
        \{x_2\}$};
        \node[left] at (1,1.5) {$
        \{x_1\}$};
        \node[above left] at (1.6,2.4) {$
        \{x_5\}$};
        \node[above right] at (2.3,2.3) {$ 
        \{x_4\}$};
        \node[below right] at (2.4,1.6) {$ 
        \{x_3\}$};

        \node[] at (1.8,1.8) {$
        \varnothing$};
    \end{tikzpicture}

    \caption{The face decomposition of the $A_2$ superunitary region indexed by subclusters}
    \label{fig:A2_superunitary_region_subclusters}
\end{figure}
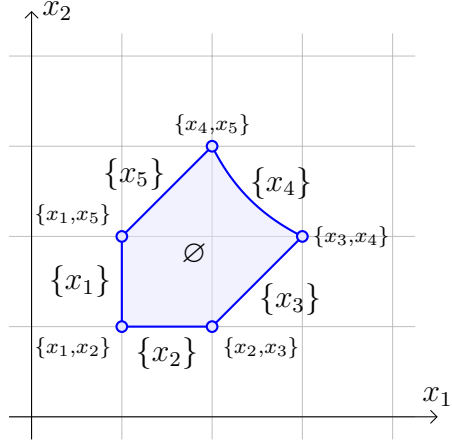

Comparing Figures \ref{fig: a2ass} and \ref{fig:A2_superunitary_region_subclusters}, one may notice that the generalized associahedron and the superunitary region of the $A_2$ cluster algebra are both (topological) pentagons whose faces are indexed by the subclusters. 
Our main theorem is that this is true for all finite type cluster algebras, which we make precise with the language of \emph{CW complexes} (see Section \ref{section: CW}).

\begin{namedthm}{Theorem~\ref{thmIntro:region is generalized associahedron}}
If $\mathcal{A}$ is finite type, then the superunitary region $\mathcal{A}(\mathbb{R}_{\geq1})$ of $\mathcal{A}$ is a regular CW complex which is cellular homeomorphic to the generalized associahedron of $\mathcal{A}$.
\end{namedthm}

\begin{rem}
We can say something stronger: the map $f_\cluster$ embeds the superunitary region into $\mathbb{R}_{>0}^r$ as a \emph{submanifold with corners} (in the sense of \cite{Joy12}) for which the CW structure in Theorem \ref{thmIntro:region is generalized associahedron} is the \emph{boundary stratification} (see Remark \ref{rem: manifoldwithcorners}).
%
\end{rem}

\subsection{Application to friezes}

Theorem \ref{thmIntro:region is generalized associahedron} shows that the superunitary region of a finite type cluster algebra is compact. 
This provides a uniform proof for the finiteness of \emph{positive integral friezes of Dynkin type}, which has so far been proven for all types except $E_7$ and $E_8$.

Given a Dynkin diagram $\Gamma$, one may construct a \emph{repetition quiver} with an infinite row of vertices for each vertex of $\Gamma$ and an infinite zig-zag of arrows for each edge in $\Gamma$. 
A \textbf{positive integral frieze of type $\Gamma$} is a $\mathbb{Z}_{\geq1}$-valued function on the vertices of the repetition quiver of $\Gamma$ which satisfies the \emph{mesh relations}: 
the product of adjacent values in the same row is equal to 1 plus the product of the intermediate values (a precise definition is in Section \ref{section: friezes}). 

\begin{figure}[h!t]
 \begin{tikzpicture}[x=0.025cm,y=-0.025cm]
\begin{scope}[every node/.style={inner sep=4pt,font=\footnotesize}]
    \node (2_a) at (-30,339) {$\dots$};
    \node (4_a) at (-30,439) {$\dots$};
    \node (6_a) at (-30,469) {$\dots$};
    \node (7_a) at (-30,539) {$\dots$};

    \node (1_0) at (20,289) {$2$};
    \node (2_0) at (70,339) {$5$};
    \node (3_0) at (20,389) {$7$};
    \node (4_0) at (70,439) {$18$};
    \node (5_0) at (20,489) {$41$};
    \node (6_0) at (70,469) {$6$};
    \node (7_0) at (70,539) {$11$};
    \node (8_0) at (20,589) {$4$};
    \node (1_1) at (120,289) {$3$};
    \node (2_1) at (170,339) {$8$};
    \node (3_1) at (120,389) {$13$};
    \node (4_1) at (170,439) {$21$};
    \node (5_1) at (120,489) {$29$};
    \node (6_1) at (170,469) {$5$};
    \node (7_1) at (170,539) {$8$};
    \node (8_1) at (120,589) {$3$};
    \node (1_2) at (220,289) {$3$};
    \node (2_2) at (270,339) {$5$};
    \node (3_2) at (220,389) {$13$};
    \node (4_2) at (270,439) {$18$};
    \node (5_2) at (220,489) {$29$};
    \node (6_2) at (270,469) {$6$};
    \node (7_2) at (270,539) {$11$};
    \node (8_2) at (220,589) {$3$};
    \node (1_3) at (320,289) {$2$};
    \node (2_3) at (370,339) {$3$};
    \node (3_3) at (320,389) {$7$};
    \node (4_3) at (370,439) {$16$};
    \node (5_3) at (320,489) {$41$};
    \node (6_3) at (370,469) {$7$};
    \node (7_3) at (370,539) {$15$};
    \node (8_3) at (320,589) {$4$};
    \node (1_4) at (420,289) {$2$};
    \node (2_4) at (470,339) {$5$};
    \node (3_4) at (420,389) {$7$};
    \node (4_4) at (470,439) {$18$};
    \node (5_4) at (420,489) {$41$};
    \node (6_4) at (470,469) {$6$};
    \node (7_4) at (470,539) {$11$};
    \node (8_4) at (420,589) {$4$};
    
    \node (1_5) at (520,289) {$\dots$};
    \node (3_5) at (520,389) {$\dots$};
    \node (5_5) at (520,489) {$\dots$};
    \node (8_5) at (520,589) {$\dots$};
\end{scope}
\begin{scope}[every node/.style={fill=white,font=\scriptsize},every path/.style={-{Latex[length=1.5mm,width=1mm]}}]
    \path (2_a) edge[densely dashed] (1_0);
    \path (2_a) edge[densely dashed] (3_0);
    \path (4_a) edge[densely dashed] (3_0);
    \path (4_a) edge[densely dashed] (5_0);
    \path (6_a) edge[densely dashed] (5_0);
    \path (7_a) edge[densely dashed] (5_0);
    \path (7_a) edge[densely dashed] (8_0);
    \path (1_0) edge (2_0);
    \path (3_0) edge (2_0);
    \path (3_0) edge (4_0);
    \path (5_0) edge (4_0);
    \path (5_0) edge (6_0);
    \path (5_0) edge (7_0);
    \path (8_0) edge (7_0);
    \path (2_0) edge[densely dashed] (1_1);
    \path (2_0) edge[densely dashed] (3_1);
    \path (4_0) edge[densely dashed] (3_1);
    \path (4_0) edge[densely dashed] (5_1);
    \path (6_0) edge[densely dashed] (5_1);
    \path (7_0) edge[densely dashed] (5_1);
    \path (7_0) edge[densely dashed] (8_1);
    \path (1_1) edge (2_1);
    \path (3_1) edge (2_1);
    \path (3_1) edge (4_1);
    \path (5_1) edge (4_1);
    \path (5_1) edge (6_1);
    \path (5_1) edge (7_1);
    \path (8_1) edge (7_1);
    \path (2_1) edge[densely dashed] (1_2);
    \path (2_1) edge[densely dashed] (3_2);
    \path (4_1) edge[densely dashed] (3_2);
    \path (4_1) edge[densely dashed] (5_2);
    \path (6_1) edge[densely dashed] (5_2);
    \path (7_1) edge[densely dashed] (5_2);
    \path (7_1) edge[densely dashed] (8_2);   
    \path (1_2) edge (2_2);
    \path (3_2) edge (2_2);
    \path (3_2) edge (4_2);
    \path (5_2) edge (4_2);
    \path (5_2) edge (6_2);
    \path (5_2) edge (7_2);
    \path (8_2) edge (7_2);
    \path (2_2) edge[densely dashed] (1_3);
    \path (2_2) edge[densely dashed] (3_3);
    \path (4_2) edge[densely dashed] (3_3);
    \path (4_2) edge[densely dashed] (5_3);
    \path (6_2) edge[densely dashed] (5_3);
    \path (7_2) edge[densely dashed] (5_3);
    \path (7_2) edge[densely dashed] (8_3);
    \path (1_3) edge (2_3);
    \path (3_3) edge (2_3);
    \path (3_3) edge (4_3);
    \path (5_3) edge (4_3);
    \path (5_3) edge (6_3);
    \path (5_3) edge (7_3);
    \path (8_3) edge (7_3);
    \path (2_3) edge[densely dashed] (1_4);
    \path (2_3) edge[densely dashed] (3_4);
    \path (4_3) edge[densely dashed] (3_4);
    \path (4_3) edge[densely dashed] (5_4);
    \path (6_3) edge[densely dashed] (5_4);
    \path (7_3) edge[densely dashed] (5_4);
    \path (7_3) edge[densely dashed] (8_4);
    \path (1_4) edge (2_4);
    \path (3_4) edge (2_4);
    \path (3_4) edge (4_4);
    \path (5_4) edge (4_4);
    \path (5_4) edge (6_4);
    \path (5_4) edge (7_4);
    \path (8_4) edge (7_4);
    \path (2_4) edge[densely dashed] (1_5);
    \path (2_4) edge[densely dashed] (3_5);
    \path (4_4) edge[densely dashed] (3_5);
    \path (4_4) edge[densely dashed] (5_5);
    \path (6_4) edge[densely dashed] (5_5);
    \path (7_4) edge[densely dashed] (5_5);
    \path (7_4) edge[densely dashed] (8_5);
\end{scope}
 \end{tikzpicture}
\caption{A positive integral frieze of type $E_8$}
\label{fig: E8frieze}
\end{figure}
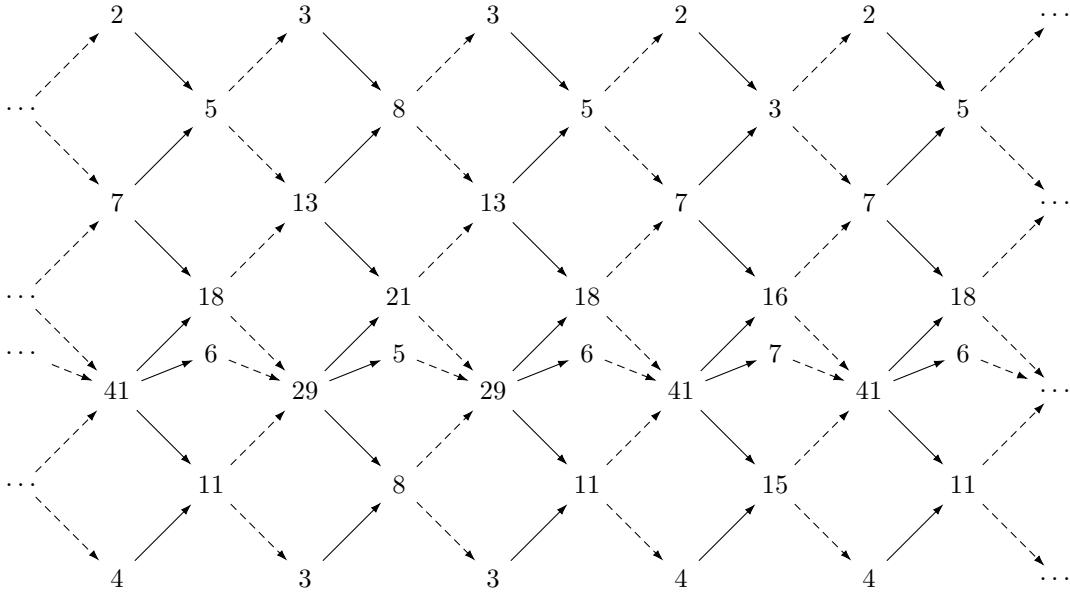

Positive integral friezes of type $A_n$ were introduced by Coxeter \cite{Cox71}
and generalized to other Dynkin diagrams in \cite{CC06,ARS10} based on the following connection to cluster algebras.
Given a Dynkin diagram $\Gamma$, there is a \emph{general frieze of type $\Gamma$} with values in the cluster algebra of type $\Gamma$ such that every positive integral frieze of type $\Gamma$ can be realized by specializing the general frieze along a ring homomorphism $\mathcal{A}\rightarrow \mathbb{R}$ which sends the cluster variables into $\mathbb{Z}_{\geq1}$. Such homomorphisms are called \textbf{frieze points} of $\mathcal{A}$.

The set of frieze points is a closed and discrete subset $\mathcal{A}(\mathbb{Z}_{\geq1})$ of the superunitary region. 
By a standard argument in topology, the compactness of the superunitary region in finite type implies the finiteness of the frieze points, and therefore the positive integral friezes.

\begin{namedthm}{Theorem~\ref{thm: finitefriezes}}
There are finitely many positive integral friezes of each Dynkin type.
\end{namedthm}

The face structure of $\mathcal{A}(\mathbb{R}_{\geq1})$ provides insight into positive integral friezes. Given a positive integral frieze, the corresponding frieze point lies in the face of $\mathcal{A}(\mathbb{R}_{\geq1})$ indexed by the subcluster of cluster variables with value $1$. 
If this subcluster is non-empty, then the frieze can be constructed as a \textbf{unitary extension} of a positive integral frieze on a proper subdiagram. 

In Sections \ref{section: enumeration} and \ref{section: countingfriezes}, 
we use known and conjectured counts of friezes to count those positive integral friezes which are not unitary extensions of smaller friezes. Remarkably, they appear to be quite rare! There are tentatively four types, listed below.

\begin{namedthm}{Theorem~\ref{thm: elementaryfriezes}}
Assuming there are $4400$ and $26952$ positive integral friezes of types $E_7$ and $E_8$,
every positive integral frieze of Dynkin type is a unitary extension of a unique (possibly empty) union of the following friezes: 
\begin{enumerate}
    \item The family of $D_n$ friezes in Figure \ref{fig: Dabfrieze} determined by a proper factorization $n=ab$.
    \item The four translations of the $E_8$ frieze in Figure \ref{fig: E8frieze}.
    \item The family of $B_n$ friezes in Figure \ref{fig: Bn2frieze}, for all $\sqrt{n+1}\in \mathbb{Z}_{\geq2}$.
    \item The $G_2$ frieze at the bottom of Figure \ref{fig: G2frieze}.
\end{enumerate}
\end{namedthm}

\noindent Equivalently, these are the positive integral friezes of connected Dynkin type without $1$s.

\begin{rem}
The vast majority of positive integral friezes are \emph{unitary friezes} (see Table \ref{table: friezes}), which are $1$ on an entire cluster and therefore unitary extensions of the empty frieze.
\end{rem}

\begin{rem}
The last two types of friezes in the theorem may be obtained by \emph{folding} friezes of the first type \cite{FP16}, so one could argue that only the first two types are `elementary'.
\end{rem}

\begin{figure}[h!t]

\begin{tikzpicture}[x=0.025cm,y=-0.025cm]
\begin{scope}[every node/.style={inner sep=4pt,font=\footnotesize}]
    \node (2_a) at (-30,339) {$\dots$};
    \node (4_a) at (-30,439) {$\dots$};
    \node (6_a) at (-30,509) {$\dots$};
    \node (7_a) at (-30,539) {$\dots$};

    \node (1_0) at (20,289) {$2$};
    \node (2_0) at (70,339) {$3$};
    \node (3_0) at (20,389) {$\vdots$};
    \node (4_0) at (70,439) {$(n-2)$};
    \node (5_0) at (20,489) {$(n-1)$};
    \node (6_0) at (70,509) {$b$};
    \node (7_0) at (70,539) {$a$};
    \node (1_1) at (120,289) {$2$};
    \node (2_1) at (170,339) {$3$};
    \node (3_1) at (120,389) {$\vdots$};
    \node (4_1) at (170,439) {$(n-2)$};
    \node (5_1) at (120,489) {$(n-1)$};
    \node (6_1) at (170,509) {$a$};
    \node (7_1) at (170,539) {$b$};
    \node (1_2) at (220,289) {$2$};
    \node (2_2) at (270,339) {$3$};
    \node (3_2) at (220,389) {$\vdots$};
    \node (4_2) at (270,439) {$(n-2)$};
    \node (5_2) at (220,489) {$(n-1)$};
    \node (6_2) at (270,509) {$b$};
    \node (7_2) at (270,539) {$a$};
    \node (1_3) at (320,289) {$2$};
    \node (2_3) at (370,339) {$3$};
    \node (3_3) at (320,389) {$\vdots$};
    \node (4_3) at (370,439) {$(n-2)$};
    \node (5_3) at (320,489) {$(n-1)$};
    \node (6_3) at (370,509) {$a$};
    \node (7_3) at (370,539) {$b$};
    \node (1_4) at (420,289) {$2$};
    \node (2_4) at (470,339) {$3$};
    \node (3_4) at (420,389) {$\vdots$};
    \node (4_4) at (470,439) {$(n-2)$};
    \node (5_4) at (420,489) {$(n-1)$};
    \node (6_4) at (470,509) {$b$};
    \node (7_4) at (470,539) {$a$};
    
    \node (1_5) at (520,289) {$\dots$};
    \node (3_5) at (520,389) {$\dots$};
    \node (5_5) at (520,489) {$\dots$};
\end{scope}
\begin{scope}[every node/.style={fill=white,font=\scriptsize},every path/.style={-{Latex[length=1.5mm,width=1mm]}}]
    \path (2_a) edge[densely dashed] (1_0);
    \path (2_a) edge[densely dashed] (3_0);
    \path (4_a) edge[densely dashed] (3_0);
    \path (4_a) edge[densely dashed] (5_0);
    \path (6_a) edge[densely dashed] (5_0);
    \path (7_a) edge[densely dashed] (5_0);
    \path (1_0) edge (2_0);
    \path (3_0) edge (2_0);
    \path (3_0) edge (4_0);
    \path (5_0) edge (4_0);
    \path (5_0) edge (6_0);
    \path (5_0) edge (7_0);
    \path (2_0) edge[densely dashed] (1_1);
    \path (2_0) edge[densely dashed] (3_1);
    \path (4_0) edge[densely dashed] (3_1);
    \path (4_0) edge[densely dashed] (5_1);
    \path (6_0) edge[densely dashed] (5_1);
    \path (7_0) edge[densely dashed] (5_1);
    \path (1_1) edge (2_1);
    \path (3_1) edge (2_1);
    \path (3_1) edge (4_1);
    \path (5_1) edge (4_1);
    \path (5_1) edge (6_1);
    \path (5_1) edge (7_1);
    \path (2_1) edge[densely dashed] (1_2);
    \path (2_1) edge[densely dashed] (3_2);
    \path (4_1) edge[densely dashed] (3_2);
    \path (4_1) edge[densely dashed] (5_2);
    \path (6_1) edge[densely dashed] (5_2);
    \path (7_1) edge[densely dashed] (5_2);
    \path (1_2) edge (2_2);
    \path (3_2) edge (2_2);
    \path (3_2) edge (4_2);
    \path (5_2) edge (4_2);
    \path (5_2) edge (6_2);
    \path (5_2) edge (7_2);
    \path (2_2) edge[densely dashed] (1_3);
    \path (2_2) edge[densely dashed] (3_3);
    \path (4_2) edge[densely dashed] (3_3);
    \path (4_2) edge[densely dashed] (5_3);
    \path (6_2) edge[densely dashed] (5_3);
    \path (7_2) edge[densely dashed] (5_3);
    \path (1_3) edge (2_3);
    \path (3_3) edge (2_3);
    \path (3_3) edge (4_3);
    \path (5_3) edge (4_3);
    \path (5_3) edge (6_3);
    \path (5_3) edge (7_3);
    \path (2_3) edge[densely dashed] (1_4);
    \path (2_3) edge[densely dashed] (3_4);
    \path (4_3) edge[densely dashed] (3_4);
    \path (4_3) edge[densely dashed] (5_4);
    \path (6_3) edge[densely dashed] (5_4);
    \path (7_3) edge[densely dashed] (5_4);
    \path (1_4) edge (2_4);
    \path (3_4) edge (2_4);
    \path (3_4) edge (4_4);
    \path (5_4) edge (4_4);
    \path (5_4) edge (6_4);
    \path (5_4) edge (7_4);
    \path (2_4) edge[densely dashed] (1_5);
    \path (2_4) edge[densely dashed] (3_5);
    \path (4_4) edge[densely dashed] (3_5);
    \path (4_4) edge[densely dashed] (5_5);
    \path (6_4) edge[densely dashed] (5_5);
    \path (7_4) edge[densely dashed] (5_5);
\end{scope}
 \end{tikzpicture}
\caption{The positive integral frieze of type $D_n$ corresponding to $n=ab$}
\label{fig: Dabfrieze}
\end{figure}
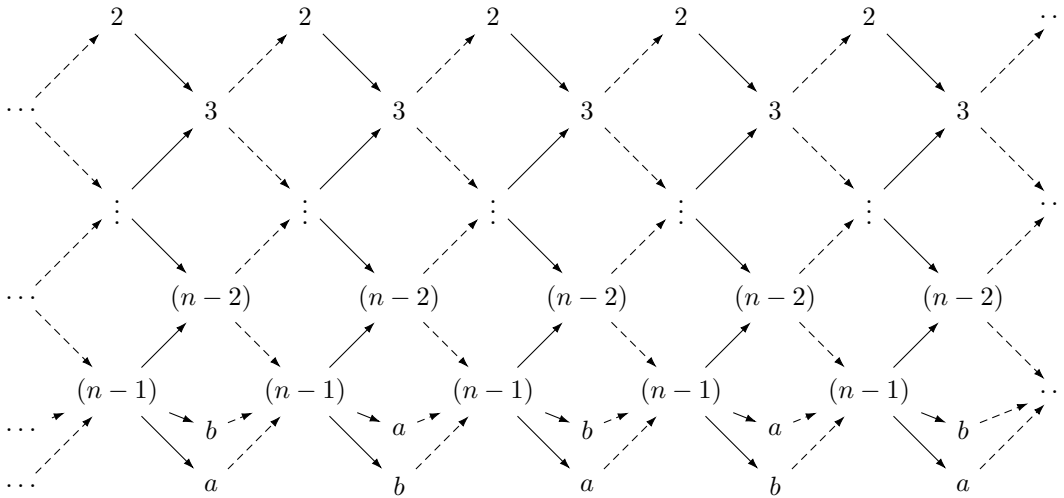


\section{Recollections on cluster algebras}\label{section: recollections}

We assume that the reader is familiar with a basic introduction to cluster algebras, such as \cite{FWZ21}, including the Laurent phenomenon and the finite type classification. 

\begin{conv}
Unless otherwise specified, we work over $\mathbb{Z}$ and \textbf{without frozen variables} or other coefficients. This allows us to avoid handling the variations in how coefficients are used in the literature (e.g.~whether the cluster algebra includes inverses to frozen variables). 

Our definitions can be extended to cluster algebras with coefficients via the definitions in Section \ref{section: infinitetype}; that is, via a positive basis containing the cluster monomials. Using this definition, the superunitary region of a cluster algebra with frozen variables and their inverses is homeomorphic to the superunitary region of the corresponding coefficient-free cluster algebra (see Remark \ref{rem: nofrozens}).
This is our justification for working without coefficients.
\end{conv}

\subsection{Valued quivers}

It will be more convenient to work with \emph{valued quivers} than skew-symmetrizable matrices, so we recall the correspondence.
Given an $n\times n$ skew-symmetrizable matrix $M$, the \textbf{valued quiver} is the quiver with vertex set $\{1,2,...,n\}$ and an arrow $i\rightarrow j$ whenever $M_{i,j}>0$. Each arrow $i\rightarrow j$ has an associated pair of \textbf{values} $(M_{i,j},-M_{j,i})$. 

A \textbf{skew-symmetrizable} valued quiver is one that can be constructed in this way, and it determines a cluster algebra (the cluster algebra of the skew-symmetrizable matrix).

\begin{ex}\label{ex: rank2}
Following \cite{SZ04}, every rank 2 cluster algebra can be described as follows. Let $(b,c)$ be either a pair of positive integers or $(0,0)$, and consider the cluster algebra $\mathcal{A}$ of the following skew-symmetrizable matrix and its associated valued quiver.
\[
\begin{bmatrix}
0 & b \\ -c & 0 
\end{bmatrix}
\hspace{2cm}
\begin{tikzpicture}[baseline=(a.base)]
    \node (a) at (0,0) {$1$};
    \node (b) at (2,0) {$2$};
    \draw[->] (a) to node[above] {$(b,c)$} (b);
\end{tikzpicture}
\]
The cluster variables may be indexed by the integers $\{x_i\}_{i\in \mathbb{Z}}$ and satisfy mutation relations
\[ x_{i-1}x_{i+1} = \left\{\begin{array}{cc}
x_i^b +1 & \text{if $i$ is odd} \\
x_i^c +1 & \text{if $i$ is even} \\
\end{array}\right\}\]
The clusters in $\mathcal{A}$ are pairs $\{x_i,x_{i+1}\}$ of cluster variables with adjacent indices.

If $bc\in \{0,1,2,3\}$, then this sequence of cluster variables is periodic with period $n\in\{4,5,6,8\}$; that is, $x_i=x_{i+n}$. These are the finite type cluster algebras corresponding to Dynkin diagrams $A_1\times A_1$, $A_2$, $B_2/C_2$, and $G_2$, respectively.
In all other cases, each of the $x_i$ are distinct and the cluster algebra is infinite type.
\end{ex}

\subsection{Positivity}

The theory of cluster algebras features a number of closely related positivity results. We summarize a few of these results here.
First, we have the following strengthening of the Laurent phenomenon, which was conjectured at the same time cluster algebras were introduced \cite{CA1} and proven in full generality by \cite{LS15} and \cite{GHKK18}.

\begin{thm}
[Positive Laurent Phenomenon]
\label{thm:positivity}
Let $y$ be a cluster variable and let $\cluster$ be a cluster in the same cluster algebra. Then $y$ may be written as a Laurent polynomial $\ell_{y,\cluster}$ in the elements of $\cluster$ with positive integral coefficients.
\end{thm}

Next, in finite type, the cluster monomials form a basis with positive structure constants; this was proven in~
\cite[Theorem~2.7]{SZ04} for rank $2$, 
\cite[Theorem~1.1]{Cer11}, \cite[Theorem~7.1]{CL12} for type $ADE$, and 
~\cite[Theorem 10.2]{FT17} for type $BCF$. 

\begin{thm} 
\label{thm:cluster monomials form a positive basis in finite type}
If $\sA$ is finite type, the cluster monomials form a $\mathbb{Z}$-basis for $\sA$ with positive integral structure constants; that is, the product of two cluster monomials is an $\mathbb{N}$-linear combination of cluster monomials.
\end{thm}
\begin{rem} 
Theorem \ref{thm:cluster monomials form a positive basis in finite type} and the Laurent phenomenon imply the positive Laurent phenomenon, since the coefficients of the Laurent polynomial $\ell_{y,\cluster}$ become structure constants when multiplying $y$ by a cluster monomial in $\cluster$ with sufficiently large exponents.
\end{rem}

\begin{rem}\label{rem: atomic}
A stronger result is proven in~\cite{SZ04,Cer11,CL12,FT17}; among all elements in a finite type cluster algebra which are \emph{positive} (that is, positive Laurent in every cluster), the cluster monomials are precisely the \emph{atomic} ones (that is, those which cannot be written as a proper sum of two positive elements).\footnote{Some references call these \emph{indecomposable positive} elements; however, other references use that term for positive elements which cannot be factored as a product of two positive elements.} For this reason, the basis of cluster monomials in finite type is sometimes called the \emph{atomic basis}.
We omit this from Theorem \ref{thm:cluster monomials form a positive basis in finite type} because the atomic elements do not form a basis in general (not even in rank 2 when $bc>4$ \cite{LLZ14PositivityAndTameness}). When this happens, one may instead 
choose a sufficiently `good' basis, as in Section \ref{section: goodbases}.
\end{rem}

We will need the following well-known fact about Laurent expansions, which has appeared in the skew-symmetric case in \cite[Lemma~3.7]{CKLP13}.

\begin{lemma}
\label{lemma: sumto1}
Let $y$ be a cluster variable and let $\cluster$ be a cluster in the same cluster algebra. The sum of the coefficients of the Laurent polynomial $\ell_{y,\cluster}$ is 1 if and only if $y\in \cluster$.
\end{lemma}

\begin{proof}
Assume the sum of the coefficients of $\ell_{y,\cluster}$ is $1$. Since the coefficients are positive integers, this forces $\ell_{y,\cluster}$ to be a single Laurent monomial $x_1^{e_1}x_2^{e_2}\cdots x_r^{e_r}$. 
Assume for contradiction one of these exponents $e_i$ is negative, and let $x_i'$ be the cluster variable obtained by mutating $\cluster$ at the cluster variable $x_i$.
Then we have
\begin{equation}
\label{eq:cluster maps to unitary point is injective}
y=x_1^{e_1} x_2^{e_2} \cdots x_r^{e_r}
=\frac
{x_1^{e_1} x_2^{e_2} \cdots x_{i-1}^{e_{i-1}}(x_i')^{-e_i} \, x_{i+1}^{e_{i+1}}\cdots  x_r^{e_r}}
{\displaystyle \left(\prod_{j, M_{ij}>0} x_j^{M_{ij}} + \prod_{j, M_{ij}<0} x_i^{-M_{ij}}\right)^{-e_1}} 
\end{equation}
Either the denominator of the right-hand expression of \eqref{eq:cluster maps to unitary point is injective} is 2 (if $M_{ij}=0$ for all $j$) or it is the sum of distinct monomials.
Either way, Equation~ \eqref{eq:cluster maps to unitary point is injective}
contradicts the fact that $y$ is a positive integral Laurent polynomial in $\cluster \smallsetminus \{ x_i\} \cup \{ x_i'\}$.  
Therefore, $y$ is an ordinary monomial in $\mathbf{x}$. Since the cluster monomials are linearly independent \cite[Theorem~7.20]{GHKK18}, this is only possible if $y\in \cluster$. 
Conversely, if $y\in \cluster$, then $\ell_{y,\cluster}=y$ and the coefficients sum to $1$.
\end{proof}

\subsection{Deletion}
\label{section: deletion}

\def\del{\operatorname{Del}}
A \textbf{subcluster} is a set of cluster variables which are contained in some cluster.
Given a cluster algebra $\mathcal{A}$ and a subcluster $\subcluster$, we may construct a cluster algebra $\mathcal{A}^\dagger$, called the \textbf{deletion of $\subcluster$ in $\mathcal{A}$}. 
Choose a cluster $\mathbf{y}$ containing $\subcluster$, and let $(\mathbf{y},B)$ be the corresponding seed for $\mathcal{A}$. Let $B^\dagger$ denote the submatrix of $B$ in which the rows and columns corresponding to $\subcluster$ have been deleted. On the valued quiver, this process deletes the vertices indexed by $\subcluster$ and any incident arrows. 
The deletion is then defined as the cluster algebra
\[
\mathcal{A}^\dagger := \mathcal{A} (\mathbf{y}\smallsetminus \cluster,B^\dagger)
\]
\begin{prop}
The deletion of $\subcluster$ in $\mathcal{A}$ does not depend on the choice of cluster $\mathbf{y}$ containing $\cluster$; i.e.~for any seed $(\mathbf{y}',B')$ with $\cluster\subset \mathbf{y}'$, there is an isomorphism of cluster algebras
\[ 
\mathcal{A}(\mathbf{y}\smallsetminus \subcluster, B^\dagger) 
\xrightarrow {\sim}
\mathcal{A}(\mathbf{y}'\smallsetminus \subcluster, B'^\dagger) 
\]
\end{prop}

\begin{proof}
By \cite{CaoLi20}, there is a sequence of mutations in $\mathcal{A}$ from $\mathbf{y}$ to $\mathbf{y}'$ which do not mutate at $\cluster$. This descends to a sequence of mutations in $\mathcal{A}^\dagger$, which provides the isomorphism.
\end{proof}

The deletion of a subcluster may be related to the original cluster algebra as follows.

\begin{prop}\label{prop:deletion map}
Let $\mathcal{A}^\dagger$ be the deletion of a subcluster $\subcluster$ in a cluster algebra $\mathcal{A}$. If $\mathcal{A}^\dagger$ is finite type, then there is a ring homomorphism
\[ d:\mathcal{A}\rightarrow \mathcal{A}^\dagger\]
called the \textbf{deletion map} such that:
\begin{enumerate}
    \item Each cluster variable in $\subcluster$ is sent to $1$.
    \item 
    Each cluster variable in $\mathcal{A}$ which is compatible with $\subcluster$ but not in $\subcluster$ is sent to a cluster variable in $\mathcal{A}^\dagger$.
    This induces a bijection between the set of clusters in $\mathcal{A}^\dagger$ and the set of clusters in $\mathcal{A}$ containing $\cluster$.
    \item Each cluster variable in $\mathcal{A}$ which is not compatible with $\cluster$ is sent to a sum of at least two cluster monomials in $\mathcal{A}^\dagger$.
\end{enumerate}
\end{prop}

\begin{proof}
A finite type cluster algebra is acyclic, so by \cite[Corollary 1.19]{CA3}, $\mathcal{A}^\dagger$ is equal to its own upper cluster algebra, and so by \cite[Prop. 3.8]{LACA} 
there is an isomorphism
\[  \mathcal{A} / \langle
x -1 \mid x
\in \subcluster\rangle 
\xrightarrow{\sim} 
\sA(\mathbf{y} \smallsetminus \subcluster,B^\dagger)
=:
\mathcal{A}^\dagger
\]
which sends cluster variables of $\mathcal{A}$ in the complement $\mathbf{y}\smallsetminus \subcluster$ to the corresponding cluster variables in $\mathcal{A}^\dagger$. Since every cluster variable in $\mathcal{A}^\dagger$ is 
of this form, 
part (2) follows.

To show (3), consider an arbitrary cluster variable $y$ in $\mathcal{A}$. 
For any cluster $\yy$ 
in $\mathcal{A}$ containing $\cluster$, the Laurent expansion of $y$ in $\yy$ 
is positive Laurent. 
Since the deletion map $d$ simply specializes $\cluster$ to $1$, the image $d(y)$ is positive Laurent in the cluster 
$\yy\smallsetminus \cluster$. 
Therefore, $d(y)$ is positive Laurent in every cluster in $\mathcal{A}^\dagger$. Since the cluster monomials constitute the atomic basis in finite type (see Remark \ref{rem: atomic}), $d(y)$ is a sum of cluster monomials in $\mathcal{A}^\dagger$.

Suppose $d(y)$ is a single cluster monomial in some cluster $\cluster'$ in $\mathcal{A}^\dagger$.
By part (2), this cluster corresponds to a cluster $\cluster\sqcup \cluster'$ in $\mathcal{A}$.
Since $d(\ell_{y,\cluster\sqcup \cluster'}) = \ell_{d(y),\cluster'}$ is a monomial in $\cluster'$, the Laurent polynomial $\ell_{y,\cluster\sqcup \cluster'}$ must be a Laurent monomial in $\cluster\sqcup \cluster'$. By Lemma~\ref{lemma: sumto1}, this forces $y\in \cluster \sqcup\cluster'$ and so $y$ is compatible with $\cluster$.
\end{proof}

\begin{rem}
Parts (1) and (2) of the proposition holds whenever $\mathcal{A}^\dagger=\mathcal{A}[\subcluster^{-1}]$; that is, $\mathcal{A}^\dagger$ is a \emph{cluster localization} in the terminology of \cite{LACA}. This is known to hold in far greater generality than finite type; for example, it holds whenever $\mathcal{A}^\dagger$ is \emph{locally acyclic}. 
\end{rem}

\section{The superunitary region (and other sets of real points)}

In this section, we consider several topological spaces associated to a cluster algebra:
\[ \mathcal{A}(\mathbb{R}) 
\supset \mathcal{A}(\mathbb{R}_{>0}) 
\supset \mathcal{A}(\mathbb{R}_{\geq1}) 
\supset \mathcal{A}(\mathbb{Z}_{\geq1}) 
\]
These are (read left to right) the \emph{space of real points}, the \emph{totally positive region}, the \emph{superunitary region}, and the \emph{set of frieze points}.
All have been studied in earlier works except the superunitary region, which is our primary object of interest.

\subsection{The space of real points}

To a cluster algebra $\mathcal{A}$, we associate the set\footnote{The notation here is a little unconventional. Normally, one would first introduce notation like $X$ for the scheme associated to the ring, and then let $X(\mathbb{R})$ denote the variety of $\mathbb{R}$-valued points. Since we are avoiding varieties and schemes entirely, we denote the set of $\mathbb{R}$-valued points using the algebra, not the scheme.}
\[ \mathcal{A}(\mathbb{R}) := \{ \text{ring homomorphisms $p:\mathcal{A}\rightarrow \mathbb{R}$}\} \]
Each element $a\in \mathcal{A}$ defines an $\mathbb{R}$-valued function $f_a$ on $\mathcal{A}(\mathbb{R})$, via the rule that $f_a(p):=p(a)$. 
We make $\mathcal{A}(\mathbb{R})$ into a topological space by endowing it with the coarsest topology for which $f_a:\mathcal{A}(\mathbb{R})\rightarrow \mathbb{R}$ is continuous for all $a\in \mathcal{A}$. 

\begin{rem}
As defined, $\mathcal{A}(\mathbb{R})$ is the variety of \emph{real-valued points} of $\mathrm{Spec}(\mathcal{A})$ endowed with the \emph{analytic} or \emph{Euclidean} topology (as opposed to the Zariski topology).
\end{rem}


More generally, to each subset $S\subset \mathbb{R}$, we may associate a subspace of $\mathcal{A}(\mathbb{R})$ as follows.

\begin{defn}\label{defn: AS}
Given a cluster algebra $\mathcal{A}$ and a subset $S\subset \mathbb{R}$, let \[ \mathcal{A}(S) := \{ p:\mathcal{A}\rightarrow \mathbb{R} \mid p(x) \in S  \text{ for all cluster variables $x$}\} \subset \mathcal{A}(\mathbb{R})\]
This set is endowed with the subspace topology from the inclusion $\mathcal{A}(S)\subset \mathcal{A}(\mathbb{R})$.
\end{defn}

\noindent This topological space can be equivalently defined as the subspace of $\mathcal{A}(\mathbb{R})$ on which $f_x$ takes values in $S$, for all cluster variables $x$.

In the following sections, we consider three special cases of this construction: the \emph{totally positive region} $\mathcal{A}(\mathbb{R}_{>0})$, the \emph{superunitary region} $\mathcal{A}(\mathbb{R}_{\geq1})$, and the \emph{frieze points} $\mathcal{A}(\mathbb{Z}_{\geq1})$.

\begin{warn}
The superior version of Definition \ref{defn: AS} is given by Definition \ref{defn: ASgeneral}, which depends on a choice of positive basis for the cluster algebra but is better behaved in general. 
Since the two definitions coincide for coefficient-free finite type cluster algebras, we use Definition \ref{defn: AS} for its simplicity and its independence from a choice of basis.
\end{warn}

\subsection{The totally positive region}

The \textbf{totally positive region} of a cluster algebra $\mathcal{A}$, denoted $\mathcal{A}(\mathbb{R}_{>0})$, has the following equivalent characterizations.
\begin{itemize}
    \item The set of ring homomorphisms $p:\mathcal{A}\rightarrow \mathbb{R}$ which send each cluster variable into $\mathbb{R}_{>0}$.
    \item The subspace of $\mathcal{A}(\mathbb{R})$ on which $f_x$ is a positive function for each cluster variable $x$.
\end{itemize}

The topology of the totally positive region is particularly simple. As the following proposition shows, $\mathcal{A}(\mathbb{R}_{>0})$ may be identified with a positive orthant in affine space.

\begin{prop}
\label{prop:homemomorphism from totally positive region to positive orthant}
Let $\mathbf{x}=(x_1,x_2,...,x_r)$ be a cluster in $\mathcal{A}$. Then the map
\[f_\mathbf{x}:\mathcal{A}(\mathbb{R}_{>0}) \rightarrow \mathbb{R}_{>0}^r\]
defined by
$ f_\mathbf{x} (p) := (f_{x_1}(p),f_{x_2}(p),...,f_{x_r}(p)) = (p(x_1),p(x_2),...,p(x_r)) $
is a homeomorphism.
\end{prop}

\begin{proof}
The map $f_\cluster$ is continuous because each $f_{x_i}$ 
is continuous by the definition of the topology on $\sA(\RR)$.
Each $v:=(v_1, v_2, ..., v_r) \in \RRpos^r$, 
determines a ring homomorphism
\[ p: \mathbb{Z}[x_1^{\pm1},x_2^{\pm1},...,x_r^{\pm1}] \rightarrow \mathbb{R} \]
by the rule that $p(x_1^{e_1}x_2^{e_2}\cdots x_r^{e_r}) = v_1^{e_1}v_2^{e_2}\cdots v_r^{e_r}$. 
By the positive Laurent phenomenon (Theorem \ref{thm:positivity}), this map restricts to a ring homomorphism $p:\mathcal{A}\rightarrow \mathbb{R}$ which sends each cluster variable into $\mathbb{R}_{>0}$. The map $v\mapsto p$ is the inverse map to $f_\cluster$.

To show that $f_\cluster^{-1}$ is continuous, note that, for each $i$, the composition $f_{\ClusterVariable{i}} \circ f_\cluster^{-1}: \RRpos^r \to \RR$ is the projection map $(v_1, ..., v_r) \mapsto v_i$, and so $f_{\ClusterVariable{i}} \circ f_\cluster^{-1}$ is continuous. Since any $a\in \mathcal{A}$ may be written as a Laurent polynomial in the cluster $\cluster$, the composition
$f_a \circ f_\cluster^{-1}$ is continuous for all $a \in \sA$, and so by the definition of the topology on $\sA(\mathbb{R})$, the map $f_\cluster^{-1}$ is continuous.
\end{proof}

\begin{rem}
Proposition~\ref{prop:homemomorphism from totally positive region to positive orthant} is well-known, and may be proven without Theorem~\ref{thm:positivity}.
\end{rem}

Under the identification $f_\cluster:\mathcal{A}(\mathbb{R}_{>0})\xrightarrow{\sim}\mathbb{R}_{>0}^r$, a function of the form $f_a:\mathcal{A}(\mathbb{R}_{>0})\rightarrow\mathbb{R}$ becomes the function 
\[ \ell_{a,\mathbf{x}}:\mathbb{R}_{>0}^r\rightarrow \mathbb{R} \]
given by writing $a$ as a Laurent polynomial in $\mathbf{x}$. Note that, if $a$ is a cluster variable, $\ell_{a,\cluster}$ is a positive Laurent polynomial, and so $\ell_{a,\cluster}$ sends $\mathbb{R}_{>0}^r$ into $\mathbb{R}_{>0}$.

Proposition~\ref{prop:homemomorphism from totally positive region to positive orthant} gives many different identifications of the form $\mathcal{A}(\mathbb{R}_{>0})\simeq \mathbb{R}_{>0}^r$; one for each cluster.  
These identifications may be related using the following maps. For any clusters $\mathbf{x}, \mathbf{x}'$ in $\mathcal{A}$, define a map
$ \mu_{\mathbf{x}',\mathbf{x}}:\mathbb{R}_{>0}^r\rightarrow \mathbb{R}_{>0}^r$
by 
\[ \mu_{\mathbf{x}',\mathbf{x}}(x_1,x_2,...,x_r) := (\ell_{x_1',\mathbf{x}}(x_1,x_2,...,x_r),\ell_{x_2',\mathbf{x}}(x_1,x_2,...,x_r),...,\ell_{x_r',\mathbf{x}}(x_1,x_2,...,x_r))\]
The following lemma summarizes the relevant properties of these maps.

\begin{lemma}
\begin{enumerate}
    \item For any clusters $\mathbf{x}, \mathbf{x}'$ in $\mathcal{A}$, $\mu_{\mathbf{x}',\mathbf{x}}\circ f_{\mathbf{x}} = f_{\mathbf{x}'}$.
    \item For any $a\in \mathcal{A}$ and any clusters $\mathbf{x}, \mathbf{x}'$ in $\mathcal{A}$, $\ell_{a,\mathbf{x}'}\circ \mu_{\mathbf{x}',\mathbf{x}} = \ell_{a,\mathbf{x}}$.
    \item For any clusters $\mathbf{x}, \mathbf{x}',\mathbf{x}''$ in $\mathcal{A}$, $\mu_{\mathbf{x}'',\mathbf{x}'}\circ \mu_{\mathbf{x}',\mathbf{x}} = \mu_{\mathbf{x}'',\mathbf{x}}$.
    \item For any clusters $\mathbf{x}, \mathbf{x}'$ in $\mathcal{A}$, $\mu_{\mathbf{x}',\mathbf{x}}:\mathbb{R}_{>0}^r\rightarrow \mathbb{R}_{>0}^r$ is a diffeomorphism with inverse $\mu_{\mathbf{x},\mathbf{x}'}$.
\end{enumerate}

\end{lemma}

\begin{rem}
The relations between $f_a$, $f_\mathbf{x}$, $\ell_{a,\mathbf{x}}$ and $\mu_{\mathbf{x}',\mathbf{x}}$ may be summarized by the commutativity of the following diagram.
\[
\begin{tikzpicture}[scale=1.5]
    \node (TPR) at (0,0) {$\mathcal{A}(\mathbb{R}_{>0})$};
    \node (R) at (0,-1) {$\mathbb{R}$};
    \node (Rr1) at (-1.25,1) {$\mathbb{R}_{>0}^r$};
    \node (Rr2) at (1.25,1) {$\mathbb{R}_{>0}^r$};
    \draw[->] (TPR) to node[ left] {$f_a$} (R);
    \draw[->] (TPR) to node[below left] {$f_{\mathbf{x}}$} (Rr1);
    \draw[->] (TPR) to node[below right] {$f_{\mathbf{x}'}$} (Rr2);
    \draw[->,out=270,in=150] (Rr1) to node[below left] {$\ell_{a,\mathbf{x}}$} (R);
    \draw[->,out=270,in=30] (Rr2) to node[below right] {$\ell_{a,\mathbf{x}'}$} (R);
    \draw[->] (Rr1) to node[above] {$\mu_{\mathbf{x}',\mathbf{x}}$} (Rr2);
\end{tikzpicture}  
\]
Note that, if $a$ is a cluster variable, the $\mathbb{R}$ on the bottom may be replaced by $\mathbb{R}_{>0}$.
\end{rem}

As a consequence, any two clusters give diffeomorphic identifications of $\mathcal{A}(\mathbb{R}_{>0})$ with $\mathbb{R}_{>0}^r$. Therefore, any $f_\mathbf{x}$ may be used as a global chart to give $\mathcal{A}(\mathbb{R}_{>0})$ a smooth manifold structure, and the resulting smooth structure does not depend on the choice of $\mathbf{x}$.

\begin{prop}
The totally positive region $\mathcal{A}(\mathbb{R}_{>0})$ has a unique smooth manifold structure such that, for each cluster $\mathbf{x}$, the map $f_{\mathbf{x}}:\mathcal{A}(\mathbb{R}_{>0})\xrightarrow {\sim} \mathbb{R}_{>0}^r$ is a diffeomorphism.
\end{prop}

\begin{rem}
The space of real points $\mathcal{A}(\mathbb{R})$ need not have a manifold structure, smooth or otherwise; e.g.~if $\mathcal{A}$ is the $A_3$ cluster algebra, $\mathcal{A}(\mathbb{R})$ is a singular real variety \cite{BFMS21}.
\end{rem}

\subsection{The superunitary region}

The \textbf{superunitary region} of a cluster algebra $\mathcal{A}$ is the space $\mathcal{A}(\mathbb{R}_{\geq1})$, which has the following equivalent characterizations.
\begin{itemize}
    \item The set of ring homomorphisms $p:\mathcal{A}\rightarrow \mathbb{R}$ which send each cluster variable into $\mathbb{R}_{\geq1}$.
    \item The subspace of $\mathcal{A}(\mathbb{R})$ on which $f_x$ is greater than or equal to $1$, for each cluster variable $x$.
\end{itemize}
The superunitary region $\mathcal{A}(\mathbb{R}_{\geq1})$ is a subspace of the totally positive region $\mathcal{A}(\mathbb{R}_{>0})$. By Proposition~\ref{prop:homemomorphism from totally positive region to positive orthant}, a choice of cluster $\mathbf{x}$ in $\mathcal{A}$ identifies the superunitary region $\mathcal{A}(\mathbb{R}_{\geq1})$ with
\[ f_\mathbf{x}(\mathcal{A}(\mathbb{R}_{\geq1})) \subset \mathbb{R}_{>0}^r\]
This subset is explicitly carved out by a family of Laurent inequalities, as follows.

\begin{prop}
For any cluster $\mathbf{x}$,
the set $f_\mathbf{x}(\mathcal{A}(\mathbb{R}_{\geq1}))$ is the subset of $\mathbb{R}_{>0}^r$ on which the Laurent polynomial $\ell_{x,\mathbf{x}}$ is greater than or equal to $1$ for each cluster variable $x$ in $\mathcal{A}$. 
\end{prop}

\begin{ex}
Consider the $B_2/C_2$ cluster algebra, the $(b,c)=(1,2)$ case of Example \ref{ex: rank2}.
The Laurent expansions of the four non-initial clusters in the initial cluster $\cluster:=(x_1,x_2)$ are
\[
\ell_{x_3,\mathbf{x}} = \frac{x_2^2+1}{x_1} ,
\;\;\;
\ell_{x_4,\mathbf{x}} = \frac{x_1+x_2^2+1}{x_1x_2},
\;\;\;
\ell_{x_5,\mathbf{x}} = \frac{x_1^2+x_2^2+2x_1+1}{x_1x_2^2},
\;\;\;
\ell_{x_6,\mathbf{x}} = \frac{x_1+1}{x_2}
\]
The map $f_\mathbf{x}$ identifies the superunitary region with the subset of $\mathbb{R}_{>0}^2$ defined by
\[
x_1\geq1 ,
\;\;\;
x_2 \geq 1,
\;\;\;
\frac{x_2^2+1}{x_1}\geq 1 ,
\;\;\; 
\frac{x_1+x_2^2+1}{x_1x_2} \geq 1
\;\;\;
\frac{x_1^2+x_2^2+2x_1+1}{x_1x_2^2} \geq 1
\;\;\;
\frac{x_1+1}{x_2}\geq 1
\]
These inequalities carve out a topological hexagon in $\mathbb{R}^2$ (Figure~\ref{fig:B2_superunitary_region}). 
\end{ex}

\begin{figure}[h!t] 
    \begin{tikzpicture}[x={(1,1)},y={(1,-1)},x={(1,-1)},scale=.75,baseline=(current bounding box.center)]
     \draw[step=1,black!25,very thin] (-.25,-.25) grid (4.25,6.25);
		\draw[-angle 90] (-.25,0) to (4.5,0) node[right] {$x_2$};
 		\draw[-angle 90] (0,-.25) to (0,6.5) node[above] {$x_1$};
		
		\draw[blue!10,fill=blue!10,opacity=\opa,variable=\t,domain=1:2] plot (\t,{\t^2+1}) plot ({\t+1},{\t+2+2/\t}) to (2,2) to (1,2);
 		\draw[blue!10,fill=blue!10,opacity=\opa,variable=\t,domain=1:2] plot ({\t+2*\t^(-1)},{1+4*\t^(-2)}) to (2,2) to (3,5);
 		\draw[blue!10,fill=blue!10,opacity=\opa,variable=\t,domain=1:2] (1,1) to (1,2) to (3,2) to (2,1) to (1,1);
     \draw[variable=\t,domain=1:2] (1,1) to (1,2) plot (\t,\t^2+1) plot (\t+1,{\t+2+2*\t^(-1)}) plot ({\t+2*\t^(-1)},{1+4*\t^(-2)}) to (2,1) to (1,1);
    
    \draw[continued] (1,-.25) to (1,1);
    \draw[continued] (1,2) to (1,6.25);
    \draw[continued,variable=\t,domain=-.25:1] plot (\t,\t^2+1);
    \draw[continued,variable=\t,domain=2:2.3] plot (\t,\t^2+1);
    \draw[continued,variable=\t,domain=.55:1] plot (\t+1,\t+2+2/\t);
    \draw[continued,variable=\t,domain=2:3.25] plot (\t+1,\t+2+2/\t);
    \draw[continued,variable=\t,domain=.88:1] plot ({\t+2*\t^(-1)},{1+4*\t^(-2)});
    \draw[continued,variable=\t,domain=2:3.75] plot ({\t+2*\t^(-1)},{1+4*\t^(-2)});
    
    \draw[continued] (-.25,1) to (1,1);
    \draw[continued] (2,1) to (4.25,1);
    \draw[continued] (.75,-.25) to (2,1);
    \draw[continued] (3,2) to (4.25,3.25);
\end{tikzpicture}
    \caption{The superunitary region of the $B_2/C_2$ cluster algebra (embedded in $\mathbb{R}^2_{>0}$)}
    \label{fig:B2_superunitary_region}
\end{figure}

\begin{ex}
Consider the $G_2$ cluster algebra, the $(b,c)=(1,3)$ case of Example \ref{ex: rank2}.
The Laurent expansions of the six non-initial clusters in the initial cluster $\cluster:=(x_1,x_2)$ are
\[
\ell_{x_3,\mathbf{x}} = \frac{x_2^3+1}{x_1} ,
\;\;\;
\ell_{x_4,\mathbf{x}} = \frac{x_1+x_2^3+1}{x_1x_2},
\;\;\;
\ell_{x_5,\mathbf{x}} = \frac{
x_2^6+3 x_1 x_2^3+2 x_2^3+x_1^3+3 x_1^2+3 x_1+1
}{x_1^2x_2^3}
\]
\[
\ell_{x_6,\mathbf{x}} = \frac{x_2^3+x_1^2+2x_1+1}{x_1x_2^2},
\;\;\;
\ell_{x_7,\mathbf{x}} = \frac{x_1^3+x_2^3+3x_1^2 + 3x_1+1}{x_1x_2^3},
\;\;\;
\ell_{x_8,\mathbf{x}} = \frac{x_1+1}{x_2}
\]
Setting each of these $\geq1$ carves out the rather sharp topological octagon in Figure \ref{fig:G2_superunitary_region}.
\end{ex}

\begin{figure}[h!t] 
\begin{tikzpicture}[x={(1,1)},y={(1,-1)},x={(1,-1)},scale=.5,baseline=(current bounding box.center)]
    \draw[step=1,black!25,very thin] (-.25,-.25) grid (6.25,15.25);
		\draw[-angle 90] (-.25,0) to (6.5,0) node[right] {$x_2$};
		\draw[-angle 90] (0,-.25) to (0,15.5) node[above] {$x_1$};
		
		\draw[blue!10,fill=blue!10,opacity=\opa,variable=\t,domain=1:2] (3,6) plot (\t,\t^3+1) to (3,6);
		\draw[blue!10,fill=blue!10,opacity=\opa,variable=\t,domain=1:2] (3,6) plot ({\t+1},{\t^2+3*\t+3+2*\t^(-1)}) to (3,6);
		\draw[blue!10,fill=blue!10,opacity=\opa,variable=\t,domain=1:2] (3,6) plot ({\t^2+2*\t^(-1)},{\t^3+5+8*\t^(-3)}) to (3,6);
		\draw[blue!10,fill=blue!10,opacity=\opa,variable=\t,domain=1:2] (3,6) plot ({\t+2+2*\t^(-1)},{\t+3+6*\t^(-1)+4*\t^(-2)}) to (3,6);
		\draw[blue!10,fill=blue!10,opacity=\opa,variable=\t,domain=1:2] (3,6) plot ({\t+4*\t^(-2)},{1+8*\t^(-3)}) to (3,6);
		\draw[blue!10,fill=blue!10,opacity=\opa,variable=\t,domain=1:2] (3,6) to (3,2) to (2,1) to (1,1) to (1,2) to (3,6);
		\draw[variable=\t,domain=1:2] plot (\t,\t^3+1) plot ({\t+1},{\t^2+3*\t+3+2*\t^(-1)}) plot ({\t^2+2*\t^(-1)},{\t^3+5+8*\t^(-3)}) plot ({\t+2+2*\t^(-1)},{\t+3+6*\t^(-1)+4*\t^(-2)}) plot ({\t+4*\t^(-2)},{1+8*\t^(-3)}) to (2,1) to (1,1) to (1,2);

		\draw[continued,variable=\t,domain=-.25:1] plot (\t,\t^3+1);
		\draw[continued,variable=\t,domain=2:2.43] plot (\t,\t^3+1);
		\draw[continued,variable=\t,domain=.17:1] plot ({\t+1},{\t^2+3*\t+3+2*\t^(-1)});
		\draw[continued,variable=\t,domain=2:2.2] plot ({\t+1},{\t^2+3*\t+3+2*\t^(-1)});
		\draw[continued,variable=\t,domain=.96:1] plot ({\t^2+2*\t^(-1)},{\t^3+5+8*\t^(-3)});
		\draw[continued,variable=\t,domain=2:2.11] plot ({\t^2+2*\t^(-1)},{\t^3+5+8*\t^(-3)});
		\draw[continued,variable=\t,domain=.92:1] plot ({\t+2+2*\t^(-1)},{\t+3+6*\t^(-1)+4*\t^(-2)});
		\draw[continued,variable=\t,domain=2:3.6] plot ({\t+2+2*\t^(-1)},{\t+3+6*\t^(-1)+4*\t^(-2)});
		\draw[continued,variable=\t,domain=.87:1] plot ({\t+4*\t^(-2)},{1+8*\t^(-3)});
		\draw[continued,variable=\t,domain=2:6] plot ({\t+4*\t^(-2)},{1+8*\t^(-3)});
    \draw[continued] (1,-.25) to (1,1);
    \draw[continued] (1,2) to (1,15.25);
    \draw[continued] (-.25,1) to (1,1);
    \draw[continued] (2,1) to (6.25,1);
    \draw[continued] (.75,-.25) to (2,1);
    \draw[continued] (3,2) to (6.25,5.25);
  \end{tikzpicture} 
\caption{The superunitary region of the $G_2$ cluster algebra (embedded in $\mathbb{R}^2_{>0}$)}\label{fig:G2_superunitary_region}
\end{figure}

\begin{ex}
Fix a seed $\cluster=(x_1,x_2,x_3)$ whose quiver is the following cyclic triangle. 
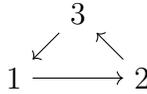
\[
\begin{tikzpicture}[scale=1.7]
\node (1) at (0,0) {$1$};
\node (2) at (1,0) {$2$};
\node (3) at (0.5,0.5) {$3$};
\draw[->] (1) to (2);
\draw[->] (2) to (3);
\draw[->] (3) to (1);
\end{tikzpicture}  
\]
While this quiver is not Dynkin, any mutation of it is a Dynkin quiver of type $A_3$, and so the associated cluster algebra is the $A_3$ cluster algebra. As Laurent polynomials in the initial cluster $\cluster$, the nine cluster variables are
\[
\begin{array}{ccc}
x_1 & x_2 & x_3 \\
\frac{x_2+x_3}{x_1} & 
\frac{x_1+x_3}{x_2} &
\frac{x_1+x_2}{x_3} \\
\frac{x_1+x_2+x_3}{x_1x_2} &
\frac{x_1+x_2+x_3}{x_2x_3} &
\frac{x_1+x_2+x_3}{x_1x_3}
\end{array}
\]
Setting each Laurent polynomial to be greater than or equal to 1 defines the topological polyhedron in Figure~\ref{fig:A3_superunitary_region}, with the same face structure as a 3-dimensional associahedron.
\end{ex}

\begin{figure}[h!t]
    \begin{tikzpicture}[xshift=2.875in,yshift=-.49cm,scale=.75,baseline=(current bounding box.center)]
   	\path[fill=blue!10,variable=\t,domain=2:3] plot ({5-\t+.35*(5-\t)},{2*(5-\t)*(4-\t)^(-1)+.49*(5-\t)}) to ($(1,3)+2*(.35,.49)$) to ($(1,2)+(.35,.49)$) to($(1,1)+(.35,.49)$) to ($(2,1)+(.35,.49)$) to ($(3,1)+2*(.35,.49)$) to ($(4,2)+2*(.35,.49)$) plot ({2*\t*(\t-1)^(-1)+.35*\t}, {\t+.49*\t});
		\draw[-angle 90] (-.25,0) to (5.25,0) node[above] {$x_1$};
		\draw[-angle 90] (0,-.25) to (0,5.25) node[right] {$x_2$};
		\draw ($-.25*(.36,.49)$) to (1.35,1.89) node[above left] {$x_3$};
		\draw[-angle 90, gray, very thin] (1.35,1.89)to ($5.25*(.36,.49)$);
		
     \draw[variable=\t,domain=1:2] plot (\t+1+.35,{2*\t^(-1)+1+.49}) to ($(2,1)+(.35,.49)$) to ($(1,1)+(.35,.49)$) to ($(1,2)+(.35,.49)$) to ($(2,3)+(.35,.49)$);
		\draw ($(1,2)+(.35,.49)$) to ($(1,3)+2*(.35,.49)$) to ($(2,4)+2*(.35,.49)$) to ($(2,3)+(.35,.49)$);
		\draw[gray,very thin] ($(1,1)+(.35,.49)$) to ($(1,1)+2*(.35,.49)$);
		\draw[gray,very thin] ($(1,1)+2*(.35,.49)$) to ($(1,2)+3*(.35,.49)$);
 		\draw[gray,very thin,variable=\t,domain=1:2] plot ({1+.35*\t+.35*1} ,{2*\t^(-1)+1+.49*\t+.49*1});
 		\draw[gray,very thin,variable=\t,domain=1:2] plot ({2*\t^(-1)+1+.35*\t+.35*1},{1+.49*\t+.49*1});
		\draw[gray,very thin] ($(1,2)+3*(.35,.49)$) to ($(2,2)+4*(.35,.49)$) to ($(2,1)+3*(.35,.49)$) to ($(1,1)+2*(.35,.49)$);
		
		\draw ($(2,1)+(.35,.49)$) to ($(3,1)+2*(.35,.49)$) to ($(4,2)+2*(.35,.49)$) to ($(3,2)+(.35,.49)$);
		\draw[variable=\t,domain=2:3] plot ({2*\t*(\t-1)^(-1)+.35*\t}, {\t+.49*\t});
 		\draw[variable=\t,domain=2:3] plot ({\t+.35*\t},{2*\t*(\t-1)^(-1)+.49*\t});
 		\draw[gray,very thin,variable=\t,domain=2:3] plot ({\t+.35*2*\t*(\t-1)^(-1)}, {\t+.49*2*\t*(\t-1)^(-1)});
  \end{tikzpicture} 
    \caption{The superunitary region of the $A_3$ cluster algebra (embedded in $\mathbb{R}^3_{>0}$)}
    \label{fig:A3_superunitary_region}
\end{figure}
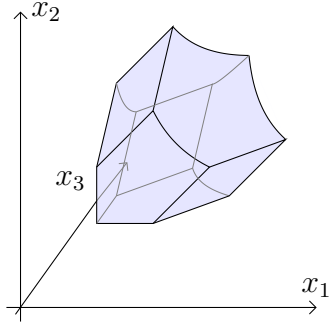

\subsection{The frieze points}
We call $\mathcal{A}(\mathbb{Z}_{\geq1})$ the set of \textbf{frieze points}, which has the following equivalent characterizations.
\begin{itemize}
    \item The set of ring homomorphisms $p:\mathcal{A}\rightarrow \mathbb{R}$ which send each cluster variable into $\mathbb{Z}_{\geq1}$.
    \item The subset of $\mathcal{A}(\mathbb{R})$ on which $f_x$ is a positive integer, for each cluster variable $x$.
\end{itemize}
The name `frieze points' is justified by their connection to \emph{friezes}; see Section \ref{section: friezes}.


Frieze points are superunitary, 
but their topology is trivial in the following sense.
\begin{prop}\label{prop: closeddiscrete}
The set of frieze points is closed and discrete in the superunitary region.
\end{prop}
\begin{proof}
Given a cluster $\cluster$, the identification $f_\cluster:\mathcal{A}(\mathbb{R}_{>0})\xrightarrow {\sim}  \mathbb{R}_{>0}^{|\cluster|}$ restricts to an inclusion $\mathcal{A}(\mathbb{Z}_{\geq1})\hookrightarrow \mathbb{Z}_{\geq1}^{|\cluster|}$. Since the latter set is closed and discrete in $\mathbb{R}_{>0}^{|\cluster|}$, the result follows.
\end{proof}

\begin{ex}\label{ex: G2friezepoints}
There are nine frieze points for the cluster algebra of type $G_2$ (see Figure~\ref{fig:G2_superunitary_region with a point in the interior}). In terms of their values on the initial cluster, they are:
\[(1, 1), (1, 2), (9,2), (14,3), (14,5), (9,5), (2,3), (2,1), (3,2)\]
Eight are at the corners of the superunitary region, and the ninth is in the interior.
\end{ex}

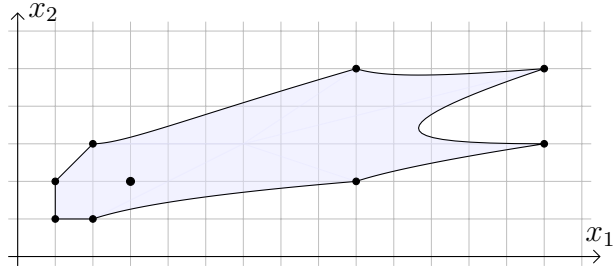
\begin{figure}[htb!]
\begin{tikzpicture}[x={(1,1)},y={(1,-1)},x={(1,-1)},scale=.5,baseline=(current bounding box.center)]
    \draw[step=1,black!25,very thin] (-.25,-.25) grid (6.25,15.25);
		\draw[-angle 90] (-.25,0) to (6.5,0) node[right] {$x_2$};
		\draw[-angle 90] (0,-.25) to (0,15.5) node[above] {$x_1$};

\draw[blue!10,fill=blue!10,variable=\t,domain=1:2,opacity=.5] (3,6) plot (\t,\t^3+1) to (3,6);
\draw[blue!10,fill=blue!10,variable=\t,domain=1:2,opacity=.5] (3,6) plot ({\t+1},{\t^2+3*\t+3+2*\t^(-1)}) to (3,6);
\draw[blue!10,fill=blue!10,variable=\t,domain=1:2,opacity=.5] (3,6) plot ({\t^2+2*\t^(-1)},{\t^3+5+8*\t^(-3)}) to (3,6);
\draw[blue!10,fill=blue!10,variable=\t,domain=1:2,opacity=.5] (3,6) plot ({\t+2+2*\t^(-1)},{\t+3+6*\t^(-1)+4*\t^(-2)}) to (3,6);
\draw[blue!10,fill=blue!10,variable=\t,domain=1:2,opacity=.5] (3,6) plot ({\t+4*\t^(-2)},{1+8*\t^(-3)}) to (3,6);
\draw[blue!10,fill=blue!10,variable=\t,domain=1:2,opacity=.5] (3,6) to (3,2) to (2,1) to (1,1) to (1,2) to (3,6);
	
\draw[variable=\t,domain=1:2] plot (\t,\t^3+1) plot ({\t+1},{\t^2+3*\t+3+2*\t^(-1)}) plot ({\t^2+2*\t^(-1)},{\t^3+5+8*\t^(-3)}) plot ({\t+2+2*\t^(-1)},{\t+3+6*\t^(-1)+4*\t^(-2)}) plot ({\t+4*\t^(-2)},{1+8*\t^(-3)}) to (2,1) to (1,1) to (1,2);

\foreach \x in {(1, 1), (1, 2), (2, 1), (2, 9), (3, 2), (3, 14), (5, 9), (5, 14)} 
{
\filldraw \x circle (2.5pt); 
}

\filldraw (2, 3) circle (3pt); 
\end{tikzpicture} 
\caption{The nine frieze points of the $G_2$ cluster algebra (drawn on the superunitary region, embedded in $\mathbb{R}^2_{>0}$)}
\label{fig:G2_superunitary_region with a point in the interior}
\end{figure}

The clusters of $\mathcal{A}$ provide the following source of frieze points.

\begin{defn}\label{defn: unitarypoint}
Given a cluster $\cluster$ in $\sA$, the \textbf{unitary point} $\mathbf{1}_\cluster\in \mathcal{A}(\mathbb{R}_{>0})$ of $\cluster$ has the following equivalent characterizations.
\begin{enumerate}
     \item $\mathbf{1}_\cluster := f_\cluster^{-1}(1,1,...,1) \in \SSS$
    \item $\mathbf{1}_\cluster$ is the unique ring homomorphism $\sA\rightarrow \mathbb{R}$ which sends each $x\in \cluster$ to $1$.
    \item For all $a\in \mathcal{A}$, $f_a ( \mathbf{1}_\cluster)$ is the sum of the coefficients of the Laurent polynomial $\ell_{a,\cluster}$.
\end{enumerate}
\end{defn}

The third characterization implies that 
every cluster variable has a positive integer value at $\mathbf{1}_\cluster$,
so each unitary point is a frieze point. 
The assignment $\cluster\mapsto \mathbf{1}_\cluster$ gives a map
\[ \{\text{clusters in $\mathcal{A}$}\} \rightarrow \mathcal{A}(\mathbb{Z}_{\geq1})\]

\begin{prop}
\label{prop:cluster maps to unitary point is injective}
This map is injective; i.e.~if $\mathbf{1}_\cluster=\mathbf{1}_\mathbf{y}$, then $\cluster=\mathbf{y}$.
\end{prop}

\noindent This was shown in the skew-symmetric case in \cite[Proposition 2.5]{GS20}.

\begin{proof}
Assume $\mathbf{1}_\cluster=\mathbf{1}_\mathbf{y}$ and let $y\in \mathbf{y}$. Then $f_y(\mathbf{1}_\cluster)=1$ and so the sum of the coefficients of $\ell_{y,\cluster}$ is $1$. 
By Lemma \ref{lemma: sumto1},  this forces $y\in\cluster$. Since this holds for all $y\in \mathbf{y}$, we see $\cluster=\mathbf{y}$.
\end{proof}

As a consequence of Proposition~\ref{prop:cluster maps to unitary point is injective}, if $\mathcal{A}$ is infinite type, then the set of frieze points $\mathcal{A}(\mathbb{Z}_{\geq1})$ is infinite. We prove the converse in Corollary \ref{coro: finite}.

\begin{ex}
In Figure \ref{fig:G2_superunitary_region with a point in the interior}, the 8 frieze points at the corners of the superunitary region of type $G_2$ are unitary points, but the frieze point in the interior is not unitary. Notably, this is the only non-unitary frieze point in a finite type cluster algebra of rank $2$.
\end{ex}

\begin{rem}
In general, the unitary points are the $0$-dimensional faces of the superunitary region; see Proposition \ref{prop:cluster face is a singleton}.
\end{rem}

\section{The faces of finite type superunitary regions}

In this section, we show that the superunitary region of a finite type cluster algebra decomposes into \emph{faces} indexed by {subclusters}, and the closure of each face is diffeomorphic to the superunitary region of the cluster algebra obtained by the deletion of that subcluster.

\subsection{Subclusters}
A \textbf{subcluster} of $\mathcal{A}$ is a set of cluster variables which are contained in some cluster. Note that any proper subcluster will be contained in multiple clusters.
The subclusters naturally form a poset under inclusion, which we call the \textbf{subcluster poset}.

\begin{ex}
The $A_2$ cluster algebra (Example \ref{ex:A2 cluster algebra}) has 11 subclusters:
\begin{itemize}
    \item The 5 clusters 
$
\{x_1,x_2\},\{x_2,x_3\},\{x_3,x_4\},\{x_4,x_5\},\{x_5,x_1\}$.
    \item The 5 singletons $\{x_1\}, \{x_2\}, \{x_3\}, \{x_4\}, \{x_5\}$.
    \item The empty set $\varnothing$.\qedhere
\end{itemize}
\end{ex}

For finite type cluster algebras, each superunitary point determines a subcluster.

\begin{prop}
\label{prop:entries 1 form a subcluster}
Let $\mathcal{A}$ be a finite type cluster algebra.
At any superunitary point $p\in \mathcal{A}(\mathbb{R}_{\geq1})$, the set of cluster variables $x$ for which $f_x(p)=1$ is a subcluster of $\mathcal{A}$.
\end{prop}
\begin{proof}
Let $p \in \SSS$, and let $\{ u_1, u_2, ..., u_m\}$ be the set of cluster variables which are $1$ at $p$. 
Theorem~\ref{thm:cluster monomials form a positive basis in finite type} tells us that 
the product $u_1 u_2 \dots u_m$ is a linear combination of 
cluster monomials 
with nonnegative integer coefficients. 
That is,
\[
u_1 u_2 \dots u_m = \sum_{\Gamma} \lambda_\Gamma x_\mathbf{\Gamma}
\]
where the sum is over all cluster monomials and the coefficients $\lambda_\Gamma$ are nonnegative integers.

We evaluate both sides of the equation at $p$.
The left-hand side is $1$ by assumption. 
Since $p$ sends every cluster variable into $\mathbb{R}_{\geq 1}$, every cluster monomial $x_\mathbf{\Gamma}$ must evaluate to at least $1$. 
This is only possible if 
all coefficients on 
the right-hand side are $0$ except for 
 a single cluster monomial $x_\mathbf{\Gamma}$ with coefficient $1$, therefore
\[
u_1 u_2 \dots u_m =  x_\mathbf{\Gamma}.
\]
Since $u_1 u_2 \dots u_m$ is a cluster monomial, the set $\{ u_1, u_2, ..., u_m\}$ is a subset of a cluster.
\end{proof}

\subsection{Subcluster faces}

Since each point in the superunitary region determines a subcluster, the superunitary region can be decomposed into subsets indexed by subclusters.

\begin{defn}
Given a subcluster $\subcluster$ in $\sA$, the \textbf{subcluster face} of $\subcluster$ is
\begin{align*}
\Subface{\subcluster} 
:= \{ 
p \in \SSS \mid
&
\text{ for all cluster variables $u$, $f_u(p)=1$ if and only if $u \in \subcluster$ }
\}
\end{align*}
\end{defn}

\begin{ex}
The superunitary region of the $A_2$ cluster algebra has 11 subcluster faces.
\begin{itemize}
    \item Each cluster corresponds to a single point at a corner of the superunitary region.
    \item Each cluster variable corresponds to the open interval on the boundary connecting the two corners (corresponding to the two clusters containing the cluster variable).
    \item The subcluster face of the empty set is the interior of the superunitary region.\qedhere
\end{itemize}
\end{ex}

The simplest subcluster faces are those corresponding to an entire cluster, in which case they are a single point that we have already encountered in a different context.

\begin{prop}
\label{prop:cluster face is a singleton}
The subcluster face $\Subface{\cluster}$ of a cluster $\cluster$ consists of a single point, which is the unitary point $\mathbf{1}_\cluster$ of $\cluster$ (see Definition \ref{defn: unitarypoint}).
\end{prop}

\begin{proof}
The unitary point $\mathbf{1}_\cluster$ is a frieze point and therefore superunitary. By Lemma~\ref{lemma: sumto1}, $f_y(\mathbf{1}_\cluster)=1$ iff $y\in \cluster$, and so $\mathbf{1}_\cluster\in \Subface{\cluster}$.
Since the homeomorphism $f_\cluster$ must send any point in $\Subface{\cluster}$ to $(1,1,...,1)$, this is the unique point in $\Subface{\cluster}$.
\end{proof}

\subsection{Local structure of superunitary regions}

The following lemma describes a neighborhood of any point in the superunitary region.

\begin{lemma}
\label{lemma:U}
Let $\mathcal{A}$ be finite type and let $p\in \mathcal{A}(\mathbb{R}_{\geq1})$. If $(x_1,x_2,...,x_m)$ is the subcluster of cluster variables which are equal to $1$ at $p$, and $\mathbf{x}:=(x_1,x_2,...,x_r)$ is any cluster containing this subcluster, 
then there is an open neighborhood $U\subset \mathcal{A}(\mathbb{R}_{>0})$ of $p$ such that 
\[ 
f_\mathbf{x}(U\cap \mathcal{A}(\mathbb{R}_{\geq 1})) 
=
f_\mathbf{x}(U)\cap (\mathbb{R}_{\geq1}^m \times \mathbb{R}_{>0}^{r-m} )
\]
\end{lemma}

\begin{proof}
Since the set of cluster variables is finite, we can define a real number $\epsilon$ by the formula
\[ 1+ \epsilon = \min f_x(p)\]
where the minimum runs over cluster variables $x$ which are not in $(x_1,x_2,...,x_m)$. Since all such cluster variables have value greater than $1$ at $p$, we see $\epsilon>0$. 
Define
\begin{align} 
\label{eq:lemma:U} U & := 
\bigcap \{ q\in \mathcal{A}(\mathbb{R}_{>0}) \mid f_x(q) >1+\epsilon/2 \} \\
\nonumber
& = \bigcap f_x^{-1} 
\left( \RR_{>1+\frac{\epsilon}{2}} \right) 
\end{align}
where the intersection runs over cluster variables $x$ which are not in $(x_1,x_2,...,x_m)$. 
The set $U$ is open in $\mathcal{A}(\mathbb{R}_{>0})$ (since each $f_x$ is continuous and since the intersection is finite) and contains $p$ because $f_x(p) \geq 1 + \epsilon > 1 + \frac{\epsilon}{2}$ for all cluster variable $x \notin (x_1, x_2, ..., x_m)$.

The containment 
$f_\mathbf{x}(U\cap \mathcal{A}(\mathbb{R}_{\geq 1})) 
\subset 
f_\mathbf{x}(U)\cap (\mathbb{R}_{\geq1}^m \times \mathbb{R}_{>0}^{r-m} )$ 
is clear 
because 
$f_\mathbf{x} (\SSS) \subset \RR_{\geq 1}^m \times \RR_{>0}^{r-m}$.
To show 
$f_\mathbf{x}(U\cap \mathcal{A}(\mathbb{R}_{\geq 1})) 
\supset 
f_\mathbf{x}(U)\cap (\mathbb{R}_{\geq1}^m \times \mathbb{R}_{>0}^{r-m} )$, 
suppose $q \in U$ and $f_{\mathbf{x}}(q) \in 
\RR_{\geq 1}^m \times \RR_{>0}^{r-m}$.
By construction of $U$, we have $f_x(q)  \geq 1$ for all $x$ not in $(x_1, x_2, ... ,x_m)$. 
Since $f_{\mathbf{x}}(q) \in 
\RR_{\geq 1}^m \times \RR_{>0}^{r-m}$, 
we have $f_{x_j}(q) \geq 1$ for each $1 \geq j \geq m$. This shows that $q \in \SSS$.
\end{proof}

\begin{rem}
\label{rem: manifoldwithcorners}
The lemma implies that, for any $p\in \mathcal{A}(\mathbb{R}_{\geq1})$, the map $q\mapsto f_\cluster(q)-f_\cluster(p)$ is a local diffeomorphism from a neighborhood of $p\in \SSS$ to a neighborhood of the origin in $\mathbb{R}_{\geq0}^m\times \mathbb{R}^{r-m}$.
Therefore, $\SSS$ is a \emph{smooth manifold with corners} in the sense of \cite{Joy12} and the subcluster faces are the boundary strata.\footnote{Technically, since we don't know the subcluster faces are connected, this only shows each subcluster face is a union of boundary strata; however, Theorem 
\ref{thmIntro:region is generalized associahedron}
will imply that each face is simply connected.}
\end{rem}

\begin{coro}
\label{coro: manifold}
If $\subcluster$ is a subcluster in a finite type cluster algebra $\mathcal{A}$ of rank $r$, then the subcluster face $\Subface{\subcluster}$ is a manifold of dimension $r-|\subcluster|$.
\end{coro}

Lemma \ref{lemma:U} tells us that the closure of any subcluster face is a union of subcluster faces, and the closure relation among faces is dual to the containment relation among subclusters.

\begin{prop}
\label{prop: union of subcluster faces}
Let $\sA$ be a finite type cluster algebra, and let $\subcluster$ be a subcluster in $\sA$. Then 
\begin{equation*}\displaystyle
\overline{\Subface{\subcluster}} = \bigsqcup_{\subcluster'\supseteq\subcluster} \Subface{\subcluster'}
\end{equation*}
where the disjoint union runs over subclusters $\subcluster'$ containing $\subcluster$. 
\end{prop}

\begin{proof}
Let $\subcluster'$ be any subcluster (not necessarily containing $\subcluster$), and let $p\in \Subface{\subcluster'}$. We show that $p$ is in the closure of $\Subface{\subcluster}$ iff $\subcluster'\supseteq \subcluster$.

If there is $x\in \subcluster \smallsetminus \subcluster'$, 
then $f_x(p)>1$ and, by continuity, $f_x>1$ on an open neighborhood of $p$ in $\mathcal{A}(\mathbb{R}_{>0})$. This open neighborhood is then disjoint from $\Subface{\subcluster}$, and so $p\not\in\overline{\Subface{\subcluster}}$.

If $\subcluster\subseteq \subcluster'$, then we may choose a cluster $(x_1,x_2,...,x_r)$ such that $\subcluster=\{x_1,x_2,...,x_n\}$ and $\subcluster'=\{x_1,x_2,...,x_m\}$ for some $m$ and $n$. By Lemma \ref{lemma:U}, we may choose an open neighborhood $U$ of $p$ in $\mathcal{A}(\mathbb{R}_{>0})$ such that
\[ 
f_{(x_1,x_2,...,x_r)}(U\cap \mathcal{A}(\mathbb{R}_{\geq 1})) 
=
f_{(x_1,x_2,...,x_r)}(U)\cap (\mathbb{R}_{\geq1}^m \times \mathbb{R}_{>0}^{r-m} )
\]
In particular, 
\begin{align*}
f_{(x_1,x_2,...,x_r)}(U\cap \Subface{\subcluster'} ) 
&=
f_{(x_1,x_2,...,x_r)}(U)\cap (\{1\}^m \times \mathbb{R}_{>0}^{r-m} )
\\
f_{(x_1,x_2,...,x_r)}(U\cap \Subface{\subcluster} ) 
&=
f_{(x_1,x_2,...,x_r)}(U)\cap (\{1\}^{n} \times \mathbb{R}_{>1}^{m-n} \times \mathbb{R}_{>0}^{r-m} )
\end{align*}
Since $\{1\}^m \times \mathbb{R}_{>0}^{r-m}$ is in the closure of $\{1\}^{n} \times \mathbb{R}_{>1}^{m-n} \times \mathbb{R}_{>0}^{r-m}$, $p$ is in the closure of $\Subface{\subcluster}$.
\end{proof}

The special case in which $\cluster=\varnothing$ is worth highlighting.

\begin{coro}\label{coro:subcluster faces are boundary of superunitary domain}
If $\mathcal{A}$ is finite type, the interior of the superunitary region $\SSS$ is $\Subface{\varnothing}= \mathcal{A}(\mathbb{R}_{>1})$ and the boundary (i.e.~the complement of the interior) is 
\[ \bigsqcup_{\subcluster\neq \varnothing} \Subface{\subcluster}\]
where the disjoint union runs over non-empty subclusters.
\end{coro}

\subsection{Deleting subclusters}

Recall from Section \ref{section: deletion} that each subcluster $\subcluster$ in a finite type cluster algebra $\mathcal{A}$ determines a \textbf{deletion map} $d:\mathcal{A}\rightarrow \mathcal{A}^\dagger$ to a cluster algebra $\mathcal{A}^\dagger$, which sends the cluster variables in $\subcluster$ to 1 and every cluster variable to a sum of cluster monomials.

Any ring homomorphism $\mathcal{A}^\dagger\rightarrow \mathbb{R}$ pulls back along $d$ to a ring homomorphism $\mathcal{A}\rightarrow \mathbb{R}$ which sends each cluster variable in $\subcluster$ to $1$, giving a pullback map $d^*:\mathcal{A}^\dagger(\mathbb{R})\rightarrow \mathcal{A}(\mathbb{R})$.

\begin{lemma}
\label{lemma: deletionhomeomorphism}
Given a subcluster $\subcluster$ in a finite type cluster algebra $\mathcal{A}$, pulling back along the deletion map $d:\mathcal{A}\rightarrow \mathcal{A}^\dagger$ gives an inclusion of superunitary regions
$ d^*: \mathcal{A}^\dagger(\mathbb{R}_{\geq1})\rightarrow \mathcal{A}(\mathbb{R}_{\geq1}) $
which restricts to a homeomorphism \[d^*:\mathcal{A}^\dagger(\mathbb{R}_{\geq1})_{\subcluster'\smallsetminus \subcluster}\xrightarrow{\sim} \mathcal{A}(\mathbb{R}_{\geq1})_{\subcluster'}\]
for each subcluster $\subcluster'$ in $\mathcal{A}$ containing $\subcluster$. 
\end{lemma}

\begin{proof}
The deletion map factors through the quotient
\[ \mathcal{A} \rightarrow \mathcal{A} / \langle x-1 \mid x\in \subcluster\rangle \xrightarrow {\sim} \mathcal{A}^\dagger \]
Therefore, pulling back along $d$ induces a bijection
\[
d^*:\{ \text{ring homomorphisms }\mathcal{A}^\dagger\rightarrow \mathbb{R} \} 
\xrightarrow{}
\{ \text{ring homomorphisms }\mathcal{A}\rightarrow \mathbb{R} \text{ sending $\subcluster$ to $1$}\}
\]

Let $y$ be a cluster variable in $\mathcal{A}$. By Proposition \ref{prop:deletion map}, $d(y)$ is a sum of cluster monomials in $\mathcal{A}^\dagger$. Therefore, a ring homomorphism $\mathcal{A}^\dagger\rightarrow \mathbb{R}$ which sends the cluster monomials into $\mathbb{R}_{\geq1}$ must pull back to a ring homomorphism $\mathcal{A}\rightarrow \mathbb{R}$ which sends $y$ into $\mathbb{R}_{\geq1}$. Therefore, $d^*$ includes $\mathcal{A}^\dagger(\mathbb{R}_{\geq1})$ into $\mathcal{A}(\mathbb{R}_{\geq1})$.

If $\subcluster'$ is a subcluster in $\mathcal{A}$ containing $\subcluster$, then $\subcluster'\smallsetminus \subcluster$ is a subcluster in $\mathcal{A}^\dagger$, 
and the subcluster face $\mathcal{A}^\dagger( \mathbb{R}_{\geq1}) _{\subcluster'\smallsetminus \subcluster}$ consists of ring homomorphisms $\mathcal{A}^\dagger\rightarrow \mathbb{R}$ which send $x\in \subcluster'\smallsetminus \subcluster$ to $1$ and all other cluster variables into $\mathbb{R}_{>1}$. Pulling back along $d$ gives ring homomorphism $\mathcal{A}\rightarrow \mathbb{R}$ which send $x\in \subcluster'$ to $1$ and all other cluster variables into $\mathbb{R}_{>1}$.
\end{proof}

The case when $\cluster=\cluster'$ gives a homeomorphism
\[\mathcal{A}^\dagger(\mathbb{R}_{\geq1})_{\varnothing} =\mathcal{A}^\dagger(\mathbb{R}_{>1}) \xrightarrow{\sim} \mathcal{A}(\mathbb{R}_{\geq1})_{\subcluster}\]
Combining Lemma \ref{lemma: deletionhomeomorphism} with Proposition \ref{prop: union of subcluster faces} gives the following.

\begin{coro}\label{coro: deletionhomeomorphism}
Given a subcluster $\subcluster$ in a finite type cluster algebra $\mathcal{A}$, pulling back along the deletion map $d:\mathcal{A}\rightarrow \mathcal{A}^\dagger$ induces homeomorphisms
\[ \mathcal{A}^\dagger(\mathbb{R}_{>1}) \xrightarrow{\sim} 
\mathcal{A}(\mathbb{R}_{\geq1})_{\subcluster}
\text{ and }
\mathcal{A}^\dagger(\mathbb{R}_{\geq1}) 
\xrightarrow{\sim} 
\bigsqcup_{\subcluster'\supseteq \subcluster} \mathcal{A}(\mathbb{R}_{\geq1})_{\subcluster'}
=
\overline{\mathcal{A}(\mathbb{R}_{\geq1})_{\subcluster}}
\]
\end{coro}


\section{Finite type superunitary regions are generalized associahedra}
\label{section: CW}

In this section, we prove that the superunitary region of a finite type cluster algebra has a face-preserving homeomorphism to the generalized associahedron. The previous section established that the faces of the superunitary region are indexed by subclusters, and the closure relations among these faces is dual to the containment relation among the subclusters. 

To complete the proof, we will need two additional facts:
\begin{itemize}
    \item Each subcluster face is homeomorphic to an open ball.
    \item These homeomorphisms make the superunitary region into a \emph{regular CW complex}.
\end{itemize}
Regular CW complexes are a type of cell complex with the virtue of being determined (up to cellular homeomorphism) by their face poset. Since the subcluster poset is known to be dual to the face poset of a polytope (the \emph{generalized associahedron}), this completes the proof.

\subsection{Recollections on regular CW complexes}
\def\DD{\mathbb{D}}
\def\oDD{\overline{\mathbb{D}}}

Let $\DD_n$ and $\oDD_n$ denote the {open ball of dimension $n$} and the {closed ball of dimension $n$}, respectively.
\begin{defn}
A \textbf{regular CW complex} is a Hausdorff space $X$ with a decomposition into subsets called \textbf{cells}, such that, for each cell $C$, the closure is a finite union of cells and is endowed with a homeomorphism $\oDD_n \xrightarrow{\sim} \overline{C}$ which restricts to a homeomorphism 
$ \DD_n\xrightarrow{\sim} C $. 
\end{defn}

\noindent For a more comprehensive overview, see \cite[Chapter III]{LW69}.

The \textbf{face poset} of a regular CW complex has an element for each cell, and one cell is `less than' another if the former is in the closure of the latter.\footnote{Note that the `face poset' consists of `cells', not `faces'. This ambiguity is standard and hopefully doesn't cause confusion.}
A \textbf{cellular homeomorphism} between regular CW complexes is a homeomorphism $X\rightarrow Y$ which restricts to a homeomorphism on each cell. A cellular homeomorphism induces an equivalence between face posets, and a well-known theorem asserts the converse is true.

\begin{thm}[{\cite[Chapter III, Theorem 1.7]{LW69}}]
\label{thm:Lundell-Weingram}
Every isomorphism between face posets of regular CW complexes can be realized by a cellular homeomorphism.
\end{thm}

\begin{proof}[Proof sketch]
A regular CW complex has a barycentric subdivision, a homeomorphic simplicial complex whose simplices are chains in the face poset of the CW complex. Since simplicial complexes can be reconstructed from their poset, so can regular CW complexes.
\end{proof}

\subsection{Generalized associahedra}

A \textbf{polytope} is a compact subset of a finite-dimensional real vector space defined by finitely many linear inequalities. 
Every polytope has the structure of a regular CW complex whose cells are the faces, and the face posets of polytopes were characterized by Bj\"orner \cite{Bjo84}.

One of the first results on the combinatorics of finite type cluster algebras was that the poset of subclusters is dual to the face poset of a polytope.

\begin{thm}[{\cite[Theorem 1.4]{CFZ02}}]
The subcluster poset of a finite type cluster algebra $\mathcal{A}$ is dual to the face poset of a polytope, called the \textbf{generalized associahedron} of $\mathcal{A}$.
\end{thm}

By Theorem \ref{thm:Lundell-Weingram}, this polytope is uniquely determined by the cluster algebra (up to cellular homeomorphism). These generalize a few families of well-studied polytopes.
\begin{itemize}
    \item The generalized associahedron of type $A_n$ is the $n$-dimensional associahedron (also called a Stasheff polytope).
    \item The generalized associahedron of type $B_n$ or $C_n$ is the $n$-dimensional cyclohedron (or Bott--Taubes polytope).
\end{itemize}
More information about generalized associahedra can be found in \cite{CFZ02} and \cite{FR07}.

\subsection{Proof of Theorem~\ref{thmIntro:region is generalized associahedron}}

Given a finite type cluster algebra, the superunitary region and the generalized associahedron both have a decomposition into faces indexed by subclusters. Since the closure relation among both kinds of faces are dual to the containment relation among subclusters, the face posets of the two spaces are equivalent. Therefore, to show that there is a cellular homeomorphism between them, it suffices to show that the subcluster faces give the superunitary region the structure of a regular CW complex.

We need a notable theorem, proven independently by Mazur and Brown \cite{Maz59,Bro60}.\footnote{An exposition and survey are found in \emph{The generalized Sch\"onflies theorem} by Putnam.}

\begin{thm}[Generalized Sch\"onflies Theorem]
\label{thm:generalized Schoenflies theorem}
Let $S^{n-1}\rightarrow S^n$ be a topological embedding, and let $C$ be a component of the complement of the image. If $C$ is a manifold, then there is a homeomorphism from the closed $n$-ball $\oDD_n\rightarrow \overline{C}$ restricting to homeomorphisms $\DD_n\xrightarrow{\sim} C$ and $\partial \DD_n\xrightarrow{\sim} S^{n-1}$.
\end{thm}

We can now prove our main theorem for finite type cluster algebras.

\begin{thmIntro}
\label{thmIntro:region is generalized associahedron}
If $\mathcal{A}$ is finite type, then the superunitary region $\mathcal{A}(\mathbb{R}_{\geq1})$ of $\mathcal{A}$ is a regular CW complex which is cellular homeomorphic to the generalized associahedron of $\mathcal{A}$.
\end{thmIntro}

\begin{proof}
We prove the theorem by induction on the rank of $\mathcal{A}$. In the base case of rank $0$, the superunitary region is a single point and theorem holds.

Assume, for induction, that the theorem holds for all finite type cluster algebras of rank less than $r$, and consider a finite type cluster algebra $\mathcal{A}$ of rank $r$.
For each non-empty subcluster $\subcluster$ in $\mathcal{A}$, the deletion of $\subcluster$ is a cluster algebra $\mathcal{A}^\dagger$ and, by Lemma~\ref{lemma: deletionhomeomorphism} and Corollary~\ref{coro: deletionhomeomorphism}, there is a homeomorphism 
\[ \mathcal{A}^\dagger(\mathbb{R}_{\geq1}) \xrightarrow{\sim} \overline{\Subface{\subcluster}} \]
which restricts to a homeomorphism $\mathcal{A}^\dagger(\mathbb{R}_{\geq1})_{\subcluster'\smallsetminus \subcluster} \xrightarrow{\sim} \Subface{\subcluster'}$ for each subcluster $\subcluster'\supseteq \subcluster$.

Since the subcluster $\subcluster$ was not empty, the deletion $\mathcal{A}^\dagger$ is a finite type cluster algebra of rank less than $r$, and so the inductive hypothesis applies; that is, $\mathcal{A}^\dagger(\mathbb{R}_{\geq1})$ admits a cellular homeomorphism to the generalized associahedron of $\mathcal{A}^\dagger$. Since polytopes are homeomorphic to closed balls, we may choose a homeomorphism
\[ 
\oDD_{r-|\subcluster|} \xrightarrow{\sim} \mathcal{A}^\dagger(\mathbb{R}_{\geq1}) \xrightarrow{\sim} \overline{\Subface{\subcluster}}
\]
which restricts to a homeomorphism
\[
\DD_{r-|\subcluster|} \xrightarrow{\sim} \mathcal{A}^\dagger(\mathbb{R}_{\geq1})_{\varnothing} \xrightarrow{\sim} \Subface{\subcluster}
\]
Running over all non-empty subclusters of $\mathcal{A}$,
these homeomorphisms define a regular CW structure on
\[ \bigsqcup_{\subcluster\neq \varnothing} \Subface{\subcluster}\]
which is the boundary of the superunitary region by Corollary~\ref{coro:subcluster faces are boundary of superunitary domain}. 
The face poset of this regular CW complex is the dual poset of non-empty subclusters in $\mathcal{A}$, which is in turn the face poset of the boundary of the generalized associahedron of $\mathcal{A}$. 
Therefore, the boundary of the superunitary region is homeomorpic to the boundary of the generalized associahedron of $\mathcal{A}$; 
in particular, it is homeomorphic to a sphere of dimension $r-1$. 

Choose a cluster $\cluster$ and consider the topological embedding $f_\cluster:\mathcal{A}(\mathbb{R}_{\geq1})\rightarrow \mathbb{R}^r$. This can be extended to a topological embedding $\mathcal{A}(\mathbb{R}_{\geq1})\rightarrow S^r$, which then restricts to a topological embedding of the boundary $\partial\mathcal{A}(\mathbb{R}_{\geq1})\rightarrow S^r$.
One of the components of the complement is the interior $\mathcal{A}(\mathbb{R}_{\geq1})_\varnothing$ of the superunitary region, which is a manifold by Corollary \ref{coro: manifold}. 
Thus, the Generalized Sch\"onflies Theorem (Theorem~\ref{thm:generalized Schoenflies theorem}) implies there is a homeomorphism $\oDD_r \xrightarrow{\sim} \mathcal{A}(\mathbb{R}_{\geq1})$ which restricts to a homeomorphism $\DD_r \xrightarrow{\sim} \mathcal{A}(\mathbb{R}_{\geq1})_\varnothing$.

This map completes the regular CW complex structure on the superunitary region $\mathcal{A}(\mathbb{R}_{\geq1})$.  
Since the face poset of this regular CW complex coincides with the face poset of the generalized associahedron of $\mathcal{A}$, there is a cellular homeomorphism between the two spaces
by Theorem~\ref{thm:Lundell-Weingram}; completing the induction.
\end{proof}

The theorem has the following immediate consequences which we will exploit later.

\begin{coro}\label{coro: finite}
If $\mathcal{A}$ is finite type, then the superunitary region $\mathcal{A}(\mathbb{R}_{\geq1})$ is compact and the set of frieze points $\mathcal{A}(\mathbb{Z}_{\geq1})$ is finite.
\end{coro}
\begin{proof}
Polytopes are compact spaces. Since the set of frieze points is a closed and discrete subset of a compact space (Proposition \ref{prop: closeddiscrete}), it is finite.
\end{proof}

\section{Finiteness of positive integral friezes of Dynkin type}
\label{section: friezes}

We now highlight an application of our main result. 
Given a Dynkin diagram $\Gamma$, a \emph{positive integral frieze of type $\Gamma$} is an assignment of positive integers to the vertices of a \emph{repetition quiver} of type $\Gamma$ such that certain \emph{mesh relations} hold.
The positive integral friezes of type $\Gamma$ are in bijection with the frieze points in the cluster algebra of type $\Gamma$, and so the finiteness of frieze points (Corollary \ref{coro: finite}) immediately implies the finiteness of positive integral friezes.  
This was conjectured in \cite{FP16} and has been proven in all but two cases (see~\cite{FP16} and~\cite[Appendix~B]{BFGST21}). For more details on friezes and their connections to cluster algebras, see the survey papers~\cite{Kel10} and \cite{Mor15}.

\subsection{Friezes from quivers}
\label{sec:friezes from quivers}

We will review the notion of friezes from quivers,  defined in~\cite{CC06,ARS10}. 
These were introduced in connection with the theory of cluster algebras, and we will explain the connection in Section~\ref{sec:friezes and cluster algebras}.

Let $Q=(Q_0,Q_1)$ be a quiver with set of vertices $Q_0$ and set of arrows $Q_1$. 
Following~\cite{Rie80}, the \textbf{repetition quiver} $\mathbb{Z}Q$ of $Q$ is constructed as follows.
\begin{itemize}
    \item The vertices of $\mathbb{Z}Q$ are pairs $(m,i)\in \mathbb{Z}\times Q_0$; note this is an infinite set.
    \item For each arrow $i\rightarrow j$ in $Q_1$, there is a family of arrows (denoted by solid arrows)
    \[ (m,i)\rightarrow (m,j) \]
    for each $m\in \mathbb{Z}$, and a family of arrows (denoted by dashed arrows)
    \[ (m,j) \rightarrow (m+1,i) \]
    for each $m\in \mathbb{Z}$.
\end{itemize}
The repetition quiver of $Q$ can be drawn as $\mathbb{Z}$-many copies of $Q$, each translated horizontally from each other, with adjacent copies of $Q$ connected by the dashed arrows; see Figure \ref{fig:D5 repetition quiver}.
Each repetition quiver admits a \textbf{translation map} 
$\tau:\mathbb{Z}Q\rightarrow \mathbb{Z}Q$ sending $(m,i)$ to $(m-1,i)$. 
Note that $Q$ embeds into $\mathbb{Z}Q$ as $\{0\}\times Q$, which we call the \textbf{initial slice} of $\mathbb{Z}Q$.

\begin{figure}[h!t]
\begin{tikzpicture}[x=0.028cm,y=-0.025cm]
    \begin{scope}[every node/.style={inner sep=4pt,font=\footnotesize}]
        \node (1_0) at (20,289) {$1$};
        \node (2_0) at (20,339) {$2$};
        \node (3_0) at (20,389) {$3$};
        \node (4_0) at (-5,439) {$4$};
        \node (5_0) at (45,489) {$5$};
    \end{scope}
    \begin{scope}[every node/.style={fill=white,font=\scriptsize},every path/.style={-{Latex[length=1.5mm,width=1mm]}}]
        \path (1_0) edge (2_0);
        \path (3_0) edge (2_0);
        \path (3_0) edge (4_0);
        \path (3_0) edge (5_0);
    \end{scope}
\end{tikzpicture}
\hspace{1cm}
\begin{tikzpicture}[x=0.028cm,y=-0.025cm]
\begin{scope}[every node/.style={inner sep=4pt,font=\footnotesize}]
\node (1_-1) at (-80,289) {$\dots$};
\node (2_-1) at (-30,339) {$(-1,2)$};
\node (3_-1) at (-80,389) {$\dots$};
\node (4_-1) at (-30,439) {$(-1, 4)$};
\node (5_-1) at (-30,489) {$(-1, 5)$};
\node (1_0) at (20,289) {$(0,1)$};
\node (2_0) at (70,339) {$(0,2)$};
\node (3_0) at (20,389) {$(0,3)$};
\node (4_0) at (70,439) {$(0, 4)$};
\node (5_0) at (70,489) {$(0, 5)$};
\node (1_1) at (120,289) {$(1,1)$};
\node (3_1) at (120,389) {$(1,3)$};
\node (2_1) at (170,339) {$(1,2)$};
\node (4_1) at (170,439) {$(1,4)$};
\node (5_1) at (170,489) {$(1,5)$};

\node (3_2) at (220,389) {$(2,3)$};
\node (1_2) at (220,289) {$(2,1)$};
\node (2_2) at (270,339) {$(2,2)$};
\node (4_2) at (270,439) {$(2,4)$};
\node (5_2) at (270,489) {$(2,5)$};

\node (1_3) at (320,289) {$\dots$};
\node (3_3) at (320,389) {$\dots$};
\end{scope}
\begin{scope}[every node/.style={fill=white,font=\scriptsize},every path/.style={-{Latex[length=1.5mm,width=1mm]}}]
\path (1_-1) edge (2_-1);
\path (2_-1) edge[densely dashed] (1_0);
\path (3_-1) edge (2_-1);
\path (3_-1) edge (4_-1);
\path (3_-1) edge (5_-1);
\path (2_-1) edge[densely dashed] (3_0);
\path (4_-1) edge[densely dashed] (3_0);
\path (5_-1) edge[densely dashed] (3_0);
\path (1_0) edge (2_0);
\path (2_0) edge[densely dashed] (1_1);
\path (3_0) edge (2_0);
\path (3_0) edge (4_0);
\path (3_0) edge (5_0);
\path (2_0) edge[densely dashed] (3_1);
\path (4_0) edge[densely dashed] (3_1);
\path (5_0) edge[densely dashed] (3_1);
\path (1_1) edge (2_1);
\path (3_1) edge (2_1);
\path (3_1) edge (4_1);
\path (3_1) edge (5_1);
\path (2_1) edge[densely dashed] (3_2);
\path (4_1) edge[densely dashed] (3_2);
\path (5_1) edge[densely dashed] (3_2);
\path (2_1) edge[densely dashed] (1_2);
\path (3_2) edge (2_2);
\path (1_2) edge (2_2);
\path (3_2) edge (4_2);
\path (3_2) edge (5_2);
\path (2_2) edge[densely dashed] (3_3);
\path (4_2) edge[densely dashed] (3_3);
\path (5_2) edge[densely dashed] (3_3);
\path (2_2) edge[densely dashed] (1_3);
\end{scope}
 \end{tikzpicture}
\caption{A quiver of type $D_5$ (left) and its repetition quiver (right)}
\label{fig:D5 repetition quiver}
\end{figure}
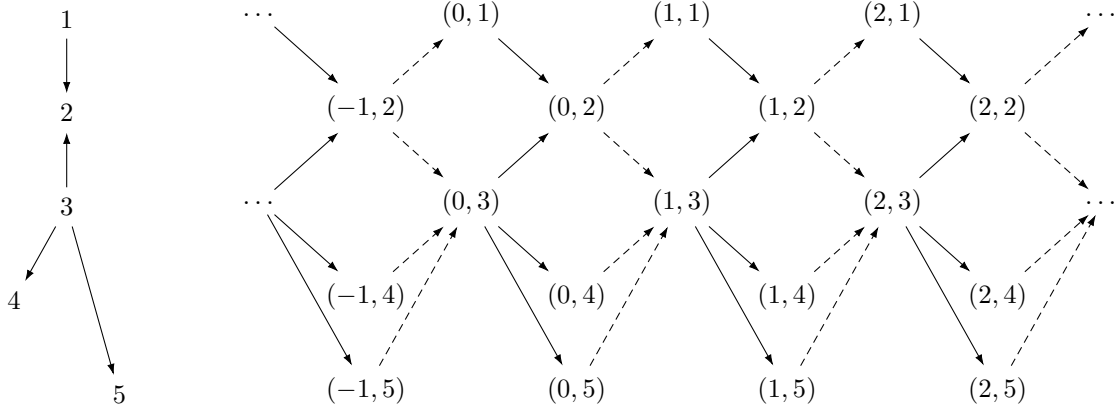

It will be useful to define friezes with values in a very general algebraic object. To that end, a \textbf{commutative semiring} is a set with two binary operations:
\begin{itemize}
    \item Addition, which must commute and associate, but need not have an identity.
    \item Multiplication, which must commute, associate, distribute over addition, and possess an identity, but need not have inverses.
\end{itemize}

Every commutative ring is a commutative semiring, as are $\mathbb{R}_{>0}$, $\mathbb{R}_{\geq1}$, and $\mathbb{Z}_{\geq1}$. The semiring $\mathbb{R}_{>0}$ is an example of a \textbf{semifield}: a commutative semiring in which every element has a multiplicative inverse.

\begin{warn}
We do not assume semirings or semifields have an additive identity (or ``zero'')! 
Moreso, since elements in a semifield must have multiplicative inverses, a non-trivial semifield will never have an additive identity. 
As a result, $\mathbb{R}$ and other fields are not semifields.
\end{warn}

\begin{defn}
Given a commutative semiring $S$ and an acyclic quiver $Q$, 
an \textbf{$S$-valued frieze of type $Q$}
is a function
\[ a:(\mathbb{Z} Q)_0\rightarrow S \]
from the vertices of the repetition quiver $\mathbb{Z}Q$ to $S$ such that the assigned values satisfy the following \textbf{mesh relation} for each vertex $v\in (\mathbb{Z}Q)_0$:
\begin{equation}
a(v) a(\tau v) = 1 + \prod_{\substack{w\rightarrow v \\ \text{in }(\mathbb{Z}Q)_1}} a(w) 
\end{equation}
\end{defn}

This mesh relation can be stated in terms of the original quiver $Q$ as a function
$a:\mathbb{Z}\times Q_0\rightarrow S $
such that, for each $(m,v)\in \mathbb{Z} \times Q_0$, 
\begin{equation}
a(m,v)a(m-1,v) = 1 + \prod_{w \to v \text{ in } Q_1} a(m,w) \prod_{v \to w \text{ in } Q_1} a(m-1,w) 
\end{equation}

Note that if $S$ has multiplicative inverses, then the mesh relation can be reformulated as
\[ a(\tau v) = a(v)^{-1}\left(1 + \prod_{\substack{w\rightarrow v \\ \text{in }(\mathbb{Z}Q)_1}} a(w) \right) \]
in which case the frieze can be constructed from a few initial values via a \emph{knitting algorithm}. 
This can be made precise as follows.

\begin{prop}\label{prop: knitting}
If $S$ is a semifield, then an $S$-valued frieze is freely determined by its values on the initial slice $\{0\}\times Q_0\subset (\mathbb{Z}Q)_0$.
\end{prop}

We are most interested in $\mathbb{Z}_{\geq1}$-valued friezes, also called \textbf{positive integral friezes}; an example is given in Figure \ref{fig:D5 frieze}.
Constructing and enumerating positive integral friezes has been an interesting problem since~\cite{ConCox73}, who constructed a bijection between positive integral friezes of type $A_n$ and triangulations of an $(n+3)$-gon. 
More enumerative results and history are given in Section \ref{section: enumeration}.

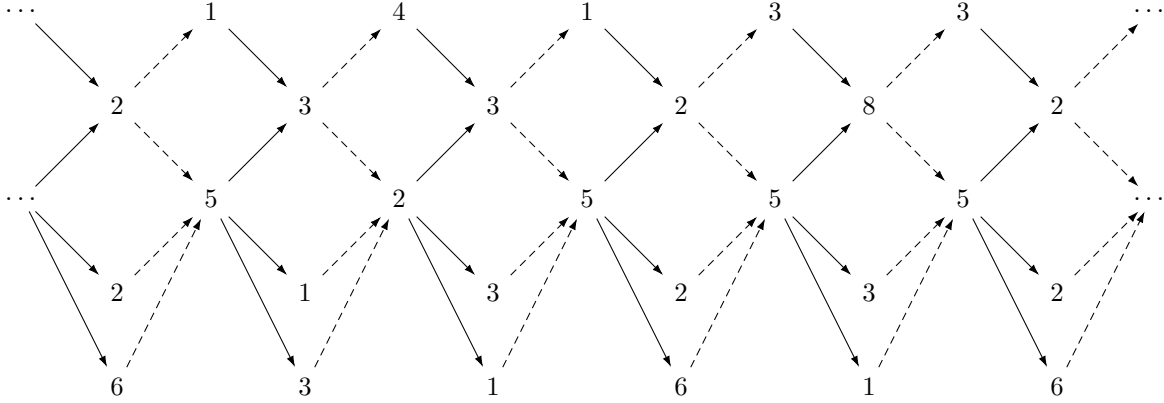
\begin{figure}[h!t]

\begin{tikzpicture}[x=0.025cm,y=-0.025cm]
\begin{scope}[every node/.style={inner sep=4pt,font=\footnotesize}]
\node (1_-1) at (-80,289) {$\dots$};
\node (2_-1) at (-30,339) {$2$};
\node (3_-1) at (-80,389) {$\dots$};
4\node (4_-1) at (-30,439) {$2$};
\node (5_-1) at (-30,489) {$6$};
\node (1_0) at (20,289) {$1$};
\node (2_0) at (70,339) {$3$};
\node (3_0) at (20,389) {$5$};
\node (4_0) at (70,439) {$1$};
\node (5_0) at (70,489) {$3$};
\node (1_1) at (120,289) {$4$};
\node (2_1) at (170,339) {$3$};
\node (3_1) at (120,389) {$2$};
\node (4_1) at (170,439) {$3$};
\node (5_1) at (170,489) {$1$};
\node (1_2) at (220,289) {$1$};
\node (2_2) at (270,339) {$2$};
\node (3_2) at (220,389) {$5$};
\node (4_2) at (270,439) {$2$};
\node (5_2) at (270,489) {$6$};
\node (1_3) at (320,289) {$3$};
\node (2_3) at (370,339) {$8$};
\node (3_3) at (320,389) {$5$};
\node (4_3) at (370,439) {$3$};
\node (5_3) at (370,489) {$1$};
\node (1_4) at (420,289) {$3$};
\node (2_4) at (470,339) {$2$};
\node (3_4) at (420,389) {$5$};
\node (4_4) at (470,439) {$2$};
\node (5_4) at (470,489) {$6$};

\node (1_5) at (520,289) {$\dots$};
\node (3_5) at (520,389) {$\dots$};
\end{scope}
\begin{scope}[every node/.style={fill=white,font=\scriptsize},every path/.style={-{Latex[length=1.5mm,width=1mm]}}]
\path (1_-1) edge (2_-1);
\path (2_-1) edge[densely dashed] (1_0);
\path (3_-1) edge (2_-1);
\path (3_-1) edge (4_-1);
\path (3_-1) edge (5_-1);
\path (2_-1) edge[densely dashed] (3_0);
\path (4_-1) edge[densely dashed] (3_0);
\path (5_-1) edge[densely dashed] (3_0);
\path (1_0) edge (2_0);
\path (2_0) edge[densely dashed] (1_1);
\path (3_0) edge (2_0);
\path (3_0) edge (4_0);
\path (3_0) edge (5_0);
\path (2_0) edge[densely dashed] (3_1);
\path (4_0) edge[densely dashed] (3_1);
\path (5_0) edge[densely dashed] (3_1);
\path (1_1) edge (2_1);
\path (3_1) edge (2_1);
\path (3_1) edge (4_1);
\path (3_1) edge (5_1);
\path (2_1) edge[densely dashed] (3_2);
\path (4_1) edge[densely dashed] (3_2);
\path (5_1) edge[densely dashed] (3_2);
\path (2_1) edge[densely dashed] (1_2);
\path (3_2) edge (2_2);
\path (1_2) edge (2_2);
\path (3_2) edge (4_2);
\path (3_2) edge (5_2);
\path (2_2) edge[densely dashed] (3_3);
\path (4_2) edge[densely dashed] (3_3);
\path (5_2) edge[densely dashed] (3_3);
\path (2_2) edge[densely dashed] (1_3);
\path (3_3) edge (2_3);
\path (1_3) edge (2_3);
\path (3_3) edge (4_3);
\path (3_3) edge (5_3);
\path (2_3) edge[densely dashed] (3_4);
\path (4_3) edge[densely dashed] (3_4);
\path (5_3) edge[densely dashed] (3_4);
\path (2_3) edge[densely dashed] (1_4);
\path (3_4) edge (2_4);
\path (1_4) edge (2_4);
\path (3_4) edge (4_4);
\path (3_4) edge (5_4);
\path (2_4) edge[densely dashed] (3_5);
\path (4_4) edge[densely dashed] (3_5);
\path (5_4) edge[densely dashed] (3_5);
\path (2_4) edge[densely dashed] (1_5);
\end{scope}
 \end{tikzpicture}
\caption{A positive integral frieze of type $D_5$}
\label{fig:D5 frieze}
\end{figure}

\begin{rem}
If $Q$ and $Q'$ are two orientations of the same tree $\Gamma$, then the repetition quivers $\mathbb{Z}Q$ and $\mathbb{Z}Q'$ are isomorphic and friezes of type $Q$ and $Q'$ are equivalent, so we may unambiguously refer to a \textbf{frieze of type $\Gamma$}. Notably, we use ADE notation to denote friezes on the repetition quiver of any orientation of the corresponding simply-laced Dynkin diagram.
\end{rem}

\begin{rem}
Working with commutative semirings combines many cases into a single framework; most of which we don't consider in this note. For example, $\mathbb{Z}$ and $\mathbb{R}$ may be 
made into
\emph{tropical semifields} in which `addition' is the maximum of two numbers and `multiplication' is the sum of two numbers. Friezes valued in a tropical semifield are called \emph{additive friezes}, 
and they are related to dimension vectors of indecomposable quiver representations \cite{Gab80}.
\end{rem}

\subsection{Friezes from valued quivers}

The constructions in Section~\ref{sec:friezes from quivers}
may be extended to \emph{valued quivers}, which allows the theory to encompass non-simply-laced Dynkin diagrams.

A \textbf{valued quiver} is a quiver $Q$ together with a pair of positive integral \text{values} $(b,c)$ for each arrow in $Q$. 
A quiver may be regarded as a valued quiver in which every arrow has values $(1,1)$, and when depicting a valued quiver, we suppress values of the form $(1,1)$.

\begin{figure}[htb]
\begin{tikzpicture}[x=0.025cm,y=-0.025cm]
\begin{scope}[every node/.style={inner sep=4pt,font=\footnotesize}]
    \node (1_0) at (20,289) {$1$};
    \node (2_0) at (70,339) {$2$};
    \node (3_0) at (20,389) {$3$};
    \node (5_0) at (70,439) {$4$};
\end{scope}
\begin{scope}[every node/.style={font=\scriptsize},every path/.style={-{Latex[length=1.5mm,width=1mm]}}]
\path (1_0) edge (2_0);
\path (3_0) edge (2_0);
\path (3_0) edge node[pos=0.5,rotate=-45,above]{\tiny (1,2)} (5_0);
\end{scope}
\end{tikzpicture}
\hspace{1cm}\begin{tikzpicture}[x=0.025cm,y=-0.025cm]
\begin{scope}[every node/.style={inner sep=4pt,font=\footnotesize}]
    \node (1_-1) at (-80,289) {$\dots$};
    \node (2_-1) at (-30,339) {$(-1,2)$};
    \node (3_-1) at (-80,389) {$\dots$};
    \node (5_-1) at (-30,439) {$(-1,4)$};

    \node (1_0) at (20,289) {$(0,1)$};
    \node (2_0) at (70,339) {$(0,2)$};
    \node (3_0) at (20,389) {$(0,3)$};
    \node (5_0) at (70,439) {$(0, 4)$};
    
    \node (1_1) at (120,289) {$(1,1)$};
    \node (2_1) at (170,339) {$(1,2)$};
    \node (3_1) at (120,389) {$(1,3)$};
    \node (5_1) at (170,439) {$(1,4)$};
    
    \node (1_2) at (220,289) {$(2,1)$};
    \node (2_2) at (270,339) {$(2,2)$};
    \node (3_2) at (220,389) {$(2,3)$};
    \node (5_2) at (270,439) {$(2,4)$};
    
    \node (1_3) at (320,289) {$\dots$};
    \node (3_3) at (320,389) {$\dots$};
\end{scope}
\begin{scope}[every node/.style={font=\scriptsize},every path/.style={-{Latex[length=1.5mm,width=1mm]}}]
    \draw (1_-1) to (2_-1);
    \draw (3_-1) to (2_-1);
    \draw (3_-1) to node[pos=0.5,rotate=-45,above]{\tiny (1,2)} (5_-1);
    \draw[densely dashed] (2_-1) to (1_0);
    \draw[densely dashed] (2_-1) to (3_0);
    \draw[densely dashed] (5_-1) to node[pos=0.5,rotate=45,below]{\tiny (2,1)} (3_0);
    
    \path (1_0) edge (2_0);
    \path (2_0) edge[densely dashed] (1_1);
    \path (3_0) edge (2_0);
    \draw (3_0) -- (5_0) node[pos=0.5,rotate=-45,above]{\tiny (1,2)};
    \path (2_0) edge[densely dashed] (3_1);
    \draw[densely dashed] (5_0) -- (3_1) node[pos=0.5,rotate=45,below]{\tiny (2,1)};
    \path (1_1) edge (2_1);
    \path (3_1) edge (2_1);
    \draw (3_1) -- (5_1) node[pos=0.5,rotate=-45,above]{\tiny (1,2)};
    \path (2_1) edge[densely dashed] (3_2);
    \draw[densely dashed] (5_1) -- (3_2) node[pos=0.5,rotate=45,below]{\tiny (2,1)};
    \path (2_1) edge[densely dashed] (1_2);
    \path (3_2) edge (2_2);
    \path (1_2) edge (2_2);
    \draw (3_2) -- (5_2) node[pos=0.5,rotate=-45,above]{\tiny (1,2)};

    \draw[densely dashed] (2_2) to (1_3);
    \draw[densely dashed] (2_2) to (3_3);
    \draw[densely dashed] (5_2) to node[pos=0.5,rotate=45,below]{\tiny (2,1)} (3_3);
\end{scope}
 \end{tikzpicture}
\caption{A valued quiver of type $B_4$ (left) and its repetition quiver (right)}
\label{fig:B4 repetition}
\end{figure}
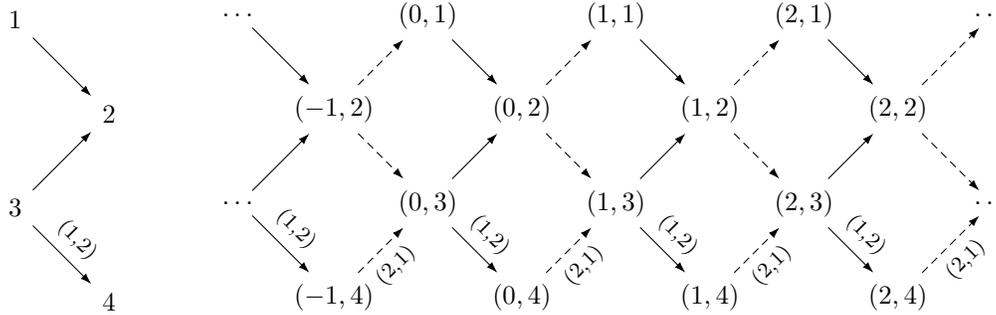

The repetition quiver $\mathbb{Z}Q$ of a valued quiver $Q$ is a valued quiver constructed as follows. 
\begin{itemize}
    \item As before, the vertices of $\mathbb{Z}Q$ are pairs $(m,i)\in \mathbb{Z}\times Q_0$.
    \item For each arrow $i\rightarrow j$ in $Q$ with values $(b,c)$, there is a family of arrows 
    \[ (m,i)\rightarrow (m,j) \]
    with values $(b,c)$ 
    for each $m\in \mathbb{Z}$ (denoted by solid arrows), and a family of arrows 
    \[ (m,j) \rightarrow (m+1,i) \]
    with values $(c,b)$
    for each $m\in \mathbb{Z}$ (denoted by dashed arrows).
\end{itemize}

\begin{ex}
A valued quiver of type $B_4$ and its repetition quiver are given in Figure~\ref{fig:B4 repetition}.
\end{ex}

\begin{defn}
Given a commutative semiring $S$ and a valued acyclic quiver $Q$,   
an \textbf{$S$-valued frieze of type $Q$} 
is a function
\[ a:\mathbb{Z}\times Q_0\rightarrow S \]
from the vertices of the repetition quiver $\mathbb{Z}Q$ to the semiring $S$, such that the assigned values satisfy the \textbf{mesh relation} for each $(m,i)\in \mathbb{Z}\times Q_0$:
\begin{equation}
a(v) a(\tau v) = 1 + \prod_{\substack{w\rightarrow v \text{ in }(\mathbb{Z}Q)_1 \\ \text{with values $(b,c)$}}} a(w)^{b}
\end{equation}
\end{defn}

As before, this can be stated in terms of the original quiver $Q$ as a function
$a:\mathbb{Z}\times Q_0\rightarrow S $
such that, for each $(m,i)\in \mathbb{Z} \times Q_0$, 
\begin{equation}
a(m,v) a(m-1, v) = 1 + 
\prod_{\substack{w\rightarrow v \text{ in }Q_1 \\ \text{with values $(b,c)$}}} a(m,w)^{b}
\prod_{\substack{v\rightarrow w \text{ in }Q_1 \\ \text{with values $(b,c)$}}} a(m-1,w)^{c}
\end{equation}

An example of such a frieze is given in Figure~\ref{fig:B4 frieze}.

\begin{figure}[htb]
\begin{tikzpicture}[x=0.025cm,y=-0.025cm]
\begin{scope}[every node/.style={inner sep=4pt,font=\footnotesize}]
    \node (1_-1) at (-80,289) {$\dots$};
    \node (2_-1) at (-30,339) {$17$};
    \node (3_-1) at (-80,389) {$\dots$};
    \node (5_-1) at (-30,439) {$3$};

    \node (1_0) at (20,289) {$3$};
    \node (2_0) at (70,339) {$2$};
    \node (3_0) at (20,389) {$11$};
    \node (5_0) at (70,439) {$4$};
    
    \node (1_1) at (120,289) {$1$};
    \node (2_1) at (170,339) {$2$};
    \node (3_1) at (120,389) {$3$};
    \node (5_1) at (170,439) {$1$};
    
    \node (1_2) at (220,289) {$3$};
    \node (2_2) at (270,339) {$2$};
    \node (3_2) at (220,389) {$1$};
    \node (5_2) at (270,439) {$2$};
    
    \node (1_3) at (320,289) {$1$};
    \node (2_3) at (370,339) {$5$};
    \node (3_3) at (320,389) {$9$};
    \node (5_3) at (370,439) {$5$};
       
    \node (1_4) at (420,289) {$6$};
    \node (2_4) at (470,339) {$17$};
    \node (3_4) at (420,389) {$14$};
    \node (5_4) at (470,439) {$3$};
    
    \node (1_5) at (520,289) {$\dots$};
    \node (3_5) at (520,389) {$\dots$};
\end{scope}
\begin{scope}[every node/.style={font=\scriptsize},every path/.style={-{Latex[length=1.5mm,width=1mm]}}]
    \draw (1_-1) to (2_-1);
    \draw (3_-1) to (2_-1);
    \draw (3_-1) to node[pos=0.5,rotate=-45,above]{\tiny (1,2)} (5_-1);
    \draw[densely dashed] (2_-1) to (1_0);
    \draw[densely dashed] (2_-1) to (3_0);
    \draw[densely dashed] (5_-1) to node[pos=0.5,rotate=45,below]{\tiny (2,1)} (3_0);
    
    \path (1_0) edge (2_0);
    \path (2_0) edge[densely dashed] (1_1);
    \path (3_0) edge (2_0);
    \draw (3_0) -- (5_0) node[pos=0.5,rotate=-45,above]{\tiny (1,2)};
    \path (2_0) edge[densely dashed] (3_1);
    \draw[densely dashed] (5_0) -- (3_1) node[pos=0.5,rotate=45,below]{\tiny (2,1)};
    \path (1_1) edge (2_1);
    \path (3_1) edge (2_1);
    \draw (3_1) -- (5_1) node[pos=0.5,rotate=-45,above]{\tiny (1,2)};
    \path (2_1) edge[densely dashed] (3_2);
    \draw[densely dashed] (5_1) -- (3_2) node[pos=0.5,rotate=45,below]{\tiny (2,1)};
    \path (2_1) edge[densely dashed] (1_2);
    \path (3_2) edge (2_2);
    \path (1_2) edge (2_2);
    \draw (3_2) -- (5_2) node[pos=0.5,rotate=-45,above]{\tiny (1,2)};
    \path (2_2) edge[densely dashed] (3_3);
    \draw[densely dashed] (5_2) -- (3_3) node[pos=0.5,rotate=45,below]{\tiny (2,1)};
    \path (2_2) edge[densely dashed] (1_3);
    \path (3_3) edge (2_3);
    \path (1_3) edge (2_3);
    \draw (3_3) -- (5_3) node[pos=0.5,rotate=-45,above]{\tiny (1,2)};
    
    \path (2_3) edge[densely dashed] (3_4);
    \draw[densely dashed] (5_3) -- (3_4) node[pos=0.5,rotate=45,below]{\tiny (2,1)};
    \path (2_3) edge[densely dashed] (1_4);
    \path (3_4) edge (2_4);
    \path (1_4) edge (2_4);
    \draw (3_4) -- (5_4) node[pos=0.5,rotate=-45,above]{\tiny (1,2)};

    \draw[densely dashed] (2_4) to (1_5);
    \draw[densely dashed] (2_4) to (3_5);
    \draw[densely dashed] (5_4) to node[pos=0.5,rotate=45,below]{\tiny (2,1)} (3_5);
\end{scope}
 \end{tikzpicture}
 \caption{A positive integral frieze of type $B_4$}
 \label{fig:B4 frieze}

\end{figure}
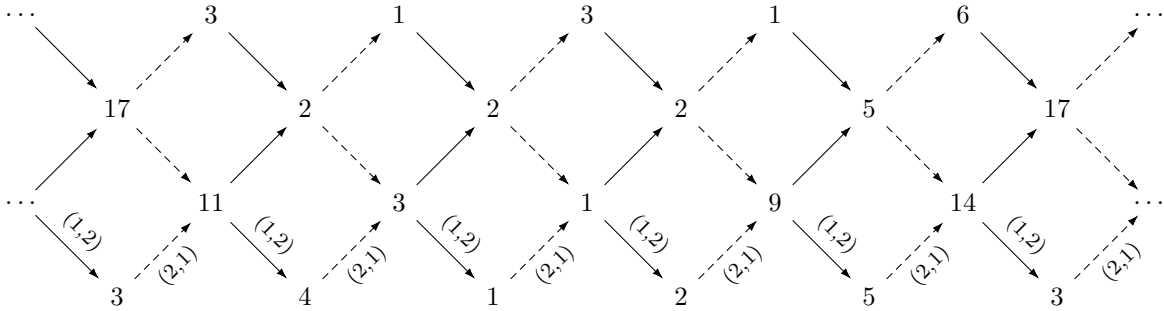

Friezes of valued quivers have many of the same properties of friezes of quivers; e.g.~the analog of Proposition~\ref{prop: knitting} holds. 

\begin{rem}
An orientation of a Dynkin diagram with Cartan matrix $C$ gives a valued quiver in which the arrow $i\rightarrow j$ has value $(|C_{ji}|,|C_{ij}|)$. As before, friezes only depend on the underlying valued graph of $Q$; therefore, we may associate friezes to Dynkin diagrams.
\end{rem}

\subsection{Friezes and cluster algebras}
\label{sec:friezes and cluster algebras}

Recall from Section \ref{section: recollections} that a valued quiver is \textbf{skew-symmetrizable} if it comes from a skew-symmetrizable matrix, and such a valued quiver determines a cluster algebra $\mathcal{A}$. Note that a quiver (regarded as a valued quiver with all values $(1,1)$) is automatically skew-symmetrizable.

\begin{thm}
If $Q$ is a (skew-symmetrizable valued) acyclic quiver with cluster algebra $\mathcal{A}$, 
then there exists a unique $\mathcal{A}$-valued frieze of type $Q$ 
sending the initial slice $\{0\}\times Q$ of $\mathbb{Z}Q$ to the initial cluster $\cluster$. This frieze sends all vertices of $\mathbb{Z}Q$ to cluster variables, and if $Q$ is Dynkin or has two vertices, then every cluster variable in $\mathcal{A}$ appears on some vertex of $\mathbb{Z}Q$.
\end{thm}

\noindent We refer to 
this $\mathcal{A}$-valued frieze 
as the \textbf{general frieze of type $Q$}.

\begin{proof}[Proof sketch]
All parts of this result have appeared in earlier works, but we sketch the proof.

The general frieze may be constructed by applying the \emph{knitting algorithm} (see \cite{Kel10} for details and examples) starting with the initial values of $(0,v)=x_v$. Since every acyclic quiver has a sink and a source, each iteration of knitting/mutating produces a new cluster variable, which is placed on a vertex in the repetition quiver. Iterating this in both directions yields the general frieze. 

If $Q$ is Dynkin, the fact that every cluster variable appears in the general frieze follows from 
\cite[Theorem 1.9]{CA2}, 
which parametrizes non-initial cluster variables by the positive roots in the associated root system. In the simply-laced cased, a deeper explanation is given by the cluster character formula \cite{CC06}, which embeds 
the Auslander--Reiten quiver of $Q$ into the general frieze.

If $Q$ has two vertices, the cluster algebra was described in Example \ref{ex: rank2} and the general frieze is given in Figure \ref{fig: bcuni}. Comparing the two, all cluster variables appear.
\end{proof}

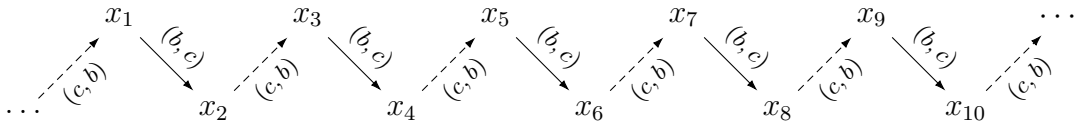
\begin{figure}[h!t]
\begin{tikzpicture}[x=0.025cm,y=-0.025cm]
\begin{scope}[every node/.style={inner sep=4pt}]
    \node (2_-1) at (-30,339) {$\dots$};
    \node (1_0) at (20,289) {$x_1$};
    \node (2_0) at (70,339) {$x_2$};
    \node (1_1) at (120,289) {$x_3$};
    \node (2_1) at (170,339) {$x_4$};
    \node (1_2) at (220,289) {$x_5$};
    \node (2_2) at (270,339) {$x_6$};
    \node (1_3) at (320,289) {$x_7$};
    \node (2_3) at (370,339) {$x_8$};
    \node (1_4) at (420,289) {$x_9$};
    \node (2_4) at (470,339) {$x_{10}$};
    \node (1_5) at (520,289) {$\cdots$};
\end{scope}
\begin{scope}[every node/.style={font=\scriptsize},every path/.style={-{Latex[length=1.5mm,width=1mm]}}]
    \draw[densely dashed] (2_-1) to node[below,rotate=45] {$(c,b)$} (1_0);
    \draw (1_0) to node[above,rotate=-45] {$(b,c)$} (2_0);
    \draw[densely dashed] (2_0) to node[below,rotate=45] {$(c,b)$} (1_1);
    \draw (1_1) to node[above,rotate=-45] {$(b,c)$} (2_1);
    \draw[densely dashed] (2_1) to node[below,rotate=45] {$(c,b)$} (1_2);
    \draw (1_2) to node[above,rotate=-45] {$(b,c)$} (2_2);
    \draw[densely dashed] (2_2) to node[below,rotate=45] {$(c,b)$} (1_3);
    \draw (1_3) to node[above,rotate=-45] {$(b,c)$} (2_3);
    \draw[densely dashed] (2_3) to node[below,rotate=45] {$(c,b)$} (1_4);
    \draw (1_4) to node[above,rotate=-45] {$(b,c)$} (2_4);
    \draw[densely dashed] (2_4) to node[below,rotate=45] {$(c,b)$} (1_5);
\end{scope}
 \end{tikzpicture}
\caption{The form of a general frieze with two vertices}
\label{fig: bcuni}
\end{figure}

\begin{ex}
Figure \ref{fig: bcuni} shows the general frieze for an arbitrary quiver with two vertices, with cluster variables indexed by $\mathbb{Z}$ as in Example \ref{ex: rank2}. This general frieze becomes more interesting when written as Laurent polynomials in the initial cluster $\{x_1,x_2\}$; three general friezes of Dynkin type are given in Figure \ref{fig: ABGuni}.
\end{ex}

\begin{figure}[h!t]
\begin{tikzpicture}[x=0.025cm,y=-0.025cm]
\begin{scope}[every node/.style={inner sep=4pt,font=\footnotesize}]
    \node (2_-1) at (-30,339) {$\dots$};
    \node (1_0) at (20,289) {$x_1$};
    \node (2_0) at (70,339) {$x_2$};
    \node (1_1) at (120,289) {$\frac{1+x_2}{x_1}$};
    \node (2_1) at (170,339) {$\frac{x_1+1+x_2}{x_1x_2}$};
    \node (1_2) at (220,289) {$\frac{x_1+1}{x_2}$};
    \node (2_2) at (270,339) {$x_1$};
    \node (1_3) at (320,289) {$x_2$};
    \node (2_3) at (370,339) {$\frac{1+x_2}{x_1}$};
    \node (1_4) at (420,289) {$\frac{x_1+1+x_2}{x_1x_2}$};
    \node (2_4) at (470,339) {$\frac{x_1+1}{x_2}$};
    \node (1_5) at (520,289) {$\dots$};
\end{scope}
\begin{scope}[every node/.style={fill=white,font=\scriptsize},every path/.style={-{Latex[length=1.5mm,width=1mm]}}]
    \draw[densely dashed] (2_-1) to (1_0);
    \draw (1_0) to (2_0);
    \draw[densely dashed] (2_0) to (1_1);
    \draw (1_1) to (2_1);
    \draw[densely dashed] (2_1) to (1_2);
    \draw (1_2) to (2_2);
    \draw[densely dashed] (2_2) to (1_3);
    \draw (1_3) to (2_3);
    \draw[densely dashed] (2_3) to (1_4);
    \draw (1_4) to (2_4);
    \draw[densely dashed] (2_4) to (1_5);
\end{scope}
 \end{tikzpicture}

 \begin{tikzpicture}[x=0.025cm,y=-0.025cm]
\begin{scope}[every node/.style={inner sep=4pt,font=\footnotesize}]
    \node (2_-1) at (-30,389) {$\dots$};
    \node (1_0) at (20,289) {$x_1$};
    \node (2_0) at (70,389) {$x_2$};
    \node (1_1) at (120,289) {$\frac{x_2^2+1}{x_1}$};
    \node (2_1) at (170,389) {$\frac{x_1+x_2^2+1}{x_1x_2}$};
    \node (1_2) at (220,289) {$\frac{x_1^2+x_2^2+2x_1+1}{x_1x_2^2}$};
    \node (2_2) at (270,389) {$\frac{x_1+1}{x_2}$};
    \node (1_3) at (320,289) {$x_1$};
    \node (2_3) at (370,389) {$x_2$};
    \node (1_4) at (420,289) {$\frac{x_2^2+1}{x_1}$};
    \node (2_4) at (470,389) {$\frac{x_1+x_2^2+1}{x_1x_2}$};
    \node (1_5) at (520,289) {$\dots$};
\end{scope}
\begin{scope}[every node/.style={fill=white,font=\scriptsize},every path/.style={-{Latex[length=1.5mm,width=1mm]}}]
    \draw[densely dashed] (2_-1) to node[below,rotate=63] {$(1,2)$} (1_0);
    \draw (1_0) to node[above,rotate=-63] {$(1,2)$} (2_0);
    \draw[densely dashed] (2_0) to node[below,rotate=63] {$(2,1)$} (1_1);
    \draw (1_1) to node[above,rotate=-63] {$(1,2)$} (2_1);
    \draw[densely dashed] (2_1) to node[below,rotate=63] {$(2,1)$} (1_2);
    \draw (1_2) to node[above,rotate=-63] {$(1,2)$} (2_2);
    \draw[densely dashed] (2_2) to node[below,rotate=63] {$(2,1)$} (1_3);
    \draw (1_3) to node[above,rotate=-63] {$(1,2)$} (2_3);
    \draw[densely dashed] (2_3) to node[below,rotate=63] {$(2,1)$} (1_4);
    \draw (1_4) to node[above,rotate=-63] {$(1,2)$} (2_4);
    \draw[densely dashed] (2_4) to node[below,rotate=63] {$(2,1)$} (1_5);
\end{scope}
\end{tikzpicture}

 \begin{tikzpicture}[x=0.025cm,y=-0.025cm]
\begin{scope}[every node/.style={inner sep=4pt,font=\footnotesize}]
    \node (2_-1) at (-30,389) {$\dots$};
    \node (1_0) at (20,289) {$x_1$};
    \node (2_0) at (70,389) {$x_2$};
    \node (1_1) at (120,289) {$\frac{x_2^3+1}{x_1}$};
    \node (2_1) at (170,389) {$\frac{x_1+x_2^3+1}{x_1x_2}$};
    \node (1_2) at (220,289) {};
    \node (2_2) at (270,389) {$\frac{x_2^3+x_1^2+2x_1+1}{x_1x_2^2}$};
    \node (1_3) at (320,289) {$\frac{x_1^3+x_2^3+3x_1^2 + 3x_1+1}{x_1x_2^3}$};
    \node (2_3) at (370,389) {$\frac{x_1+1}{x_2}$};
    \node (1_4) at (420,289) {$x_1$};
    \node (2_4) at (470,389) {$x_2$};
    \node (1_5) at (520,289) {$\dots$};
    
    \node[inner sep=0] (label) at (220,259) {$\underbrace{\scriptstyle \frac{x_2^6+3 x_1 x_2^3+2 x_2^3+x_1^3+3 x_1^2+3 x_1+1}{x_1^2x_2^3}}$};
    \draw[thick] (1_2.center) to (label.270);
\end{scope}
\begin{scope}[every node/.style={font=\scriptsize},every path/.style={-{Latex[length=1.5mm,width=1mm]}}]
    \draw[densely dashed] (2_-1) to node[below,rotate=61] {$(3,1)$} (1_0);
    \draw (1_0) to node[above,rotate=-63] {$(1,3)$} (2_0);
    \draw[densely dashed] (2_0) to node[below,rotate=61] {$(3,1)$} (1_1);
    \draw (1_1) to node[above,rotate=-63] {$(1,3)$} (2_1);
    \draw[densely dashed] (2_1) to node[below,rotate=61] {$(3,1)$} (1_2);
    \draw (1_2) to node[above,rotate=-63] {$(1,3)$} (2_2);
    \draw[densely dashed] (2_2) to node[below,rotate=61] {$(3,1)$} (1_3);
    \draw (1_3) to node[above,rotate=-63] {$(1,3)$} (2_3);
    \draw[densely dashed] (2_3) to node[below,rotate=61] {$(3,1)$} (1_4);
    \draw (1_4) to node[above,rotate=-63] {$(1,3)$} (2_4);
    \draw[densely dashed] (2_4) to node[below,rotate=61] {$(3,1)$} (1_5);
\end{scope}
\end{tikzpicture}
\caption{The general friezes of types $A_2$, $B_2/C_2$, and $G_2$ (as Laurent polynomials)}
\label{fig: ABGuni}
\end{figure}
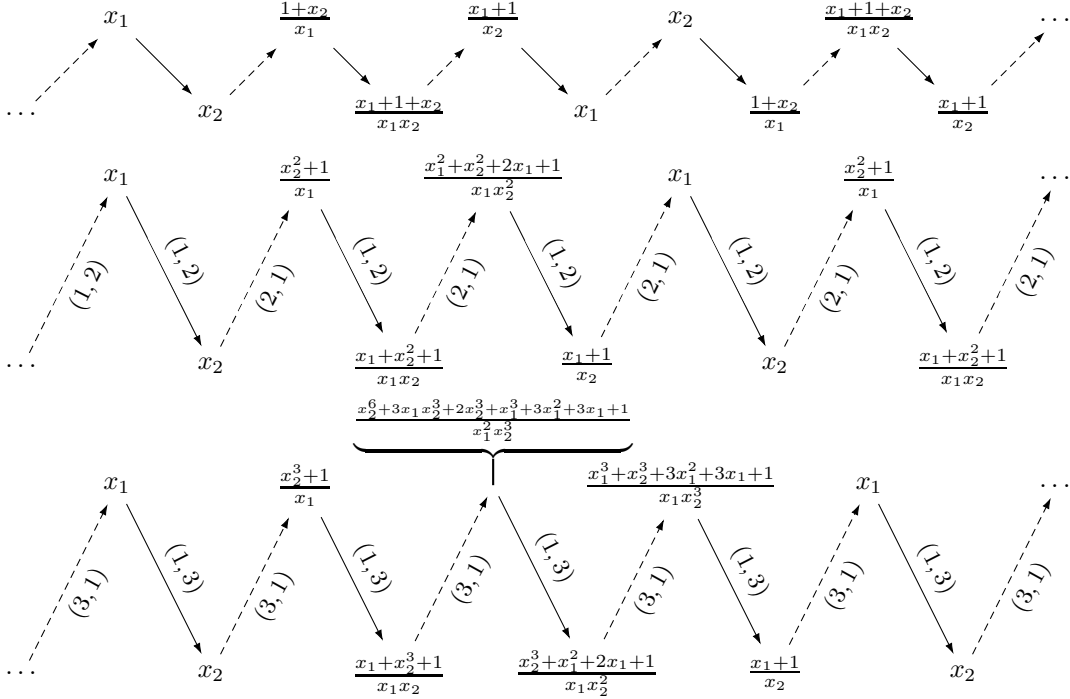

\subsection{Specializing the general frieze}

We can use the general frieze of type $Q$ to construct friezes valued in real semirings, as follows. 
Given a ring homomorphism $p:\mathcal{A}\rightarrow \mathbb{R}$, we may apply $p$ to the vertices of the general frieze. Since a ring homomorphism preserves the mesh relations, the result is an $\mathbb{R}$-valued frieze; let us call this the \textbf{specialization of the general frieze at $p$}.
If $p\in \mathcal{A}(S)$ for a subsemiring $S\subset \mathbb{R}$ (that is, $p$ sends cluster variables into $S$), then this specialization is an $S$-valued frieze of type $Q$.

\begin{ex}
There are nine frieze points for the $G_2$ cluster algebra (Example \ref{ex: G2friezepoints}). Consider the three frieze points which, as ring homomorphisms, send $(x_1,x_2)$ to $(1,1)$, $(1,2)$, and $(3,2)$, 
respectively. 
Applying each of these ring homomorphisms to the general frieze of type $G_2$ (Figure \ref{fig: ABGuni}) gives a positive integral frieze of type $G_2$ (Figure \ref{fig: G2frieze}).
\end{ex}

\begin{figure}[h!!t]
\begin{tikzpicture}[x=0.025cm,y=-0.025cm]
\begin{scope}[every node/.style={inner sep=4pt,font=\footnotesize}]
    \node (2_-1) at (-30,389) {$\dots$};
    \node (1_0) at (20,289) {$1$};
    \node (2_0) at (70,389) {$1$};
    \node (1_1) at (120,289) {$2$};
    \node (2_1) at (170,389) {$3$};
    \node (1_2) at (220,289) {$14$};
    \node (2_2) at (270,389) {$5$};
    \node (1_3) at (320,289) {$9$};
    \node (2_3) at (370,389) {$2$};
    \node (1_4) at (420,289) {$1$};
    \node (2_4) at (470,389) {$1$};
    \node (1_5) at (520,289) {$\dots$};
\end{scope}
\begin{scope}[every node/.style={font=\scriptsize},every path/.style={-{Latex[length=1.5mm,width=1mm]}}]
    \draw[densely dashed] (2_-1) to node[below,rotate=61] {$(3,1)$} (1_0);
    \draw (1_0) to node[above,rotate=-63] {$(1,3)$} (2_0);
    \draw[densely dashed] (2_0) to node[below,rotate=61] {$(3,1)$} (1_1);
    \draw (1_1) to node[above,rotate=-63] {$(1,3)$} (2_1);
    \draw[densely dashed] (2_1) to node[below,rotate=61] {$(3,1)$} (1_2);
    \draw (1_2) to node[above,rotate=-63] {$(1,3)$} (2_2);
    \draw[densely dashed] (2_2) to node[below,rotate=61] {$(3,1)$} (1_3);
    \draw (1_3) to node[above,rotate=-63] {$(1,3)$} (2_3);
    \draw[densely dashed] (2_3) to node[below,rotate=61] {$(3,1)$} (1_4);
    \draw (1_4) to node[above,rotate=-63] {$(1,3)$} (2_4);
    \draw[densely dashed] (2_4) to node[below,rotate=61] {$(3,1)$} (1_5);
\end{scope}
\end{tikzpicture}

 \begin{tikzpicture}[x=0.025cm,y=-0.025cm]
\begin{scope}[every node/.style={inner sep=4pt,font=\footnotesize}]
    \node (2_-1) at (-30,389) {$\dots$};
    \node (1_0) at (20,289) {$1$};
    \node (2_0) at (70,389) {$2$};
    \node (1_1) at (120,289) {$9$};
    \node (2_1) at (170,389) {$5$};
    \node (1_2) at (220,289) {$14$};
    \node (2_2) at (270,389) {$3$};
    \node (1_3) at (320,289) {$2$};
    \node (2_3) at (370,389) {$1$};
    \node (1_4) at (420,289) {$1$};
    \node (2_4) at (470,389) {$2$};
    \node (1_5) at (520,289) {$\dots$};
\end{scope}
\begin{scope}[every node/.style={font=\scriptsize},every path/.style={-{Latex[length=1.5mm,width=1mm]}}]
    \draw[densely dashed] (2_-1) to node[below,rotate=61] {$(3,1)$} (1_0);
    \draw (1_0) to node[above,rotate=-63] {$(1,3)$} (2_0);
    \draw[densely dashed] (2_0) to node[below,rotate=61] {$(3,1)$} (1_1);
    \draw (1_1) to node[above,rotate=-63] {$(1,3)$} (2_1);
    \draw[densely dashed] (2_1) to node[below,rotate=61] {$(3,1)$} (1_2);
    \draw (1_2) to node[above,rotate=-63] {$(1,3)$} (2_2);
    \draw[densely dashed] (2_2) to node[below,rotate=61] {$(3,1)$} (1_3);
    \draw (1_3) to node[above,rotate=-63] {$(1,3)$} (2_3);
    \draw[densely dashed] (2_3) to node[below,rotate=61] {$(3,1)$} (1_4);
    \draw (1_4) to node[above,rotate=-63] {$(1,3)$} (2_4);
    \draw[densely dashed] (2_4) to node[below,rotate=61] {$(3,1)$} (1_5);
\end{scope}
\end{tikzpicture}

 \begin{tikzpicture}[x=0.025cm,y=-0.025cm]
\begin{scope}[every node/.style={inner sep=4pt,font=\footnotesize}]
    \node (2_-1) at (-30,389) {$\dots$};
    \node (1_0) at (20,289) {$3$};
    \node (2_0) at (70,389) {$2$};
    \node (1_1) at (120,289) {$3$};
    \node (2_1) at (170,389) {$2$};
    \node (1_2) at (220,289) {$3$};
    \node (2_2) at (270,389) {$2$};
    \node (1_3) at (320,289) {$3$};
    \node (2_3) at (370,389) {$2$};
    \node (1_4) at (420,289) {$3$};
    \node (2_4) at (470,389) {$2$};
    \node (1_5) at (520,289) {$\dots$};
\end{scope}
\begin{scope}[every node/.style={font=\scriptsize},every path/.style={-{Latex[length=1.5mm,width=1mm]}}]
    \draw[densely dashed] (2_-1) to node[below,rotate=61] {$(3,1)$} (1_0);
    \draw (1_0) to node[above,rotate=-63] {$(1,3)$} (2_0);
    \draw[densely dashed] (2_0) to node[below,rotate=61] {$(3,1)$} (1_1);
    \draw (1_1) to node[above,rotate=-63] {$(1,3)$} (2_1);
    \draw[densely dashed] (2_1) to node[below,rotate=61] {$(3,1)$} (1_2);
    \draw (1_2) to node[above,rotate=-63] {$(1,3)$} (2_2);
    \draw[densely dashed] (2_2) to node[below,rotate=61] {$(3,1)$} (1_3);
    \draw (1_3) to node[above,rotate=-63] {$(1,3)$} (2_3);
    \draw[densely dashed] (2_3) to node[below,rotate=61] {$(3,1)$} (1_4);
    \draw (1_4) to node[above,rotate=-63] {$(1,3)$} (2_4);
    \draw[densely dashed] (2_4) to node[below,rotate=61] {$(3,1)$} (1_5);
\end{scope}
\end{tikzpicture}
\caption{Three positive integral friezes of type $G_2$, corresponding to the frieze points in which $(x_1,x_2)$ equals $(1,1),(1,2)$, and $(3,2)$}
\label{fig: G2frieze}
\end{figure}

\begin{thm}\label{thm: specialization}
If $Q$ is a (skew-symmetrizable valued) acyclic quiver and $S\subseteq \mathbb{R}_{>0}$ is a subsemiring, then specializing the general frieze of type $Q$ defines an injection
\[ \sigma: \mathcal{A}(S) \hookrightarrow \{ \text{$S$-valued friezes of type $Q$} \} \]
Furthermore, this map is a bijection if $S=\mathbb{R}_{>0}$, if $Q$ is Dynkin, or if $Q$ has two vertices.
\end{thm}

\begin{proof}
First, we claim that the map $\sigma$ fits into a commutative diagram.
\[
\begin{tikzpicture}[yscale=0.8]
    \node (1) at (0,0) {$\mathcal{A}(S)$};
    \node (2) at (7,0) {$\{\text{$S$-valued friezes of type $Q$} \}$};
    \node (3) at (0,-2) {$\mathcal{A}(\mathbb{R}_{>0})$};
    \node (4) at (2.5,-2) {$\mathbb{R}_{>0}^r$};
    \node (5) at (7,-2) {$\{\text{$\mathbb{R}_{>0}$-valued friezes of type $Q$} \}$};
    \draw[->] (1) to node[above] {$\sigma$} (2);
    \draw[->] (3) to node[above] {$f_\mathbf{x}$} (4);
    \draw[->] (5) to node[above] {$g$} (4);
    \draw[right hook->] (1) to node[left] {$\iota$} (3);
    \draw[right hook->] (2) to node[left] {$\iota'$} (5);
\end{tikzpicture}
\]
where $f_\cluster$ sends a point in $\mathcal{A}(\mathbb{R}_{>0})$ to its values on the initial cluster $\cluster$, $g$ sends a frieze to its values on the initial slice $\{0\}\times Q$, and the inclusions $\iota,\iota'$ are induced by $S\subseteq \mathbb{R}_{>0}$. The maps $f_\cluster$ and $g$ are bijections by Propositions \ref{prop:homemomorphism from totally positive region to positive orthant} and \ref{prop: knitting}.

To check commutativity, the composition $g\circ \iota' \circ \sigma$ sends a point in $\mathcal{A}(S)$ to its values on the initial slice of the general frieze. Since this initial slice is the initial cluster $\cluster$, the composition $g\circ \iota' \circ \sigma$ sends a point in $\mathcal{A}(S)$ to its values on the initial cluster. This is the definition of $f_\cluster\circ \iota$, and so the diagram commutes. 
Furthermore, $\sigma$ is a bijection if $S=\mathbb{R}_{>0}$.

Let $F$ be an $S$-valued frieze of type $Q$. 
By the commutativity of the diagram, $F$ is the specialization of the general frieze at a unique point $p:=f_\cluster^{-1}\circ g \circ \iota'(F) \in \mathcal{A}(\mathbb{R}_{>0})$, which may or may not be in $\mathcal{A}(S)$. Since the preimage $\sigma^{-1}(F)$ has at most one element, $\sigma$ is injective. 
In the case that $Q$ is Dynkin or has two vertices, every cluster variable occurs somewhere in the general frieze. Since $p$ sends the general frieze to $F$ and $F$ has values in $S$, it follows that $p$ sends every cluster variable into $S$; that is, $p\in \mathcal{A}(S)$. Therefore, $\sigma$ is surjective.
\end{proof}

\begin{rem}
\label{rem: integralspecialization}
The map $\sigma$ is a bijection whenever $S=\mathbb{Z}_{\geq1}$, but showing this requires a presentation of $\mathcal{A}$ which hasn't explicitly appeared in the literature. 
In \cite{BNI20}, the authors show that the cluster algebra of an acyclic quiver $Q$ may be generated by the cluster variables which appear in the general frieze on the initial slice and its translation. We claim that
\begin{enumerate}
    \item their argument extends \emph{mutatis mutandis} to the skew-symmetrizable case, and
    \item the relations among this generating set are generated by the mesh relations.
\end{enumerate}

Assuming such a presentation, the proof is as follows.
Given a $\mathbb{Z}_{\geq 1}$-valued frieze of type $Q$, its values on this generating set satisfy the mesh relations and extend to a ring homomorphism $p:\mathcal{A}\rightarrow \mathbb{Z}$. 
The specialization $\sigma(p)$ of the general frieze is then $\mathbb{Z}$-valued and agrees with original frieze on the initial slice; by Proposition \ref{prop: knitting}, they coincide.
\end{rem}

\begin{rem}
However, $\sigma$ may not be a bijection for other semirings $S$. For example,
\[ \sigma: \mathcal{A}(\mathbb{R}_{\geq1}) \hookrightarrow \{\text{$ \mathbb{R}_{\geq1}$-valued friezes of type $Q$}\}\]
need not be a 
surjection. 
A counterexample is given by the following acyclic quiver. 
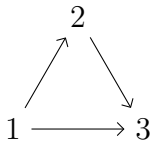
\[ 
\begin{tikzpicture}
    \node (1) at (210:1) {$1$};
    \node (2) at (90:1) {$2$};
    \node (3) at (-30:1) {$3$};
    \draw[-angle 90] (1) to (2);
    \draw[-angle 90] (2) to (3);
    \draw[-angle 90] (1) to (3);
\end{tikzpicture}
\]
The totally positive point $p$ of the associated cluster algebra with $(x_1,x_2,x_3) = (1,3,1)$ sends every cluster variable in the general frieze into $\mathbb{R}_{\geq1}$, and so $\sigma(p)$ is an $\mathbb{R}_{\geq1}$-valued frieze. This cluster algebra has two cluster variables not in the general frieze, one of which is the mutation $x_2'$ of $x_2$ in the initial seed. Since $p(x_2')=2/3<1$, the point  $p\not\in \mathcal{A}(\mathbb{R}_{\geq1})$.
\end{rem}

\begin{rem}
If $S\subseteq\mathbb{R}$ but $S\not\subseteq\mathbb{R}_{>0}$, $\sigma$ need not be surjective even in Dynkin type. For example, for any aperiodic sequence $...,a_0,a_1,a_2,...$ of real numbers, the $\mathbb{R}$-valued frieze in Figure \ref{fig: wildfrieze} is aperiodic and cannot be a specialization of the general frieze of type $A_3$.
\end{rem}

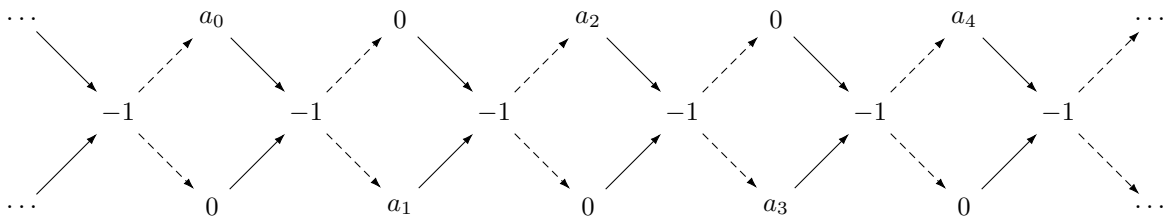
\begin{figure}[h!t]
\begin{tikzpicture}[x=0.025cm,y=-0.025cm]
\begin{scope}[every node/.style={inner sep=4pt,font=\footnotesize}]
    \node (1_-1) at (-80,289) {$\dots$};
    \node (2_-1) at (-30,339) {$-1$};
    \node (3_-1) at (-80,389) {$\dots$};

    \node (1_0) at (20,289) {$a_0$};
    \node (2_0) at (70,339) {$-1$};
    \node (3_0) at (20,389) {$0$};
    
    \node (1_1) at (120,289) {$0$};
    \node (2_1) at (170,339) {$-1$};
    \node (3_1) at (120,389) {$a_1$};
   
    \node (1_2) at (220,289) {$a_2$};
    \node (2_2) at (270,339) {$-1$};
    \node (3_2) at (220,389) {$0$};
    
    \node (1_3) at (320,289) {$0$};
    \node (2_3) at (370,339) {$-1$};
    \node (3_3) at (320,389) {$a_3$};
       
    \node (1_4) at (420,289) {$a_4$};
    \node (2_4) at (470,339) {$-1$};
    \node (3_4) at (420,389) {$0$};
   
    \node (1_5) at (520,289) {$\dots$};
    \node (3_5) at (520,389) {$\dots$};
\end{scope}
\begin{scope}[every node/.style={fill=white,font=\scriptsize},every path/.style={-{Latex[length=1.5mm,width=1mm]}}]
    \draw (1_-1) to (2_-1);
    \draw (3_-1) to (2_-1);
    \draw[densely dashed] (2_-1) to (1_0);
    \draw[densely dashed] (2_-1) to (3_0);
   
    \path (1_0) edge (2_0);
    \path (2_0) edge[densely dashed] (1_1);
    \path (3_0) edge (2_0);
    \path (2_0) edge[densely dashed] (3_1);
    \path (1_1) edge (2_1);
    \path (3_1) edge (2_1);
    \path (2_1) edge[densely dashed] (3_2);
    \path (2_1) edge[densely dashed] (1_2);
    \path (3_2) edge (2_2);
    \path (1_2) edge (2_2);
    \path (2_2) edge[densely dashed] (3_3);
    \path (2_2) edge[densely dashed] (1_3);
    \path (3_3) edge (2_3);
    \path (1_3) edge (2_3);
    
    \path (2_3) edge[densely dashed] (3_4);
    \path (2_3) edge[densely dashed] (1_4);
    \path (3_4) edge (2_4);
    \path (1_4) edge (2_4);

    \draw[densely dashed] (2_4) to (1_5);
    \draw[densely dashed] (2_4) to (3_5);
\end{scope}
 \end{tikzpicture}
 \caption{A family of wild non-positive friezes of type $A_3$}
 \label{fig: wildfrieze}
\end{figure}

\subsection{Positive integral friezes}
\label{section: enumeration}

For $S=\mathbb{Z}_{\geq1}$, specializing the general frieze
is a bijection\footnote{Theorem~\ref{thm: specialization} covers bijectivity for $Q$ Dynkin and rank $2$, and Remark \ref{rem: integralspecialization} addresses the general case.}
\begin{equation}
\sigma:\mathcal{A}(\mathbb{Z}_{\geq1}) :=\{\text{frieze points of $\mathcal{A}$}\} \rightarrow \{\text{positive integral friezes of type $Q$}\} 
\end{equation}
This is the justification for calling elements of the former set `\emph{frieze points}'.

\begin{rem}
This is somewhat inconsistent with our own terminology, since the frieze points only correspond to positive integral (i.e.~$\mathbb{Z}_{\geq1}$-valued) friezes, while other real points of $\mathcal{A}$ determine other kinds of real-valued friezes of type $Q$. Our weak justification is that this is inherited from the literature, which tends to regard positive integral friezes as the `true friezes', and the rest as inferior approximations.
\end{rem}

\begin{ex}
There are nine frieze points in the $G_2$ cluster algebra (Example \ref{ex: G2friezepoints}), which correspond to nine positive integral friezes of type $G_2$. Three are given in Figure \ref{fig: G2frieze} and the rest are obtained by translations of the first two.
\end{ex}

Recall that the map sending a cluster $\mathbf{x}$ to the unitary point $\mathbf{1}_\cluster$ is an inclusion, and so we have a composable pair of inclusions
\begin{equation}
\{\text{clusters in $\mathcal{A}$}\} \hookrightarrow \mathcal{A}(\mathbb{Z}_{\geq1})\hookrightarrow \{\text{positive integral friezes of type $Q$}\}
\end{equation}
Friezes in the image of this composition will be called \textbf{unitary friezes}; see \cite{FP16,GS20,CFGT22} for further consideration.
When $Q$ is not Dynkin, the classification of finite type cluster algebras implies there are infinitely many positive integral friezes of type $Q$. 

The converse to this observation is the main result of this section.

\begin{thmIntro}
\label{thm: finitefriezes}
There are finitely many positive integral friezes of each Dynkin type.
\end{thmIntro}
\begin{proof}
For a valued quiver of Dynkin type $Q$, positive integral friezes of type $Q$ are in bijection with frieze points 
by Theorem \ref{thm: specialization}, 
which are 
finite by Corollary \ref{coro: finite}.
\end{proof}

Earlier works have shown the finiteness of positive integral friezes for many specific Dynkin types, usually by direct enumeration. 
Table \ref{table: friezes} collects results known to the authors, and the citations are as follows.

\begin{table}
\[\begin{array}{|c|c|c|}
\hline
& & \\[-.4cm]
\text{Dynkin Type} & \text{\# of Positive Integral Friezes} & 
\begin{array}{c}
\text{\# of Unitary Friezes}  \\
=\text{(\# of Clusters)}
\end{array}
\\[.1cm]
\hline
& & \\[-.3cm]
A_n & \displaystyle \frac{1}{n+2} \binom{2n+2}{n+1} & \displaystyle \frac{1}{n+2} \binom{2n+2}{n+1}  
\\[.4cm]
\hline
& & \\[-.3cm]
B_n & \displaystyle \sum_{m=1}^{\sqrt{n+1}} \binom{2n-m^2+1}{n} & \displaystyle \binom{2n}{n} 
\\[.5cm]
\hline
& & \\[-.3cm]
C_n & \displaystyle \binom{2n}{n} & \displaystyle \binom{2n}{n} 
\\[.4cm]
\hline
& & \\[-.3cm]
D_n & \displaystyle \sum_{m=1}^n d(m) \binom{2n-m-1}{n-m} & \displaystyle \frac{3n-2}{n} \binom{2n-2}{n-1} 
\\[.5cm]
\hline
& & \\[-.4cm]
E_6 & 868 & 833 
\\[.1cm]
\hline
& & \\[-.4cm]
E_7 & \text{Open (conjectured: $4400$)} & 4160 
\\[.1cm]
\hline
& & \\[-.4cm]
E_8 & \text{Open (conjectured: $26952$)} & 25080
\\[.1cm]
\hline
& & \\[-.4cm]
F_4 & 112 & 105 
\\[.1cm]
\hline
& & \\[-.4cm]
G_2 & 9 & 8 
\\[.1cm]
\hline
\end{array}\]
\[ \footnotesize d(m) := \text{ the number of divisors of $m$} \]
\caption{Counts of positive integral friezes}
\label{table: friezes}
\end{table}

\begin{itemize}
    \item Positive integral friezes of type $A_n$ were first enumerated in \cite{ConCox73}, who constructed a bijection with triangulations of polygons.
    \item The number of positive integral friezes of type $D_4$ was conjectured to be $51$ during early 2000s in~\cite{Pro20}, and later proven in~\cite{MOT12}.
    \item Positive integral friezes of types $B_n,C_n,D_n,$ and $G_2$ were enumerated in \cite{FP16}, who also conjectured counts in the remaining Dynkin types based on a computer search.
    \item In types $E_6$ and $F_4$, these counts were verified by Cuntz and Plamondon in~\cite[Appendix B]{BFGST21}.
    \item In each case, the number of unitary friezes is equal to the number of clusters in the corresponding cluster algebra, which were enumerated in \cite{FZ03}. 
\end{itemize}
Unfortunately, our general argument for finiteness is purely topological, and doesn't provide a direct method to enumerate positive integral friezes. 

However, we observe 
that any frieze point which does not lie in the interior of the superunitary region can be `extended' from a frieze point in a cluster algebra of lower rank.

\begin{lemma}\label{lemma: lifting}
Given a subcluster $\cluster$ in finite type $\mathcal{A}$ with deletion map $d:\mathcal{A}\rightarrow \mathcal{A}^\dagger$, the induced homeomorphism $d^*: \mathcal{A}^\dagger(\mathbb{R}_{>1}) \xrightarrow{\sim}\mathcal{A}(\mathbb{R}_{\geq1}) _\cluster$ (Corollary \ref{coro: deletionhomeomorphism}) restricts to bijections
\[
\mathcal{A}^\dagger(\mathbb{Z}_{\geq2}) \xrightarrow{\sim}\mathcal{A}(\mathbb{R}_{\geq1}) _\cluster \cap \mathcal{A}(\mathbb{Z}_{\geq1})
\text{ and }
\mathcal{A}^\dagger(\mathbb{Z}_{\geq1}) \xrightarrow{\sim}\overline{\mathcal{A}(\mathbb{R}_{\geq1}) _\cluster} \cap \mathcal{A}(\mathbb{Z}_{\geq1})
\]
\end{lemma}


\begin{proof}
Proposition \ref{prop:deletion map} gives the following possibilities for each cluster variable $y\in \mathcal{A}$.
\begin{itemize}
    \item If $y\in \cluster$, then $d(y)=1$.
    \item If $y$ is compatible with $\cluster$ but not in $\cluster$, then $d(y)$ is a cluster variable in $\mathcal{A}^\dagger$. Every cluster variable in $\mathcal{A}^\dagger$ is of this form.
    \item Otherwise, $d(y)$ is a sum of at least two cluster monomials.
\end{itemize}

If $p\in \mathcal{A}^\dagger(\mathbb{Z}_{\geq2})$, then $p$ sends every cluster monomial other than $1$ into $\mathbb{Z}_{\geq2}$; therefore, $p(d(y))\geq1$ with equality iff $y\in \cluster$. That is,
\[ d^*(\mathcal{A}^\dagger(\mathbb{Z}_{\geq2})) \subseteq \mathcal{A}(\mathbb{R}_{\geq1}) _\cluster \cap \mathcal{A}(\mathbb{Z}_{\geq1})\]

If $q\in \mathcal{A}(\mathbb{R}_{\geq1}) _\cluster \cap \mathcal{A}(\mathbb{Z}_{\geq1})$, then by the invertibility of $d^*: \mathcal{A}^\dagger(\mathbb{R}_{>1}) \xrightarrow{\sim}\mathcal{A}(\mathbb{R}_{\geq1}) _\cluster$ there is some $p\in \mathcal{A}^\dagger(\mathbb{R}_{\geq1})$ such that $d^*(p)=p\circ d=q$. If $z$ is a cluster variable in $\mathcal{A}^\dagger$, then $z=d(y)$ for some cluster variable $y$ in $\mathcal{A}$ which is compatible with $\cluster$ but not in $\cluster$. Then $p(z) = p(d(y))=q(y)$, which is in $\mathbb{Z}_{\geq2}$ by the assumption on $q$. Therefore, $p\in \mathcal{A}^\dagger(\mathbb{Z}_{\geq2})$, so \[d^*:\mathcal{A}^\dagger(\mathbb{Z}_{\geq2}) \xrightarrow{\sim}\mathcal{A}(\mathbb{R}_{\geq1}) _\cluster \cap \mathcal{A}(\mathbb{Z}_{\geq1})\]
The second bijection follows an almost identical proof.
\end{proof}

Translated into positive integral friezes, the lemma provides bijections
\begin{align*}
\{\text{$\mathbb{Z}_{\geq1}$-valued friezes of type $\Gamma'$}\}
&\xrightarrow{\sim}
\{\text{$\mathbb{Z}_{\geq1}$-valued friezes of type $\Gamma$ which are $1$ at $\cluster$}\}
\\
\{\text{$\mathbb{Z}_{\geq2}$-valued friezes of type $\Gamma'$}\}
&\xrightarrow{\sim}
\{\text{$\mathbb{Z}_{\geq1}$-valued friezes of type $\Gamma$ which are $1$ only at $\cluster$}\}
\end{align*}
where $\Gamma'$ is the Dynkin diagram of the deletion of $\cluster$. The image of a positive integral frieze under either of these maps will be called a \textbf{unitary extension} of that frieze, because it is `extending' the smaller frieze by adding $1$s.


The second bijection above implies that each positive integral frieze is a unitary extension of a unique $\mathbb{Z}_{\geq2}$-valued frieze. 
As a consequence,
\begin{equation}\label{eq: sumoveratleast2}
\text{\# of $\mathbb{Z}_{\geq1}$-valued friezes of type $\Gamma$} 
= \sum_{\text{subclusters $\cluster$}}\text{(\# of $\mathbb{Z}_{\geq2}$-valued friezes of type $\Gamma'$)} 
\end{equation}
where $\Gamma'$ denotes the Dynkin diagram of the deletion of $\cluster$. 

\begin{warn}
The sum \eqref{eq: sumoveratleast2} is \emph{not} over subdiagrams $\Gamma'$ of $\Gamma$, as there are generally many subclusters $\cluster$ whose deletion has the same Dynkin diagram $\Gamma'$. The number of subclusters of each type (equivalently, the number of faces of each type) is computed in Appendix \ref{section: countingfaces}. 
\end{warn}

When $\cluster$ is a cluster, the deletion of $\cluster$ is the empty Dynkin diagram, and there is a unique $\mathbb{Z}_{\geq2}$-valued frieze (the empty frieze). A unitary extension of the empty frieze is a unitary frieze, and so the summands of \eqref{eq: sumoveratleast2} corresponding to clusters count the unitary friezes.

\begin{rem}\label{rem: emptydiagram}
The `empty frieze' is more natural in the language of frieze points.
When $\cluster$ is a cluster, the deletion $\mathcal{A}^\dagger$ of $\cluster$ is just the ground ring $\mathbb{Z}$, which we regard as a cluster algebra with no cluster variables. As a consequence, the unique ring homomorphism $\mathbb{Z}\rightarrow \mathbb{R}$ is vacuously in $\mathcal{A}^\dagger(S)$ for every $S\subset \mathbb{R}$. In particular, $\mathcal{A}^\dagger(\mathbb{Z}_{\geq2})$ has a single element which maps to the unitary point $\mathbf{1}_\cluster$ in $\mathcal{A}(\mathbb{Z}_{\geq1})$ along the homomorphism in Lemma \ref{lemma: lifting}.
\end{rem}

Using \eqref{eq: sumoveratleast2} to reverse engineer counts of $\mathbb{Z}_{\geq2}$-valued friezes yields a surprisingly short list.

\begin{thmIntro}
\label{thm: elementaryfriezes}
Assuming there are $4400$ and $26952$ positive integral friezes of types $E_7$ and $E_8$,
every positive integral frieze of Dynkin type is a unitary extension of a unique (possibly empty) union of the following friezes: 
\begin{enumerate}
    \item The family of $D_n$ friezes in Figure \ref{fig: Dabfrieze} determined by a proper factorization $n=ab$.
    \item The four translations of the $E_8$ frieze in Figure \ref{fig: E8frieze}.
    \item The family of $B_n$ friezes in Figure \ref{fig: Bn2frieze}, for all $\sqrt{n+1}\in \mathbb{Z}_{\geq2}$.
    \item The $G_2$ frieze at the bottom of Figure \ref{fig: G2frieze}.
\end{enumerate}
\end{thmIntro}

\noindent This theorem is proven is Appendix \ref{section: countingfriezes}.

\begin{rem}
Proving the number of $\mathbb{Z}_{\geq2}$-valued friezes of types $E_7$ and $E_8$ are $0$ and $4$ would prove the number of $\mathbb{Z}_{\geq1}$-valued friezes of types $E_7$ and $E_8$ are $4400$ and $26952$.
\end{rem}

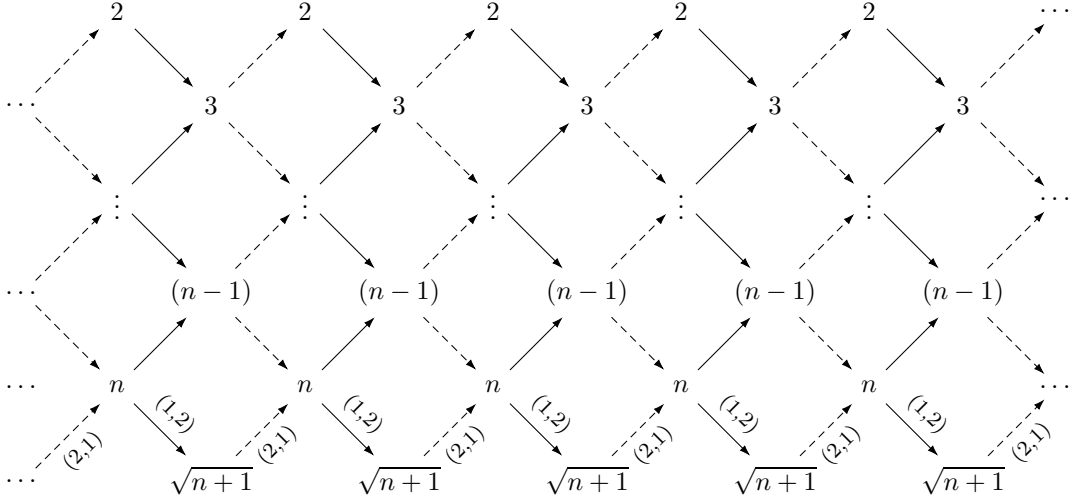
\begin{figure}[h!t]

\begin{tikzpicture}[x=0.025cm,y=-0.025cm]
\begin{scope}[every node/.style={inner sep=4pt,font=\footnotesize}]
    \node (2_a) at (-30,339) {$\dots$};
    \node (4_a) at (-30,439) {$\dots$};
    \node (6_a) at (-30,489) {$\dots$};
    \node (7_a) at (-30,539) {$\dots$};

    \node (1_0) at (20,289) {$2$};
    \node (2_0) at (70,339) {$3$};
    \node (3_0) at (20,389) {$\vdots$};
    \node (4_0) at (70,439) {$(n-1)$};
    \node (5_0) at (20,489) {$n$};
    \node (7_0) at (70,539) {$\sqrt{n+1}$};
    \node (1_1) at (120,289) {$2$};
    \node (2_1) at (170,339) {$3$};
    \node (3_1) at (120,389) {$\vdots$};
    \node (4_1) at (170,439) {$(n-1)$};
    \node (5_1) at (120,489) {$n$};
    \node (7_1) at (170,539) {$\sqrt{n+1}$};
    \node (1_2) at (220,289) {$2$};
    \node (2_2) at (270,339) {$3$};
    \node (3_2) at (220,389) {$\vdots$};
    \node (4_2) at (270,439) {$(n-1)$};
    \node (5_2) at (220,489) {$n$};
    \node (7_2) at (270,539) {$\sqrt{n+1}$};
    \node (1_3) at (320,289) {$2$};
    \node (2_3) at (370,339) {$3$};
    \node (3_3) at (320,389) {$\vdots$};
    \node (4_3) at (370,439) {$(n-1)$};
    \node (5_3) at (320,489) {$n$};
    \node (7_3) at (370,539) {$\sqrt{n+1}$};
    \node (1_4) at (420,289) {$2$};
    \node (2_4) at (470,339) {$3$};
    \node (3_4) at (420,389) {$\vdots$};
    \node (4_4) at (470,439) {$(n-1)$};
    \node (5_4) at (420,489) {$n$};
    \node (7_4) at (470,539) {$\sqrt{n+1}$};
    
    \node (1_5) at (520,289) {$\dots$};
    \node (3_5) at (520,389) {$\dots$};
    \node (5_5) at (520,489) {$\dots$};
\end{scope}
\begin{scope}[every node/.style={font=\scriptsize},every path/.style={-{Latex[length=1.5mm,width=1mm]}}]
    \path (2_a) edge[densely dashed] (1_0);
    \path (2_a) edge[densely dashed] (3_0);
    \path (4_a) edge[densely dashed] (3_0);
    \path (4_a) edge[densely dashed] (5_0);
    \path (7_a) edge[densely dashed] node[pos=0.5,rotate=45,below]{\tiny (2,1)} (5_0);
    \path (1_0) edge (2_0);
    \path (3_0) edge (2_0);
    \path (3_0) edge (4_0);
    \path (5_0) edge (4_0);
    \path (5_0) edge node[pos=0.5,rotate=-45,above]{\tiny (1,2)} (7_0);
    \path (2_0) edge[densely dashed] (1_1);
    \path (2_0) edge[densely dashed] (3_1);
    \path (4_0) edge[densely dashed] (3_1);
    \path (4_0) edge[densely dashed] (5_1);
    \path (7_0) edge[densely dashed] node[pos=0.5,rotate=45,below]{\tiny (2,1)} (5_1);
    \path (1_1) edge (2_1);
    \path (3_1) edge (2_1);
    \path (3_1) edge (4_1);
    \path (5_1) edge (4_1);
    \path (5_1) edge node[pos=0.5,rotate=-45,above]{\tiny (1,2)} (7_1);
    \path (2_1) edge[densely dashed] (1_2);
    \path (2_1) edge[densely dashed] (3_2);
    \path (4_1) edge[densely dashed] (3_2);
    \path (4_1) edge[densely dashed] (5_2);
    \path (7_1) edge[densely dashed] node[pos=0.5,rotate=45,below]{\tiny (2,1)} (5_2);
    \path (1_2) edge (2_2);
    \path (3_2) edge (2_2);
    \path (3_2) edge (4_2);
    \path (5_2) edge (4_2);
    \path (5_2) edge node[pos=0.5,rotate=-45,above]{\tiny (1,2)} (7_2);
    \path (2_2) edge[densely dashed] (1_3);
    \path (2_2) edge[densely dashed] (3_3);
    \path (4_2) edge[densely dashed] (3_3);
    \path (4_2) edge[densely dashed] (5_3);
    \path (7_2) edge[densely dashed] node[pos=0.5,rotate=45,below]{\tiny (2,1)} (5_3);
    \path (1_3) edge (2_3);
    \path (3_3) edge (2_3);
    \path (3_3) edge (4_3);
    \path (5_3) edge (4_3);
    \path (5_3) edge node[pos=0.5,rotate=-45,above]{\tiny (1,2)} (7_3);
    \path (2_3) edge[densely dashed] (1_4);
    \path (2_3) edge[densely dashed] (3_4);
    \path (4_3) edge[densely dashed] (3_4);
    \path (4_3) edge[densely dashed] (5_4);
    \path (7_3) edge[densely dashed] node[pos=0.5,rotate=45,below]{\tiny (2,1)} (5_4);
    \path (1_4) edge (2_4);
    \path (3_4) edge (2_4);
    \path (3_4) edge (4_4);
    \path (5_4) edge (4_4);
    \path (5_4) edge node[pos=0.5,rotate=-45,above]{\tiny (1,2)} (7_4);
    \path (2_4) edge[densely dashed] (1_5);
    \path (2_4) edge[densely dashed] (3_5);
    \path (4_4) edge[densely dashed] (3_5);
    \path (4_4) edge[densely dashed] (5_5);
    \path (7_4) edge[densely dashed] node[pos=0.5,rotate=45,below]{\tiny (2,1)} (5_5);
\end{scope}
 \end{tikzpicture}
\caption{The unique $\mathbb{Z}_{\geq2}$-valued frieze of type $B_{n}$ for $\sqrt{n+1}\in \mathbb{Z}_{\geq2}$}
\label{fig: Bn2frieze}
\end{figure}

\section{Superunitary regions in infinite type}
\label{section: infinitetype}

While $\mathcal{A}(S)$ (Definition \ref{defn: AS}) makes sense for infinite type cluster algebras, our two key tools (Proposition~\ref{prop:deletion map} and Lemma~\ref{lemma:U}) both assume finite type. This assumption is necessary so that the cluster monomials form a basis (Theorem \ref{thm:cluster monomials form a positive basis in finite type}). This suggests that, in infinite type, we should extend the preceding definition from the set of cluster variables to a sufficiently nice basis containing the cluster monomials. The resulting set, denoted $\mathcal{A}_\mathfrak{B}(S)$ (Definition~\ref{defn: ASgeneral}), depends on a choice of basis $\mathfrak{B}$, but we expect it to be better behaved in general. In many cases, this extra choice of basis is irrelevant and $\mathcal{A}_\mathfrak{B}(S)$ coincides with $\mathcal{A}(S)$.

\subsection{Good bases for cluster algebras}
\label{section: goodbases}

An essential feature of cluster algebras is the existence of certain bases satisfying positivity conditions. 
For the sake of discussion, let's establish a definition.

\begin{idefn}
A \textbf{good basis} for a cluster algebra $\mathcal{A}$ is a $\mathbb{Z}$-basis $\mathfrak{B}$ which contains the cluster monomials and has non-negative integer structure constants; that is, the product of two basis elements is a non-negative integer linear combination of basis elements.
\end{idefn}

Good bases are fundamental to cluster algebras and their origins; cluster algebras were introduced to axiomatize patterns found in the dual canonical basis of Lusztig and Kashiwara.
Accordingly, there has been extensive work on constructing good bases for cluster algebras.
\begin{itemize}
    \item In finite type, the \textbf{cluster monomials} form the unique good basis (Theorem \ref{thm:cluster monomials form a positive basis in finite type}).
    \item Rank 2 cluster algebras have a good basis of \textbf{greedy elements} \cite{LLZ14Greedy}, which are defined by a \emph{greedy recurrence relation} among their Laurent coefficients.
    \item The atomic positive elements sometimes form a good basis, called the \textbf{atomic basis}.
    \item Cluster algebras of marked surfaces with enough marked points admit several topologically defined bases \cite{MSW13}.  Among these, the \textbf{bracelet basis} is good, the \textbf{bangles basis} is almost never good, and the \textbf{bands basis} might be good \cite{Thu14}.
    \item Cluster algebras from Lie theory (such as double Bruhat cells, Grassmannians, and flag varieties) admit several representation theoretic bases, such as the \textbf{dual canonical basis}, the \textbf{dual semicanonical basis}, the \textbf{web basis}, and the \textbf{Mirkovi\'c---Vilonen basis}. Among these, the web basis is known to not be good \cite[Proposition 9.11]{FP16tensor} and the Mirkovi\'c--Vilonen basis was conjectured to be good by \cite{Kam22,Qin21}.
    \item Every cluster algebra admits a good basis of \textbf{theta functions} \cite{GHKK18},\footnote{This is a slight lie; in general, the theta functions are a good basis for an algebra that contains the cluster algebra and is contained in the upper cluster algebra. However, these two algebras coincide for most cluster algebras of interest, including virtually all previous bullet points.} whose coefficients count \emph{broken lines} in a \emph{scattering diagram}. 
\end{itemize}

\noindent In summary: good bases always exist,\footnote{At least, if one is willing to replace a cluster algebra with an algebra contained in the upper cluster algebra.} but a cluster algebra may admit several good bases and the relative virtues of each basis are an active area of study and debate.

\subsection{Semiring-valued points}

Given a good basis $\mathfrak{B}$ for a cluster algebra $\mathcal{A}$, let 
\[ \mathcal{A}_\mathfrak{B}^+ := \mathrm{span}_\mathbb{N}(\mathfrak{B}) 
:= 
\left\{
\sum_{b\in \mathfrak{B}} c_bb \in \mathcal{A} \text{ such that  $c_b\in \mathbb{Z}_{\geq0}$ 
and $0<\sum_{b\in \mathfrak{B}} c_b <\infty$}
\right\}
\]
denote the set of non-zero finite sums of the basis elements. The positivity of the structure constants implies $\mathcal{A}_\mathfrak{B}^+$ is closed under multiplication, and is therefore a commutative semiring.

A choice of good basis for $\mathcal{A}$ facilitates the following definition.

\begin{defn}\label{defn: ASgeneral}
Given a cluster algebra $\mathcal{A}$ with a good basis $\mathfrak{B}$ and a commutative semiring $S$, define the set
\[ \mathcal{A}_\mathfrak{B}(S) := \{ \text{semiring homomorphisms $p:\mathcal{A}_\mathfrak{B}^+\rightarrow S$}\} \]
\end{defn}

Each element $a\in \mathcal{A}_\mathfrak{B}^+$ defines an $S$-valued function $f_a$ on $\mathcal{A}_\mathfrak{B}(S)$, via the rule that $f_a(p):= p(a)$. 
If $S$ is a topological semiring (that is, a semiring with a topology for which addition and multiplication are continuous), then we endow $\mathcal{A}_\mathfrak{B}(S)$ with the coarsest topology for which each $f_a$ is continuous for all $a\in \mathcal{A}_\mathfrak{B}^+$ (though it suffices to consider $a\in\mathfrak{B}$).

When $S$ is a real semiring, this may be related to Definition \ref{defn: AS} as follows. 

\begin{prop}
Let $\mathcal{A}$ be a cluster algebra with a good basis $\mathfrak{B}$, and let $S\subset \mathbb{R}$ be a subsemiring. Then
\[ \mathcal{A}_\mathfrak{B}(S) = \{ \text{ring homomorphisms $p:\mathcal{A}\rightarrow \mathbb{R}$ such that $p(\mathfrak{B})\subset S$}\} \subseteq \mathcal{A}(S)\]
\end{prop}

\begin{proof}
A semiring homomorphism $\mathcal{A}_\mathfrak{B}^+\rightarrow \mathbb{R}$ and a ring homomorphism $\mathcal{A}\rightarrow \mathbb{R}$ are both determined by their restriction to the basis $\mathfrak{B}$, which will be a function $p:\mathfrak{B}\rightarrow \mathbb{R}$ satisfying
\begin{align}\label{eq: multiplicativepropertyofbasis}
p(a) p(b) = \sum_{c\in \mathfrak{B}} \lambda^c_{a,b} p(c) 
\end{align}
for every $a,b\in \mathfrak{B}$ with $ ab = \sum_{c\in \mathfrak{B}} \lambda^c_{a,b} c$. Furthermore, any function $\mathfrak{B}\rightarrow S$ satisfying Equation \eqref{eq: multiplicativepropertyofbasis} will extend to a semiring homomorphism $p:\mathcal{A}_\mathfrak{B}^+\rightarrow S$ and to a ring homomorphism $p:\mathcal{A}\rightarrow \mathbb{R}$ such that $p(\mathfrak{B})\subset S$. Therefore, every semiring homomorphism $p:\mathcal{A}_\mathfrak{B}^+\rightarrow S$ is the restriction of a unique ring homomorphism $p:\mathcal{A}\rightarrow \mathbb{R}$ such that $p(\mathfrak{B})\subset S$. Since $\mathfrak{B}$ contains the cluster variables, the latter set is a subset of $\mathcal{A}(S)$.
\end{proof}

It is therefore natural to ask when $\mathcal{A}_\mathfrak{B}(S)=\mathcal{A}(S)$; that is: for which $\mathcal{A}, \mathfrak{B}$, and $S$ do Definitions \ref{defn: AS} and \ref{defn: ASgeneral} agree?
Two of the simpler cases follow; more interesting cases are considered by Theorem \ref{thm: Bgreedy} and Conjecture \ref{conj: Bir}.

\begin{prop}
If $\mathcal{A}$ is finite type and $S$ is a semiring in $\mathbb{R}$, then $\mathcal{A}_\mathfrak{B}(S) = \mathcal{A}(S)$.
\end{prop}

\begin{proof}
By Theorem \ref{thm:cluster monomials form a positive basis in finite type}, $\mathfrak{B}$ must be the set of cluster monomials.
For any $p\in \mathcal{A}(S)$, $f_x(p)=p(x)\in S$ for every cluster variable $x$. 
Given a cluster monomial $x_1^{e_1}x_2^{e_2}\cdots x_r^{e_r}$, 
\[ 
p(x_1^{e_1}x_2^{e_2}\cdots x_r^{e_r}) = p(x_1)^{e_1}p(x_2)^{e_2}\cdots p(x_r)^{e_r} \]
Since $S$ is closed under multiplication, this is in $S$. Thus, $p(\mathfrak{B})\subset S$ and so $p\in \mathcal{A}_\mathfrak{B}(S)$.
\end{proof}

\begin{prop}
If $S\subset \mathbb{R}$ is a semifield, then $\mathcal{A}_\mathfrak{B}(S) = \mathcal{A}(S)\simeq S^{\rank(\mathcal{A})}$. 
\end{prop}

\begin{proof}[Proof sketch]
By an identical argument to Proposition \ref{prop:homemomorphism from totally positive region to positive orthant}, a point in $\mathcal{A}_\mathfrak{B}(S)$ is freely determined by its values on a fixed cluster $\cluster$.
\end{proof}

\subsection{$\mathfrak{B}$-superunitary regions}

Let us return to our favorite semirings $\mathbb{R}_{>0},\mathbb{R}_{\geq1},\mathbb{Z}_{\geq1}\subset \mathbb{R}$. The last three propositions imply that
\[ 
\mathcal{A}_\mathfrak{B}(\mathbb{R}_{> 0}) 
= 
\mathcal{A}(\mathbb{R}_{> 0}) \simeq \mathbb{R}_{>0}^{\rank(\mathcal{A})}
\hspace{1cm}
\mathcal{A}_\mathfrak{B}(\mathbb{R}_{\geq 1}) 
\subseteq 
\mathcal{A}(\mathbb{R}_{\geq1}) 
\hspace{1cm}
\mathcal{A}_\mathfrak{B}(\mathbb{Z}_{\geq 1}) 
\subseteq 
\mathcal{A}(\mathbb{Z}_{\geq1}) 
\]
with equalities in finite type. 

The \textbf{$\mathfrak{B}$-superunitary region} of a cluster algebra $\mathcal{A}$ with good basis $\mathfrak{B}$ is the space $\mathcal{A}_\mathfrak{B}(\mathbb{R}_{\geq 1})$, which has the following equivalent characterizations.
\begin{itemize}
    \item The set of semiring homomorphisms $p:\mathcal{A}_\mathfrak{B}^+\rightarrow \mathbb{R}_{\geq1}$.
    \item The set of ring homomorphisms $p:\mathcal{A}\rightarrow \mathbb{R}$ which send each element in $\mathfrak{B}$ into $\mathbb{R}_{\geq1}$.
    \item The subspace of $\mathcal{A}(\mathbb{R})$ on which $f_b$ is greater than or equal to $1$, for each  $b\in \mathfrak{B}$.
\end{itemize}
At a glance, one might expect this to be a smaller space than the superunitary region $\mathcal{A}(\mathbb{R}_{\geq1})$, since it is carved out by a larger set of inequalities; specifically,  $f_b\geq1$ for all $b\in \mathfrak{B}$ rather than just the cluster variables. However, these additional conditions are often vacuous.

\begin{ex}\label{ex: (2,2)}
Consider the rank 2 cluster algebra $\mathcal{A}$ with $(b,c)=(2,2)$ (c.f.~Example \ref{ex: rank2}). 
The clusters are adjacent pairs of cluster variables $\{x_i,x_{i+1}\}$, and so the cluster monomials are of the form $x_i^ax_{i+1}^b$ for $a,b\in \mathbb{N}$. 
%

To extend the cluster monomials to a good basis, define a family of elements $\{t_i\}_{i\geq0}$ by the rule that
$ x_j t_i := x_{j+i}+x_{j-i} $
for all $j\in \mathbb{Z}$; note these elements satisfy the {Chebyshev recurrence} $ t_1 t_i = t_{i+1}+t_{|i-1|}$.
The set
\[ \mathfrak{B} := \{t_i\}_{i\geq 1} \cup \bigcup_{i\in \mathbb{Z}} \{x_i^\mathbb{N}x_{i+1}^\mathbb{N}\} \]
forms a good basis for the cluster algebra $\mathcal{A}$.\footnote{Note that $t_0=2$ is \emph{not} a basis element; however, this value of $t_0$ is force by the recurrence.} This basis is an instance of the greedy basis, the bracelet basis, and the theta basis.

The superunitary region of $\mathcal{A}$ (or rather, its embedding in $\mathbb{R}^2$) is depicted in Figure \ref{fig: (2,2)}. By definition, this region is carved out by the infinite system of inequalities $f_{x_i}\geq1$ as $x_i$ runs over all cluster variables. Note that the region is unbounded.

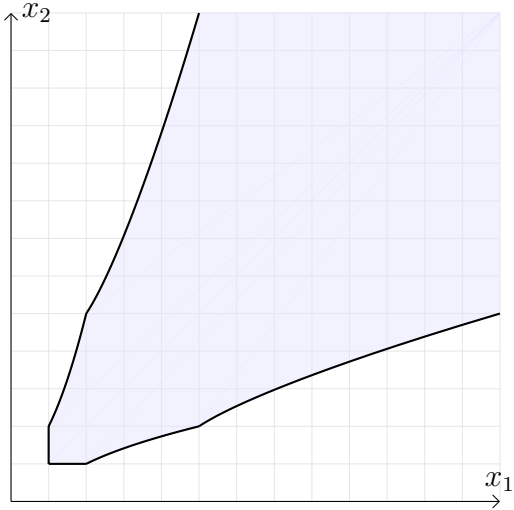
\begin{figure}[h!t]
\begin{tikzpicture}[xshift=.75in,scale=.5,baseline=(current bounding box.center)]
    \draw[step=1,black!10,very thin] (0,0) grid (13,13);
		\draw[-angle 90] (0,0) to (13,0) node[above] {$x_1$};
		\draw[-angle 90] (0,0) to (0,13) node[right] {$x_2$};
		
 		\draw[blue!10,fill=blue!10,opacity=\opa,variable=\t,domain=1:2] (13,13) plot 
 		({\t+4*\t^(-1)+8*\t^(-3)},{1+4*\t^(-2)}) to (13,13);
 		\draw[blue!10,fill=blue!10,opacity=\opa,variable=\t,domain=1:2] (13,13) plot 
 		({1+4*\t^(-2)},{2*\t^(-1)}) to (13,13);
 		\draw[blue!10,fill=blue!10,opacity=\opa,variable=\t,domain=1:2] (13,13) plot 
 		({2*\t^(-1)},1) to (13,13);
		\draw[blue!10,fill=blue!10,opacity=\opa,variable=\t,domain=1:2] (13,13) plot 
		(1,\t) to (13,13);
 		\draw[blue!10,fill=blue!10,opacity=\opa,variable=\t,domain=1:2] (13,13) plot 
 		(\t,\t^2+1) to (13,13);
 		\draw[blue!10,fill=blue!10,opacity=\opa,variable=\t,domain=1:2] (13,13) plot 
 		({\t^2+1},{\t^3+2*\t+2*\t^(-1)}) to (13,13);
		\draw[variable=\t,domain=1:2,thick] 
		plot ({\t+4*\t^(-1)+8*\t^(-3)},{1+4*\t^(-2)}) 
		plot ({1+4*\t^(-2)},{2*\t^(-1)}) 
		plot ({2*\t^(-1)},1) plot (1,\t) plot (\t,\t^2+1) 
		plot ({\t^2+1},{\t^3+2*\t+2*\t^(-1)});
  \end{tikzpicture}

\caption{The superunitary region for $(b,c)=(2,2)$}
\label{fig: (2,2)}
\end{figure}

The definition of the $\mathfrak{B}$-superunitary region imposes two types of additional conditions: 
\begin{itemize}
    \item The cluster monomials are greater than or equal to $1$ (not just the cluster variables).
    \item The non-cluster basis elements $t_1,t_2,t_3,...$ are greater than or equal to $1$.
\end{itemize}

The first family of conditions is redundant; if $f_{x_i}(p)\geq1$ and $f_{x_{i+1}}(p)\geq1$, then $ f_{x_i^ax_{i+1}^b}(p)  = f_{x_i }(p) ^af_{x_{i+1}}(p) ^b\geq1$.
The second family of conditions is vacuous because, at any totally positive point $p\in \mathcal{A}(\mathbb{R}_{>0})$,  
\[ 2 < f_{t_1}(p) < f_{t_2}(p) < f_{t_3}(p) < \cdots < f_{t_i}(p) < \cdots \]
This can be shown by induction on $i$ using $ t_1 t_i = t_{i+1}+t_{|i-1|}$. As a consequence, the $\mathfrak{B}$-superunitary region $\mathcal{A}_\mathfrak{B}(\mathbb{R}_{\geq1})$ coincides with the superunitary region $\mathcal{A}(\mathbb{R}_{\geq1})$.
\end{ex}

In general, the `extra' elements in a greedy basis contribute vacuous inequalities to $\mathcal{A}_\mathfrak{B}(\mathbb{R}_{\geq 1}) $.

\begin{thm}\label{thm: Bgreedy}
If $\mathcal{A}$ is rank 2 and $\mathfrak{B}$ is the greedy basis, then
$\mathcal{A}_\mathfrak{B}(\mathbb{R}_{\geq 1}) 
= 
\mathcal{A}(\mathbb{R}_{\geq1}) $.
\end{thm}

The theorem follows from the following fact about  positive Laurent polynomials. 

\begin{lemma}\label{lemma: support}
Let $\ell=\sum_{m\in \mathbb{Z}^d} c_m\mathbf{x}^m$ be a Laurent polynomial with positive integer coefficients. Then $\ell(p)\geq 1$ for all $p\in \mathbb{R}_{>0}^d$ iff the Newton polytope of $\ell$ contains the origin.
\end{lemma}

\begin{proof}
If the Newton polytope of $\ell$ contains the origin, then there are $m_1,m_2,..,m_n$ in the support of $\ell$ and $a_1,a_2,...,a_n\in \mathbb{N}$ such that 
$ a_1m_1+ a_2m_2 + \cdots + a_nm_n = 0 $.
Then the multinomial formula for $\ell^{a_1+a_2+\cdots +a_n}$ contains a constant term of the form
\[ c_{m_1}^{a_1} c_{m_2}^{a_2} \cdots c_{m_n}^{a_n} \binom{a_1+a_2+\cdots + a_m}{a_1 \; a_2 \; \cdots \; a_n} \]
Since the $c_1,c_2,...,c_n$ are positive integers, this term is at least $1$.
Therefore, $\ell^{a_1+a_2+\cdots +a_n}-1$ is also a Laurent polynomial with positive integer coefficients, and so it must have non-negative values at each $p\in \mathbb{R}_{>0}^d$. This implies $ \ell(p)^{a_1+a_2+\cdots +a_n} \geq 1 $, and so $\ell(p) \geq1$.

If the Newton polytope of $\ell$ does not contain the origin, then there is a linear map $v:\mathbb{Z}^d\rightarrow \mathbb{Z}$ such that $v(m)\leq -1$ for all $m$ in the support. Given $p=(p_1,p_2,...,p_d)\in \mathbb{R}_{>0}^d$,  define
\[ p_t = (e^{tv(e_1)}p_1,e^{tv(e_2)}p_2,...,e^{tv(e_d)}p_d)\]
for each $t\in \mathbb{R}_{\geq0}$, where $e_1,e_2,...,e_n$ is the standard basis for $\mathbb{Z}^n$. Then 
\[ \ell(p_t) = \sum c_mp_t^m = \sum c_m p^m e^{tv(m)} \leq \sum c_m p^me^{-t} = e^{-t} \ell(p) \]
where the sums are over the support of $\ell$. Therefore, $\ell(p_t)<1$ for $t\gg0$.
\end{proof}

\begin{proof}[Proof of Theorem \ref{thm: Bgreedy}]
The greedy basis elements which are not cluster monomials correspond to positive imaginary roots of the Cartan matrix \cite[Remark~1.9]{LLZ14Greedy}. The Laurent support of these greedy basis elements is described in \cite[Proposition~4.1, Case 6]{LLZ14Greedy}, where it may be observed that their convex hull (the Newton polytope) contains the origin. 
By Lemma \ref{lemma: support}, they are $\geq1$ everywhere on the totally positive region. 
\end{proof}

Despite the theorem, it is not always true that $\mathcal{A}_\mathfrak{B}(\mathbb{R}_{\geq1})=\mathcal{A}(\mathbb{R}_{\geq1})$.

\begin{ex}
Consider the \emph{Markov cluster algebra} determined by the following matrix.
\[
\begin{bmatrix}
0 & -2 & 2 \\
2 & 0 & -2 \\
-2 & 2 & 0
\end{bmatrix}
\]
This cluster algebra is an important source of counterexamples; e.g.~it is not equal to its own upper cluster algebra $\mathcal{U}$, it is infinitely generated, it has no maximal green sequences, etc. 

Since $\mathcal{A}$ is the cluster algebra of a once-punctured torus, we may consider the \emph{bracelet basis} $\mathfrak{B}\subset \mathcal{U}$ (including tagged arcs, as defined in \cite[Appendix A]{MSW13}). While $\mathfrak{B}$ spans an algebra, this algebra properly contains the cluster algebra and is properly contained in the upper cluster algebra. Nevertheless, we extend the preceding definitions to this algebra.

Every element in $\mathfrak{B}$ is a product of three types of elements: the plain arcs, the notched arcs, and the bracelets. 
The plain arcs are the cluster variables, and their inequalities carved out the superunitary region $\mathcal{A}(\mathbb{R}_{\geq1})$. By an application of Lemma \ref{lemma: support}, the inequalities from the bracelets are vacuous. However, each notched arc provides a non-trivial inequality, ensuring that $\mathcal{A}_\mathfrak{B}(\mathbb{R}_{\geq1})\subsetneq \mathcal{A}(\mathbb{R}_{\geq1})$.
\end{ex}

\begin{rem}\label{rem: nofrozens}
Working with a choice of good basis also provides a uniform approach to working with the many kinds of coefficients a cluster algebra may possess. 

For example, if $\mathcal{A}$ is a cluster algebra with frozen variables and their inverses, and $\mathfrak{B}$ is a good basis for $\mathcal{A}$ which is invariant under multiplication by each frozen variable (which holds in the examples of good bases mentioned above), then a $\mathfrak{B}$-superunitary point $p\in \mathcal{A}_\mathfrak{B}(\mathbb{R}_{\geq1})$ must send each frozen variable  $y$ and its inverse $y^{-1}$ into $\mathbb{R}_{\geq1}$. This forces $p(y)=1$, and so $p$ factors through the map $\mathcal{A}\rightarrow\mathcal{A}^\dagger$ which deletes the frozen variables. Therefore, the $\mathfrak{B}$-superunitary regions of $\mathcal{A}$ and $\mathcal{A}^\dagger$ coincide.
\end{rem}

\subsection{Conjectures}

We describe how our finite type results might extend to the general case. The first is a conjectural characterization of the elements in $\mathfrak{B}$ which give non-trivial inequalities in the definition of the superunitary region. 
An element $b$ in a basis $\mathfrak{B}$ is called...
\begin{itemize}
    \item \textbf{real} if $b^2\in \mathfrak{B}$, and
    \item \textbf{indecomposable} if, whenever $b=cd$ for $c,d\in \mathfrak{B}$, then one of $c,d$ must be invertible.
\end{itemize}
The cluster variables are indecomposable real elements in any good basis, and they are the only indecomposable real elements in many good bases (e.g.~greedy bases). 


Let $\mathfrak{B}_{ir}\subset\mathfrak{B}$ denote the subset of indecomposable real elements. The following conjecture asserts that these elements suffice to define the $\mathfrak{B}$-superunitary region.

\begin{conj}\label{conj: Bir}
$\mathcal{A}_\mathfrak{B}(\mathbb{R}_{\geq1})$ is the subset of $\mathcal{A}(\mathbb{R}_{>0})$ on which $f_b\geq1$ for each $b\in \mathfrak{B}_{ir}$.
\end{conj}

\noindent The conjecture would imply that $ \mathcal{A}_\mathfrak{B}(\mathbb{R}_{\geq1}) = \mathcal{A}(\mathbb{R}_{\geq1}) $ whenever $\mathfrak{B}_{ir}$ is the set of cluster variables, generalizing Theorem \ref{thm: Bgreedy}.


A subset of $\mathfrak{B}_{ir}$ is \textbf{strongly compatible} if any monomial in the elements of that subset is in $\mathfrak{B}$. Each subcluster is strongly compatible, and we have the following generalization of Proposition \ref{prop:entries 1 form a subcluster}, with essentially the same proof. 

\begin{prop}
If $p\in \mathcal{A}_\mathfrak{B}(\mathbb{R}_{\geq1})$, then the set $\{b\in \mathfrak{B}_{ir} \mid f_b(p)=1\}$ is strongly compatible.
\end{prop}

We conjecture that most of our results on finite type superunitary regions extend to $\mathfrak{B}$-superunitary regions, subject to the following generalizations: 
\begin{align*} 
\text{cluster variable} &\mapsto \text{indecomposable real basis element} 
\\
\text{subcluster} &\mapsto \text{strongly compatible set of indecomposable real basis elements} 
\end{align*}
Note that, for many good bases, these generalizations are equalities.

The conjectured generalization of our results may be summarized as follows.

\begin{conj}
Let $\mathcal{A}$ be a cluster algebra with good basis $\mathfrak{B}$. 
Then the $\mathfrak{B}$-superunitary region $\mathcal{A}_\mathfrak{B}(\mathbb{R}_{\geq1})$ is a manifold with boundary, and the connected components of the boundary stratification are
\[ \mathcal{A}_\mathfrak{B}(\mathbb{R}_{\geq1})_\subcluster
:= \{ p \in \mathcal{A}_\mathfrak{B}(\mathbb{R}_{\geq1}) \mid \text{for all $b\in \mathfrak{B}_{ir}$, $f_b(p)=1$ if and only if $b\in \subcluster$}\} 
\]
where $\subcluster$ runs over all strongly compatible subsets of $\mathfrak{B}_{ir}$. Each $\mathcal{A}_\mathfrak{B}(\mathbb{R}_{\geq1})_\subcluster$ is homeomorphic to an open ball of dimension $\rank(\mathcal{A})-|\subcluster|$, and 
\[ \overline{\mathcal{A}_\mathfrak{B}(\mathbb{R}_{\geq1})_\subcluster}
= \bigsqcup_{\subcluster'\supseteq \subcluster} \mathcal{A}_\mathfrak{B}(\mathbb{R}_{\geq1})_{\subcluster'}
\]
\end{conj}

\begin{rem}
One result that does not generalize is the regular CW structure, because the closure of a face $\mathcal{A}_\mathfrak{B}(\mathbb{R}_{\geq1})_\subcluster$ is generally not homeomorphic to a closed ball. In Example \ref{ex: (2,2)}, $\mathcal{A}_\mathfrak{B}(\mathbb{R}_{\geq1})_\varnothing$ is the interior of the superunitary region and is homeomorphic to an open ball, but its closure 
is not compact and therefore cannot be homeomorphic to a closed ball.
\end{rem}

\section{Relation to Teichm\"uller spaces}

To close, we briefly sketch the interpretation of superunitary regions in the case of cluster algebras of marked surfaces. While we include an incomplete review of the relevant background, interested readers should consult \cite{FST08,GSV05, FG06} for the details. 

Let $\Sigma$ be a \textbf{unpunctured marked surface}; that is, an oriented surface with boundary together with a finite set $M$ of \textbf{marked points} in $\partial \Sigma$. We furthermore assume that $\Sigma$ admits a triangulation consisting of arcs between marked points. This assumption ensures that $\Sigma$  determines a cluster algebra $\mathcal{A}$ in which the cluster variables are parametrized by arcs between marked points, and the clusters are parametrized by triangulations of $\Sigma$.

Penner's Theorem \cite{Pen87} identifies the totally positive region of $\mathcal{A}$ with the \textbf{decorated Teichm\"uller space} $\widehat{\mathcal{T}}(\Sigma)$ of $\Sigma$. This is the moduli space of finite volume hyperbolic structures on $\Sigma \smallsetminus M$ which are geodesic on $\partial \Sigma$ (up to conformal equivalence), together with a choice of horocycle around each marked point. If $x\in \mathcal{A}$ is a cluster variable, then the function $f_x$ on $\mathcal{A}(\mathbb{R}_{>0})$ may be identified with the \textbf{lambda length} of the corresponding arc in $\Sigma$; that is, the quantity $ e^{\ell/2}$ where $\ell$ is the geodesic distance between decorating horocycles.

Under this identification $\mathcal{A}(\mathbb{R}_{>0})\simeq \widehat{\mathcal{T}}(\Sigma)$, the superunitary region $\mathcal{A}(\mathbb{R}_{\geq1})$ corresponds to hyperbolic metrics in which each lambda length is at least 1; equivalently, in which the geodesic distance between decorating horocycles is non-negative. This is equivalent to requiring that the decorating horocycles have disjoint interior. 

There is an ambiguity to address, since $\mathcal{A}$ has frozen variables represented by geodesics along the boundary (which we otherwise avoided considering). If one is using a definition of superunitary region that forces frozen variables to be 1 (see Remark \ref{rem: nofrozens}), then these geodesics must have lambda length $1$; that is, the decorating horocycles must be tangent.

We summarize these observations as follows.

\begin{prop}
Let $\Sigma$ be an unpunctured marked surface, 
and let $\mathcal{A}$ be the cluster algebra of $\Sigma$. Under $\mathcal{A}(\mathbb{R}_{>0})\simeq \widehat{\mathcal{T}}(\Sigma)$, the superunitary region $\mathcal{A}(\mathbb{R}_{\geq1})$ of $\mathcal{A}$ parametrizes the decorated hyperbolic metrics in which the decorating horocycles have disjoint interiors. If the frozen variables must be $1$, then adjacent horocycles must be tangent.
\end{prop}

\begin{figure}[h!t]
    \begin{tikzpicture}[scale=3]
        \draw (0,0) circle (1);
        \node[dot] (m1) at \MarkedPoint{2}{0} {};
        \node[dot] (m2) at \MarkedPoint{2}{1} {};
        \node[dot] (m3) at \MarkedPoint{1}{2} {};
        \node[dot] (m4) at \MarkedPoint{(-1)}{1} {};
        \node[dot] (m5) at \MarkedPoint{(-2)}{1} {};
        \node[right] at (m1) {$1$};
        \node[above right] at (m2) {$2$};
        \node[above left] at (m3) {$3$};
        \node[below] at (m4) {$4$};
        \node[below right] at (m5) {$5$};
        \clip (0,0) circle (1);
        
        \draw[fill=black!10, even odd rule] 
        (0,0) circle (1)
        \Geodesic{2}{0}{2}{1}
        \Geodesic{2}{1}{1}{2}
        \Geodesic{1}{2}{(-1)}{1}
        \Geodesic{(-1)}{1}{(-2)}{1}
        \Geodesic{(-2)}{1}{2}{0};
        
        \draw[blue] \Horocycle{2}{0};
        \draw[blue] \Horocycle{2}{1};
        \draw[blue] \Horocycle{1}{2};
        \draw[blue] \Horocycle{(-1)}{1};
        \draw[blue] \Horocycle{(-2)}{1};
    \end{tikzpicture}
    \hspace{1cm}
    \begin{tikzpicture}[scale=3,even odd rule]
        \draw (0,0) circle (1);
        \node[dot] (m1) at \MarkedPoint{2}{0} {};
        \node[dot] (m2) at \MarkedPoint{2}{1} {};
        \node[dot] (m3) at \MarkedPoint{1}{2} {};
        \node[dot] (m4) at \MarkedPoint{(-1)}{1} {};
        \node[dot] (m5) at \MarkedPoint{(-2)}{1} {};
        \node[right] at (m1) {$1$};
        \node[above right] at (m2) {$2$};
        \node[above left] at (m3) {$3$};
        \node[below] at (m4) {$4$};
        \node[below right] at (m5) {$5$};
        
        \clip (0,0) circle (1);
        
        \draw[fill=black!10, even odd rule] 
        (0,0) circle (1)
        \Geodesic{2}{0}{2}{1}
        \Geodesic{2}{1}{1}{2}
        \Geodesic{1}{2}{(-1)}{1}
        \Geodesic{(-1)}{1}{(-2)}{1}
        \Geodesic{(-2)}{1}{2}{0};
        
        \draw[blue] \Horocycle{2}{0};
        \draw[blue] \Horocycle{2}{1};
        \draw[blue] \Horocycle{1}{2};
        \draw[blue] \Horocycle{(-1)}{1};
        \draw[blue] \Horocycle{(-2)}{1};
        
        \draw[dashed] \Geodesic{2}{1}{(-1)}{1};
        \path[clip] (0,0) circle (1) \Horocycle{2}{1} \Horocycle{(-1)}{1};
        \draw[ultra thick,red] \Geodesic{2}{1}{(-1)}{1};
    \end{tikzpicture}
    \caption{A decorated hyperbolic pentagon (left, decorations in \textcolor{blue}{blue}) and the geodesic segment between decorating horocycles 2 and 4 (right, in \textcolor{red}{red})}
    \label{fig: hyperbolicpentagon}
\end{figure}
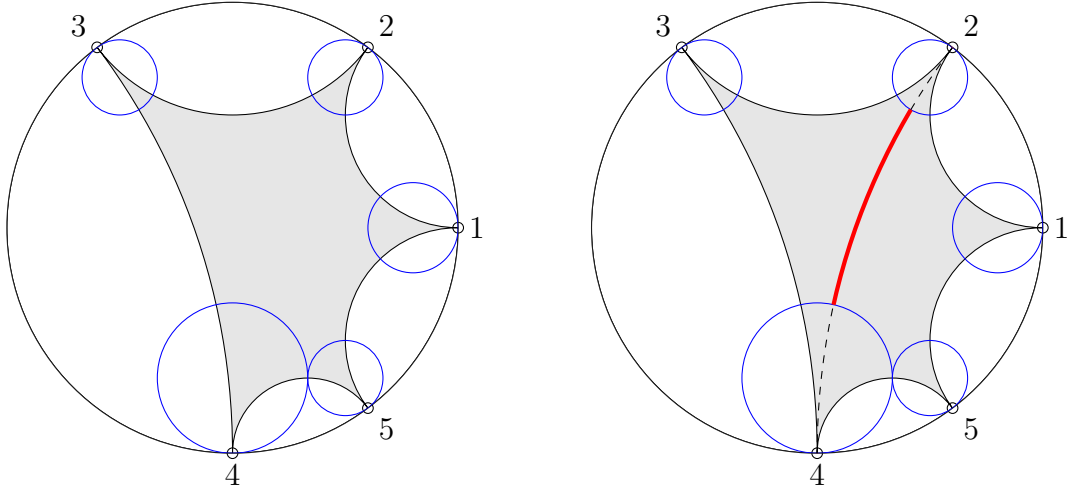

\begin{ex}
Let $\Sigma$ be the disc with $5$ marked points around the boundary, indexed counterclockwise by $1,2,3,4,5$. By the Riemann Mapping Theorem, a hyperbolic metric on $\Sigma\smallsetminus M$ is conformally equivalent to an an ideal pentagon in the hyperbolic disc, and this is unique up to M\"obius transformation. Note that such a pentagon is freely determined by its vertices (indexed by $1,2,3,4,5$). 
A decorating horocycle at a marked point $m$ becomes a circle in the hyperbolic disc tangent to the boundary at $m$ (see Figure \ref{fig: hyperbolicpentagon}). Since these five circles determine the ideal pentagon, $\widehat{\mathcal{T}}(\Sigma)$ is equivalent to the moduli space of five circles inscribed in the hyperbolic disc (up to M\"obius transformation).

There is a unique hyperbolic geodesic connecting any pair of distinct marked points, given by the circular chord perpendicular to the boundary at those points. If $\ell$ is the length of the segment of this geodesic between the decorating horocycles (with a negative sign if the horocycles overlap), then the lambda length of the geodesic is $e^{\ell/2}$. As a function on  $\mathcal{A}(\mathbb{R}_{>0})\simeq \widehat{\mathcal{T}}(\Sigma)$, the lambda length equals $f_x$ for some cluster variable $x$.

Since $e^{\ell/2}\geq1$ iff $\ell\geq0$, the superunitary region corresponds to the moduli space of five circles inscribed in the hyperbolic disc with disjoint interior (up to M\"obius transformation). 
Furthermore, two circles are tangent if and only if the lambda length of the geodesic connecting them is $1$. 
Therefore, such a collection of five circles is in the subcluster face indexed by the collection of arcs dual to the tangencies between circles.

If this collection of arcs is a triangulation, then the subcluster face consists of a single unitary point. Therefore, there is a unique-up-to-M\"obius-transformation collection of five circles inscribed in the disc whose tangencies are dual to the triangulation (Figure \ref{fig: circlepacking}); this is a special case of the Koebe--Andreev--Thurston circle packing theorem (see \cite{Ste03}).
\end{ex}

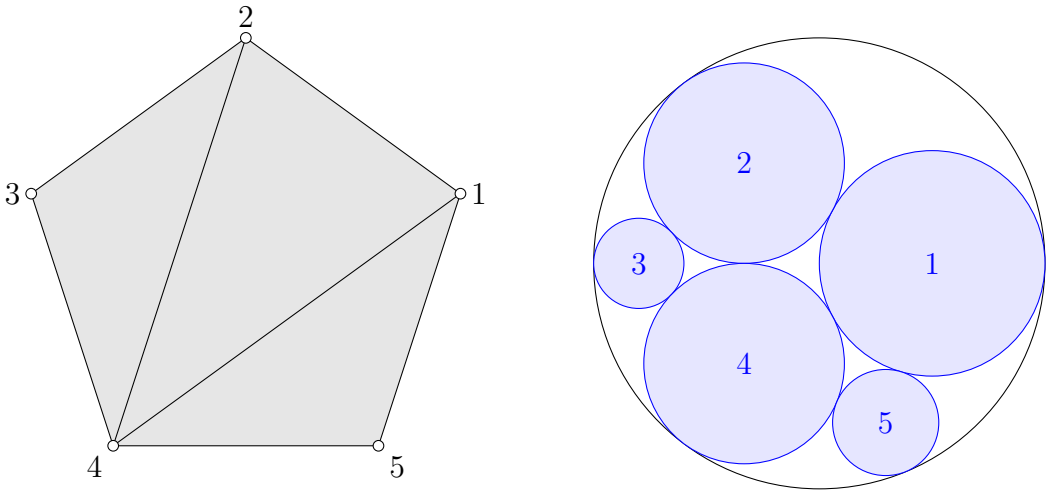
\begin{figure}[h!t]
    \begin{tikzpicture}[scale=3,baseline={(0,0)}]
        \draw[fill=black!10] (18+0*72:1) to (18+1*72:1) to (18+2*72:1) to (18+3*72:1) to (18+4*72:1) to (18+5*72:1);
        \draw (18+3*72:1) to (18+0*72:1);
        \draw (18+3*72:1) to (18+1*72:1);
        
        \node[dot,fill=white] (1) at (18+0*72:1) {};
        \node[dot,fill=white] (2) at (18+1*72:1) {};
        \node[dot,fill=white] (3) at (18+2*72:1) {};
        \node[dot,fill=white] (4) at (18+3*72:1) {};
        \node[dot,fill=white] (5) at ({18+(4*72)}:1) {};
        \node[right] at (1) {$1$};
        \node[above] at (2) {$2$};
        \node[left] at (3) {$3$};
        \node[below left] at (4) {$4$};
        \node[below right] at (5) {$5$};
    \end{tikzpicture}
    \hspace{1cm}
    \begin{tikzpicture}[scale=3,even odd rule,baseline={(0,0)}]
        \draw (0,0) circle (1);
        \draw[blue,fill=blue!10] \Horocycle{1}{0};
        \draw[blue,fill=blue!10] \Horocycle{.5}{1};
        \draw[blue,fill=blue!10] \Horocycle{0}{2};
        \draw[blue,fill=blue!10] \Horocycle{(-.5)}{1};
        \draw[blue,fill=blue!10] \Horocycle{(-1.5)}{1};
        
        \node[blue] at \HorocycleCenter{1}{0} {$1$};
        \node[blue] at \HorocycleCenter{.5}{1} {$2$};
        \node[blue] at \HorocycleCenter{0}{2} {$3$};
        \node[blue] at \HorocycleCenter{(-.5)}{1} {$4$};
        \node[blue] at \HorocycleCenter{(-1.5)}{1} {$5$};
    \end{tikzpicture}
    \caption{A triangulation (left) and the corresponding circle packing (right)}
    \label{fig: circlepacking}
\end{figure}

\addtocontents{toc}{\SkipTocEntry}
\appendix

\section{Counting faces and friezes}
\label{section: counting}

This section contains the necessary computations for Theorem \ref{thm: elementaryfriezes}. The key tool is a formula for the number of faces of each Dynkin type in a generalized associahedron. 

\addtocontents{toc}{\SkipTocEntry}
\subsection{Sorting faces by Dynkin type}

Given a generalized associahedron of type $\Gamma$, let us say a face has \textbf{type} $\Gamma'$ if the deletion of the associated subcluster is a cluster algebra with Dynkin diagram $\Gamma'$. Note that the Dynkin diagram of a face may be disconnected even if $\Gamma$ is connected, so we must take care to consider disconnected diagrams.
Define the \textbf{multiplicity} of $\Gamma'$ in $\Gamma$ to be
\[ \mu(\Gamma',\Gamma) := \text{(\# of faces in the generalized associahedron of $\Gamma$ of type $\Gamma'$}) \]

\begin{ex}
A $1$-dimensional face corresponds to the complement of one element in a cluster. The deletion of such a subcluster has Dynkin diagram $A_1$. 
Therefore, 
\[ \mu(A_1,\Gamma) = \text{(\# of edges in the generalized associahedron)} \qedhere \]
\end{ex}

\begin{ex}
A $0$-dimensional face corresponds to a cluster, whose deletion is a cluster algebra with empty Dynkin diagram $\varnothing$. Therefore,
\[ \mu(\varnothing,\Gamma) = \text{(\# of clusters in cluster algebra of $\Gamma$)}
= \text{(\# of unitary friezes of type $\Gamma$)}
\]
These counts are listed in Table \ref{table: friezes}. 
\end{ex}

Multiplicity is useful in counting positive integral friezes, because the number of frieze points in each subcluster face is equal to the number of $\mathbb{Z}_{\geq2}$-valued frieze points in the deletion. Because this only depends on the Dynkin type of that face, we can simplify Equation \eqref{eq: sumoveratleast2} by combining faces of the same type.
\[
\text{(\# of $\mathbb{Z}_{\geq1}$-valued friezes of type $\Gamma$)}
= \sum_{\Gamma'} \mu(\Gamma',\Gamma)\text{(\# of $\mathbb{Z}_{\geq2}$-valued friezes of type $\Gamma'$)}
\]
In the next section, we derive a formula for $\mu(\Gamma', \Gamma)$ which will allow us to relate the counts of the two types of friezes.

\addtocontents{toc}{\SkipTocEntry}
\subsection{Counting facets}

A general formula for multiplicity will follow from a formula for the number of facets of each type. 
If $h(\Gamma)$ is the Coxeter number of a connected Dynkin diagram $\Gamma$, define the \textbf{order} of $\Gamma$ to be $\ord(\Gamma) := h(\Gamma)/2+1$ (see Table \ref{table: orders}). Note that this is not a standard term, and it is an integer unless $\Gamma=A_{2k}$.

\begin{table}[h!t]
\[\begin{array}{|c|c|c|c|c|c|c|c|c|}
\hline
\Gamma & A_n & B_n\text{ or }C_n & D_n & E_6 & E_7 & E_8 & F_4 & G_2 \\
\hline
h(\Gamma) & n+1 & 2n & 2n-2 & 12 & 18 & 30 & 12 & 6 \\
\hline
\ord(\Gamma) & \frac{n+3}{2} & n+1 & n & 7 & 10 & 16 & 7 & 4 \\
\hline
\end{array}\]
\caption{Coxeter numbers and orders of Dynkin diagrams}
\label{table: orders}
\end{table}

\begin{prop}
Let $\Gamma$ be a connected Dynkin diagram.
The number of facets of $\Gamma$ of type $\Gamma'$ is equal to $\ord(\Gamma)$ times the number of vertices in $\Gamma$ whose complement is $\Gamma'$.
\end{prop}

\noindent Equivalently, when $\Gamma$ is connected and $\Gamma'$ has one fewer vertex, $ \mu(\Gamma',\Gamma)$ is equal to $\ord(\Gamma)$ times the number of vertices in $\Gamma$ whose complement is $\Gamma'$.

\begin{proof}
Let $v$ be a vertex in $\Gamma$, and consider the cluster variables in Row $v$ of the general frieze of type $\Gamma$. Given such a cluster variable $x$, there is a seed whose quiver is an orientation of $\Gamma$ with $x$ on vertex $v$. Therefore, the deletion of $\{x\}$ is a cluster algebra of type $\Gamma\smallsetminus \{v\}$.

If $\Gamma$ is not type $A$  or if $v$ is the middle vertex of $\Gamma=A_{2n+1}$, then the cluster variables in the $v$th row of the general frieze are $\ord(\Gamma)$-periodic and do not appear in any other rows. 
Therefore, the vertex $v$ accounts for $\ord(\Gamma)$-many facets of type $\Gamma\smallsetminus\{v\}$.

If $\Gamma$ is type $A$ and $v$ is not a middle vertex, then the cluster variables in the $v$th row are $2\ord(\Gamma)$-periodic. However, the same sequence of cluster variables appears in Row $v^{op}$, where $v^{op}$ is the image of $v$ under the reflection of $\Gamma$. Therefore, the pair of vertices $(v,v^{op})$ accounts for $2\ord(\Gamma)$-many facets of type $\Gamma\smallsetminus\{v\}$.
\end{proof}

\begin{ex}\label{ex: multiplicityA12}
The Dynkin diagram $A_2$ has order $5/2$ and two vertices with complement $A_1$. Therefore,
\[ \mu(A_1,A_2) = (5/2) \cdot 2 = 5  \]
So, the $A_2$ generalized associahedron has $5$ facets of type $A_1$.
\end{ex}

\begin{ex}
The Dynkin diagram $A_3$ has order $3$, two vertices with complement $A_2$, and one vertex with complement $A_1\times A_1$. Therefore,
\[ \mu(A_2,A_3) = 3\cdot 2 = 6
\text{ and }
\mu(A_1\times A_1,A_3) = 3
\]
So, the $A_3$ generalized associahedron has $6$ facets of type $A_2$ and $3$ facets of type $A_1\times A_1$.
\end{ex}

The proposition can be extended to disconnected Dynkin diagrams with more careful notation. Given a vertex $v$ in a Dynkin diagram $\Gamma$, let
\[ \ord(v,\Gamma) := \text{order of the connected component of $\Gamma$ containing $v$} \]
Then each vertex $v$ of $\Gamma$ contributes $\ord(v,\Gamma)$-many facets of type $\Gamma\smallsetminus \{v\}$. Equivalently, whenever $\rank(\Gamma)-\rank(\Gamma')=1$, 
\[ \mu(\Gamma',\Gamma) = \sum \ord(v,\Gamma) 
\]
where the sum runs over $v\in \Gamma$ such that $\Gamma\smallsetminus \{v\}=\Gamma'$.

\addtocontents{toc}{\SkipTocEntry}
\subsection{Counting faces}
\label{section: countingfaces}

Every subcluster $\cluster$ contains $|\cluster|$-many cluster variables, so each face of the generalized associahedron of codimension $k$ is contained in $k$-many facets. 
Therefore, every face of codimension $k$ in a generalized associahedron is a face in $k$-many facets. This translates into the following recursive formula for face multiplicies.

\begin{thm}
\label{thm: A6}
Let $\Gamma',\Gamma$ be Dynkin diagrams with $k:=\mathrm{rank}(\Gamma)-\mathrm{rank}(\Gamma')>0$, then 
\begin{equation}
\label{eq: facetsum}
\mu(\Gamma',\Gamma)
= \frac{1}{k} \sum_{v\in \Gamma} \ord(v,\Gamma) \mu(\Gamma',\Gamma\smallsetminus \{v\}) 
\end{equation}
\end{thm}
Note that, if $\Gamma$ is connected, the order is independent of $v$ and this formula becomes
\[ \mu(\Gamma',\Gamma)
= \frac{\ord(\Gamma)}{\mathrm{rank}(\Gamma)-\mathrm{rank}(\Gamma')} \sum_{v\in \Gamma}  \mu(\Gamma',\Gamma\smallsetminus \{v\}) 
\]

\begin{ex}
We may count $A_1$ faces in the $A_3$ generalized associahedron with
\begin{align*}
\mu(A_1,A_3) 
&= \frac{\ord(A_3)}{\mathrm{rank}(A_3)-\mathrm{rank}(A_1)}
\left(\mu(A_1,A_2) + \mu(A_1,A_1\times A_1) + \mu(A_1,A_2)  \right) \\
&= \frac{3}{2}
\left(2\mu(A_1,A_2) + \mu(A_1,A_1\times A_1)  \right)
\end{align*}
Example \ref{ex: multiplicityA12} showed $\mu(A_1,A_2)=5$. To compute $\mu(A_1,A_1\times A_1)$, observe that both vertices in $A_1\times A_1$ have complement $A_1$ and order $2$, so 
$\mu(A_1,A_1\times A_1) = 2 + 2 = 4$. 
Then
\[\mu(A_1,A_3) = \frac{3}{2}(2\cdot 5 +4 ) =21 \]
That is, the $A_3$ generalized associahedron has $21$ faces of type $A_1$.
\end{ex}

In general, we can reduce to the case of connected $\Gamma$ with the following. Let
\[ N(\Gamma) := \text{(\# of clusters in the cluster algebra of $\Gamma$)} \] 
Since this coincides with the number of unitary friezes, the values of $N$ are in Table \ref{table: friezes}.

\begin{prop}
If $\Gamma',\Gamma_1,$ and $\Gamma_2$ are connected Dynkin diagrams, then
\[ \mu(\Gamma',\Gamma_1\times \Gamma_2) = N(\Gamma_2)\mu(\Gamma',\Gamma_1) + N(\Gamma_1)\mu(\Gamma',\Gamma_2) \]
\end{prop}
\noindent In particular, if $\mu(\Gamma',\Gamma_2)=0$, then 
$ \mu(\Gamma',\Gamma_1\times \Gamma_2) = N(\Gamma_2)\mu(\Gamma',\Gamma_1) $.

A notable special case of Theorem~\ref{thm: A6} is when $\Gamma'$ is the empty diagram, when the multiplicity becomes $\mu(\varnothing,\Gamma) = \text{\# of clusters in $\Gamma$} =: N(\Gamma) $.
\begin{prop}
Let $\Gamma$ be a Dynkin diagram. Then 
\[ N(\Gamma)
= \frac{1}{\mathrm{rank}(\Gamma)} \sum_{v\in \Gamma} \ord(v,\Gamma) N(\Gamma\smallsetminus \{v\}) 
\]
\end{prop}

\begin{ex}
The Dynkin diagram $B_3$ has order $4$, and so
\[ N(B_3) = \frac{4}{3} \left( N(B_2) + N(A_1\times A_1) + N(A_2)\right) 
= \frac{4}{3} (6+4+5) = 20 \]
which coincides with the standard formula $N(B_3)=\binom{6}{3} =20$ in Table \ref{table: friezes}.
\end{ex}

\addtocontents{toc}{\SkipTocEntry}
\subsection{Preliminary computations}

Using the formulas from the previous section, we compute several multiplicities used to count friezes in the next section.
In the computations that follow, we omit the summands of \eqref{eq: facetsum} for which $\Gamma\smallsetminus \{v\}$
does not contain $\Gamma'$ as a subdiagram, as the multiplicity in those cases is $0$. 

\allowdisplaybreaks

\begin{align*}
\mu(B_3,F_4)
&= \frac{\ord(F_4)}{1}\mu(B_3,B_3)
= \frac{7}{1}(1) = 7 \\
\mu(D_4,D_5) 
&= \frac{\ord(D_5)}{1} 
\left(
\mu(D_4,D_4)
\right) 
= \frac{5}{1}(1) 
=5\\
\mu(D_4,D_6) 
&=\frac{\ord(D_6)}{2} 
\left(
\mu(D_4,D_5)
+\mu(D_4,D_4\times A_1)
\right) 
= \frac{6}{2} (5+
2\cdot 1)  
= 21 \\
\mu(D_4,E_6) 
&=\frac{\ord(E_6)}{2} 
\left(
\mu(D_4,D_5)
+\mu(D_4,D_5)
\right) 
= \frac{7}{2} (5+5) = 35 \\
\mu(D_4,D_7) 
&=\frac{\ord(D_7)}{3} 
\left(
\mu(D_4,D_6)
+\mu(D_4,D_5\times A_1)
+\mu(D_4,D_4\times A_2)
\right) \\
&= \frac{7}{3}(21 + 
2\cdot 5 
+ 
5 \cdot 1) 
= 84 \\
\mu(D_4,E_7) 
&=\frac{\ord(E_7)}{3} 
\left(
\mu(D_4,D_6)
+\mu(D_4,E_6)
+\mu(D_4,D_5\times A_1)
\right) \\
&= \frac{10}{3} (21+35+2\cdot 5) = 220 \\
\mu(D_4,E_8) 
&=\frac{\ord(E_8)}{4} 
\left(
\mu(D_4,D_7)
+\mu(D_4,E_7)
+\mu(D_4,E_6\times A_1)
+\mu(D_4,D_5\times A_2)
\right) \\
&= \frac{16}{4}
(84 + 220 + 
2\cdot 35 
+ 5 \cdot 5) =1596 \\
\mu(D_6,E_7)
&= \frac{\ord(E_7)}{1} \left(
\mu(D_6,D_6)
\right) 
= \frac{10}{1}(1) = 10
\\
\mu(D_6,E_8) 
&= \frac{\ord(E_8)}{2} \left(
\mu(D_6,E_7) + \mu(D_6,D_7)
\right)
= \frac{16}{2}(10 + 7) = 136 
\end{align*}

\addtocontents{toc}{\SkipTocEntry}
\subsection{Counting $\mathbb{Z}_{\geq2}$-valued friezes}
\label{section: countingfriezes}

We can now use known and conjectured counts of positive integral friezes to reverse engineer counts of $\mathbb{Z}_{\geq2}$-valued friezes, and show that they are all accounted for by Theorem \ref{thm: elementaryfriezes}.
Let 
\[ e(\Gamma) := \text{(\# of $\mathbb{Z}_{\geq2}$-valued friezes of type $\Gamma$)} \]
Then Equation \eqref{eq: sumoveratleast2} can be rewritten as
\begin{equation}
\label{eq: multiplicitye}
\text{(\# of positive integral friezes of type $\Gamma$)}
= \sum_{\Gamma'}e(\Gamma')\mu(\Gamma',\Gamma)
\end{equation}
where the sum runs over Dynkin diagrams contained in $\Gamma$.

\begin{warn}
Don't forget the case when $\Gamma'$ is the \emph{empty} Dynkin diagram! Since $e(\varnothing)=1$ and $\mu(\varnothing,\Gamma)=N(\Gamma)$, this case contributes the $N(\Gamma)$-many unitary friezes to \eqref{eq: multiplicitye}.
\end{warn}

\begin{proof}[Proof of Theorem \ref{thm: elementaryfriezes}]
Every positive integral frieze of type $A$ is unitary by \cite{ConCox73}, and so none are $\mathbb{Z}_{\geq2}$. While it is stated in the language of marked surfaces, \cite[Lemma~3.5]{FP16} can be translated into an explicit construction of $\mathbb{Z}_{\geq2}$-valued friezes of type $D$, yielding the $d(n)-2$ friezes in Figure \ref{fig: Dabfrieze}. This construction and the results of \cite[Section~4]{FP16} then imply that the only $\mathbb{Z}_{\geq2}$-valued friezes of types $BCG$ are those in Figures \ref{fig: Bn2frieze} and \ref{fig: G2frieze}. 

Counting the examples from the previous paragraph, if $\Gamma$ is type $ABCDG$, then
\[
e(\Gamma)
:= \left\{
\begin{array}{cc}
d(n)-2 & \text{if $\Gamma=D_n$} \\
1 & \text{if $\Gamma=B_{n^2-1}$ or $G_2$} \\
0 & \text{otherwise}
\end{array}
\right\}
\]
where $d(n)$ denotes the number of divisors of $n$.
Therefore, we can discard any summands from Equation \eqref{eq: multiplicitye} in which $\Gamma'$ are not unions of $D_n$, $B_{n^2+1}$, and the unknown types $EF$.

The number of positive integral friezes of type $F_4$ is known to be $112$. 
The only subdiagrams of $F_4$ which could admit $\mathbb{Z}_{\geq2}$-valued friezes are $\varnothing$, $B_3$, and $F_4$. Then
\begin{align*}
112 
&= e(\varnothing)\mu(\varnothing,F_4) + e(B_3) \mu(B_3,F_4) + e(F_4)\mu(F_4,F_4) \\
&= 1\cdot 105 + 1\cdot 7 + e(F_4)\cdot 1 = 112 + e(F_4)
\end{align*}
Therefore, $e(F_4)=0$; that is, there are no $\mathbb{Z}_{\geq2}$-valued friezes of type $F_4$.

The number of positive integral friezes of type $E_6$ is known to be $868$. 
The only subdiagrams of $E_6$ which could admit $\mathbb{Z}_{\geq2}$-valued friezes are $\varnothing$, $D_4$, and $E_6$. Then
\begin{align*}
868
&= e(\varnothing)\mu(\varnothing,E_6) + e(D_4) \mu(D_4,E_6) + e(E_6)\mu(E_6,E_6) \\
&= 1\cdot 833 + 1\cdot 35 + e(E_6)\cdot 1 = 868 + e(E_6)
\end{align*}
Therefore, $e(E_6)=0$; that is, there are no $\mathbb{Z}_{\geq2}$-valued friezes of type $E_6$.

The number of positive integral friezes of type $E_7$ is conjectured to be $4400$.
The only subdiagrams of $E_7$ which could admit $\mathbb{Z}_{\geq2}$-valued friezes are $\varnothing$, $D_4$, $D_6$, and $E_7$. Assuming the conjectured count is correct,
\begin{align*}
4400
&= e(\varnothing)\mu(\varnothing,E_7) + e(D_4) \mu(D_4,E_7) + e(D_6) \mu(D_6,E_7) + e(E_7)\mu(E_7,E_7) \\
&= 1\cdot 4160 + 1\cdot 220 + 2\cdot 10 + e(E_7)\cdot 1 = 4400 + e(E_7)
\end{align*}
Therefore, $e(E_7)=0$; that is, there are no $\mathbb{Z}_{\geq2}$-valued friezes of type $E_7$.

The number of positive integral friezes of type $E_8$ is conjectured to be $26952$.
The only subdiagrams of $E_8$ which could admit $\mathbb{Z}_{\geq2}$-valued friezes are $\varnothing$, $D_4$, $D_6$, and $E_8$. Assuming the conjectured counts for $E_7$ and $E_8$ are correct,
\begin{align*}
26952
&= e(\varnothing)\mu(\varnothing,E_8) + e(D_4) \mu(D_4,E_8) + e(D_6) \mu(D_6,E_8) + e(E_8)\mu(E_8,E_8) \\
&= 1\cdot 25080 + 1\cdot 1596 + 2\cdot 136 + e(E_8)\cdot 1 = 26948 + e(E_8)
\end{align*}
Therefore, $e(E_8)=4$; that is, there are four $\mathbb{Z}_{\geq2}$-valued friezes of type $E_8$. The four translations of Figure \ref{fig: E8frieze} are $\mathbb{Z}_{\geq2}$-valued friezes of type $E_8$; therefore, they are the only ones.
\end{proof}

\begin{rem}
Besides its existence and apparent uniqueness-up-to-translation, we know of no significant meaning behind the $\mathbb{Z}_{\geq2}$-valued frieze appearing in Figure \ref{fig: E8frieze}. 
\end{rem}

\subsection*{Acknowledgements}
This project has been simmering for more than a decade and has been productively discussed with a great number of people. An earlier version of this project was supported by MSRI's 2012 thematic semester on cluster algebras, including productive conversations with Fr\'ed\'eric Chapoton, Joel Geiger, Karl Marlburg, and Jacob Matherne. The more recent incarnation and its connection with friezes benefited from the input and feedback of Anna Felikson, 
Ana Garcia Elsener, 
Atabey Kaygun, 
Lang Mou, 
Pierre-Guy Plamondon,
Ralf Schiffler, and
Pavel Tumarkin.

The authors would like to thank the Isaac Newton Institute for Mathematical Sciences for support and hospitality during the programme Cluster Algebras and Representation Theory when work on this paper was undertaken. This work was supported by: EPSRC Grant Number EP/R014604/1.


\newpage

\renewcommand\contentsname{Table of Contents}
\tableofcontents

\vspace{-1cm}

\listoffigures

\end{document}